\documentclass[11pt,reqno]{amsart}
\usepackage{amssymb,mathrsfs,color,mathtools,empheq, verbatim, epstopdf}
\usepackage{pinlabel}
\mathtoolsset{showonlyrefs}
\usepackage{hyperref} 

\usepackage{enumerate}
\usepackage{tensor}

\usepackage{graphicx}
\usepackage{xcolor} 
\usepackage{tensor}
\usepackage{slashed}
\usepackage{cite}

\usepackage[shortlabels]{enumitem}

\usepackage{geometry}

\numberwithin{equation}{section}

\newtheorem{mainthm}{Theorem}
\newtheorem{thm}{Theorem}[section]
\newtheorem{cor}[thm]{Corollary}
\newtheorem{lem}[thm]{Lemma}
\newtheorem{prop}[thm]{Proposition}

\theoremstyle{definition} 
\newtheorem{rem}[thm]{Remark}
\newtheorem{defn}[thm]{Definition}

\theoremstyle{remark}

\def\bN {\mathbb{N}}

\def\bR {\mathbb{R}}
\def\bS {\mathbb{S}}

\def\bZ {\mathbb{Z}}

\def\fv {\mathfrak{v}}

\def\cA {\mathcal{A}}

\def\cE {\mathcal{E}}

\def\scrL{\mathscr{L}}

\def\grad {{\nabla}}

\newcommand{\la}{\langle}
\newcommand{\ra}{\rangle}
\newcommand{\La}{\big\langle}
\newcommand{\Ra}{\big\rangle}

\newcommand{\tx}[1]{\mathrm{#1}}

\newcommand{\wt}[1]{\widetilde{#1}}
\newcommand{\bs}[1]{\boldsymbol{#1}}

\newcommand{\eee}{\mathrm e}

\newcommand{\vd}{\mathrm{d}}
\newcommand{\udr}{\,r\vd r}
\newcommand{\vD}{\mathrm{D}}

\newcommand{\uln}[1]{{\underline{ #1 }}}
\newcommand{\lin}{_{\textsc{l}}}
\newcommand{\nl}{\textsc{nl}}

\newcommand{\Lin}{\textsc{l}} 

\newcommand{\conj}[1]{\overline{#1}}

\newcommand{\dd}[1]{\frac{\ud}{\ud{#1}}}

\definecolor{deepgreen}{cmyk}{1,0,1,0.5}

\newcommand{\E}{\mathcal{E}}

\newcommand{\LL}{\mathcal{L}}

\newcommand{\N}{\mathbb{N}}
\newcommand{\R}{\mathbb{R}}

\newcommand{\Z}{\mathbb{Z}}

\newcommand{\al}{\alpha}
\newcommand{\be}{\beta}

\newcommand{\de}{\delta}

\newcommand{\om}{\omega}
\newcommand{\lam}{\lambda}
\newcommand{\te}{\theta}

\newcommand{\s}{\sigma}

\newcommand{\si}{\varsigma}

\newcommand{\De}{\Delta}
\newcommand{\Om}{\Omega}

\newcommand{\Lam}{\Lambda}

\newcommand{\p}{\partial}

\newcommand{\loc}{\operatorname{loc}}

\makeatletter

\newcommand{\Rmnum}[1]{\expandafter\@slowromancap\romannumeral #1@}
\makeatother

\newcommand{\lec}{\lesssim}

\newcommand{\I}{\infty}

\newcommand{\ti}{\widetilde}

\newcommand{\U}{\underline}

\newcommand{\ang}[1]{\left\langle{#1}\right\rangle}
\newcommand{\abs}[1]{\left\lvert{#1}\right\rvert}

\newcommand{\EQ}[1]{\begin{equation}\begin{split} #1 \end{split}\end{equation}}

\newcommand{\Del}[1]{}

\newcommand{\pt}{&}
\newcommand{\pr}{\\ &}

\newcommand{\pn}{}

\newcommand{\mand}{{\ \ \text{and} \ \  }}

\newcommand{\mif}{{\ \ \text{if} \ \ }}
\newcommand{\mfor}{{\ \ \text{for} \ \ }}
\newcommand{\mas}{{\ \ \text{as} \ \ }}

\newcommand{\uD}{\operatorname{D}}

\def\glei{\mathrm{eq}}

\newcommand{\rdr}{\, r\mathrm{d}r}

\newcommand{\rest}{\!\!\restriction}

\definecolor{green}{rgb}{0,0.8,0} 

\newcommand{\ud}{\mathrm{d}}

\newcommand{\eps}{\epsilon}

\newcommand{\bfd}{{\bf d}}

\newcommand{\bfi}{{\bf i}}

\newcommand{\bfp}{{\bf p}}
\newcommand{\bfq}{{\bf q}}

\newcommand{\calA}{\mathcal A}

\newcommand{\calC}{\mathcal C}

\newcommand{\calE}{\mathcal E}

\newcommand{\calJ}{\mathcal J}
\newcommand{\calK}{\mathcal K}
\newcommand{\calL}{\mathcal L}

\newcommand{\calQ}{\mathcal Q}

\newcommand{\calS}{\mathcal S}

\newcommand{\calZ}{\mathcal Z}

\newcommand{\ULam}{\U{\Lam}}

\newcommand{\oo}{\textrm{o}}
\newcommand{\ii}{\textrm{i}}

\begin{document}

\title[Soliton resolution for wave maps]{Soliton resolution for energy-critical wave maps \\
in the equivariant case}
\author{Jacek Jendrej}
\author{Andrew Lawrie}

\begin{abstract}
We consider the equivariant wave maps equation $\bR^{1+2} \to \bS^2$,  in all equivariance classes $k  \in \N$. 
We prove that every finite energy solution resolves, continuously in time, into a superposition of asymptotically decoupling harmonic maps and free radiation.  
\end{abstract}


\thanks{J.Jendrej is supported by  ANR-18-CE40-0028 project ESSED.  A. Lawrie is supported by NSF grant DMS-1954455, a Sloan Research Fellowship, and the Solomon Buchsbaum Research Fund}

\maketitle

\tableofcontents

\section{Introduction} 
\subsection{Setting of the problem}

We study wave maps  from  the Minkowski space $\mathbb{R}^{1+2}_{t, x}$ into the two-sphere $\mathbb{S}^2$, under $k$-equivariant symmetry. These are formal critical points of the Lagrangian action, 
\begin{equation} 
\scrL( \Psi)  = \frac{1}{2} \iint_{\bR^{1+2}_{t, x}} \big( {-} |\partial_t \Psi(t, x)|^2 + |\grad \Psi(t, x)|^2 \big) \, \ud x  \ud t, 
\end{equation}
restricted to the class of maps  $\Psi: \bR^{1+2}_{t, x} \to \bS^2 \subset \bR^3$ that take the form, 
\begin{equation}
\Psi(t, r\eee^{i\theta}) = ( \sin u (t, r)\cos k\theta , \sin u(t, r) \sin  k \theta, \cos u(t, r)) \in \bS^2 \subset \bR^3,
\end{equation}
for some fixed $k \in \{1, 2, \ldots\}$. Above $u$ is the colatitude measured from the north pole,  the metric on $\bS^2$ is  $\vd s^2 = \vd u^2+ \sin^2 u\,  \vd \omega^2$,  and $(r, \theta)$ are polar coordinates on $\bR^2$.

The general $\mathbb{S}^2$-valued wave maps equation in two space dimensions is called the $O(3)$ sigma model in high energy physics literature. It is a canonical example of a geometric wave equation as it generalizes the free scalar wave equation to the setting of manifold-valued maps. The static solutions given by finite energy harmonic maps are amongst the simplest examples of topological solitons as they admit Bogomol'nyi structure~\cite{Bog};  other examples include kinks in scalar field theories on the line,  vortices in Ginzburg-Landau equations, magnetic monopoles, Skyrmions, and Yang-Mills instantons; see~\cite{MS} for an extensive treatment of field theories admitting topological solitons from the point of view of mathematical physics.

Our interest in $k$-equivariant wave maps stems from the richness of their nonlinear dynamics in the relatively simple setting of the geometrically natural scalar semilinear wave equation, 
\EQ{ \label{eq:wmk} 
\p_t^2 u(t, r) - \De u(t, r) + \frac{k^2}{r^2} \frac{\sin 2 u(t, r)}{2} &= 0, \quad (t, r) \in \R \times (0, \infty), 
}
which is the Euler-Lagrange equation associated to $\scrL(\Psi)$ under the $k$-equivariant symmetry reduction. Here $\De := \p_r^2 + r^{-1} \p_r$ is the radial Laplacian in $2$-dimensions.  The conserved energy for~\eqref{eq:wmk} is given by 
\EQ{ \label{eq:energy} 
E(u, \p_t u)(t) := 2\pi \int_0^\infty \frac{1}{2} \Big((\p_t u(t, r))^2 + ( \p_r u(t, r))^2 + k^2 \frac{ \sin^2 u(t, r)}{r^2} \Big) \, r \ud r.  
}
We will often write pairs of functions using boldface,  $\bs v = (v, \dot v)$, noting that the notation $\dot v$ will not, in general, refer to a time derivative of $v$ but rather just to the second component of $\bs v$. With this notation the Cauchy problem for~\eqref{eq:wmk} can be rephrased as the Hamiltonian system 
\begin{equation} \label{eq:u-ham} 
\partial_t \bs u(t) =  J \circ \vD E( \bs u(t)),\qquad \bs u(T_0) = \bs u_0,
\end{equation}
where 
\begin{equation} \label{eq:DE}
J = \begin{pmatrix} 0 &1 \\ -1 &0\end{pmatrix}, \quad \vD E( \bs u(t))  =  \begin{pmatrix}- \De u(t)  + k^2r^{-2}2^{-1}\sin(2u(t)) \\ \partial_t u(t) \end{pmatrix}.
\end{equation}
Both~\eqref{eq:u-ham}  and~\eqref{eq:energy} are invariant under the scaling
\begin{equation} \label{eq:uscale} 
(u(t, r), \partial_t u(t, r))\mapsto
\big(u(t/ \lambda,  r/ \lambda), \lambda^{-1} \partial_t u( t/ \lambda, r/ \lambda)\big), \qquad \lambda >0,
\end{equation}
and thus~\eqref{eq:wmk} is called energy-critical.

The natural setting in which to consider the Cauchy problem for~\eqref{eq:wmk} is the space of initial data $\bs u_0$ with finite energy, $E( \bs u_0) < \infty$. The set of finite energy data is split into disjoint sectors, $\E_{\ell, m}$, which for $\ell, m\in \bZ$, are defined by 
\EQ{
\calE_{\ell, m}:= \big\{ (u_0, \dot u_0) \mid E(u_0, \dot u_0) < \infty, \quad \lim_{r \to 0} u_0(r) = \ell\pi, \quad \lim_{r \to \infty} u_0(r) = m \pi \big\}.
}
These sectors, which are preserved by the flow, are related to the topological degree of the full map $\Psi_0: \bR^2 \to  \bS^2$:
 if $m - \ell$ is even and $(u_0, 0) \in \E_{\ell, m}$, then the corresponding map $\Psi$ with polar angle $u_0$ is topologically trivial,
 whereas for odd $m - \ell$ the map has degree $k$.

The sets $\calE_{\ell, m}$ are affine spaces, parallel to the linear space $\E := \E_{0, 0} = H \times L^2$, which we endow with the norm, 
\EQ{
\| \bs u_0 \|_{\E}^2:=   \| \dot u_0 \|_{L^2}^2+ \| u_0\|_H^2  := \int_0^\infty \Big((\dot u_0(r))^2 + ( \p_r u_0(r))^2 + k^2 \frac{ (u_0(r))^2}{r^2} \Big) \, r \ud r.
} 
The linearization of~\eqref{eq:wmk} about the zero solution is given by  
\EQ{ \label{eq:lin} 
\p_t^2 v - \De v + \frac{k^2}{r^2} v = 0 , 
}
and the flow for~\eqref{eq:lin} preserves the $\E$ norm. 

The unique $k$-equivariant harmonic map is given  explicitly by
\EQ{
Q(r) := 2 \arctan (r^k).
}
Here uniqueness means up to scaling,  sign change, and adding a multiple of $\pi$, i.e.,   every finite energy stationary solution to~\eqref{eq:wmk}  takes the form $Q_{\mu, \sigma, m}(r) =  m \pi + \sigma Q(r/ \mu)$ for some  $\mu \in (0, \infty), \sigma \in \{0, -1, 1\}$ and $m \in \Z$. 
The pair $\bs Q_\lambda := (Q_\lambda, 0)$ and its rescaled versions $\bs Q_\lambda(r) := (Q_\lambda(r), 0)  := Q(\lambda^{-1}r)$ for $\lambda > 0$,
are minimizers of the energy $E$  within the class $\E_{0, 1}$;  in fact, $E(\bs  Q_\lambda ) = 4  \pi k$.
 We denote, $\bs\pi := (\pi, 0)$.

\subsection{Statement of the results}  Our main result is formulated as follows. 

\begin{mainthm}[Soliton Resolution] \label{thm:main}  Let $k \in \N$, let $\ell, m \in \Z$,  and let $\bs u(t)$ be a finite energy solution to~\eqref{eq:wmk} with initial data $\bs u(0) = \bs u_0 \in \E_{\ell, m}$,  defined on its maximal forward interval of existence $[0,T_+)$. 
 
 \emph{({Global solution})} If $T_+ = \infty$, there exist a time $T_0>0$, a solution $\bs u^*\lin(t) \in \E$ to the linear wave equation~\eqref{eq:lin}, an integer $N \ge 0$, continuous functions $\lam_1(t), \dots,  \lam_N(t) \in C^0([T_0, \infty))$, signs $\iota_1, \dots, \iota_N \in \{-1, 1\}$,  and $\bs g(t) \in \E$ defined by 
 \EQ{
 \bs u(t) = m  \bs \pi + \sum_{j =1}^N \iota_j (\bs Q_{\lam_j(t)} - \bs \pi) + \bs u^*\lin(t) + \bs g(t) , 
 }
 such that 
 \EQ{
 \| \bs g(t)\|_{\E} + \sum_{j =1}^{N} \frac{\lam_{j}(t)}{\lam_{j+1}(t)}  \to 0 \mas t \to \infty,  
 }
 where above we use the convention that $\lam_{N+1}(t) = t$.

 \emph{({Blow-up solution})} If $T_+ < \infty$, there exists a time $T_0< T_+$,  an integer $m_\De$, a mapping $\bs u_0^*\in \E_{m_{\De}, m}$, an integer $N \ge 1$, continuous functions $\lam_1(t), \dots,  \lam_N(t) \in C^0([T_0, T_+))$, signs $\iota_1, \dots, \iota_N \in \{-1, 1\}$, and $\bs g(t) \in \E$ defined by 
 \EQ{
 \bs u(t) = m_\De  \bs \pi + \sum_{j =1}^N \iota_j(\bs Q_{\lam_j(t)} - \bs \pi) + \bs u^*_0 + \bs g(t) , 
 }
 such that 
 \EQ{
 \| \bs g(t)\|_{\E} + \sum_{j =1}^{N} \frac{\lam_{j}(t)}{\lam_{j+1}(t)}  \to 0 \mas t \to T_+, 
 }
 where above we use the convention that $\lam_{N+1}(t) = T_+-t$. 

Analogous statements hold for the backwards-in-time evolution. 
\end{mainthm} 


\begin{rem} This type of behavior is referred to as soliton resolution. 
A recent preprint by Duyckaerts, Kenig, Martel, and Merle proved Theorem~\ref{thm:main} in the case $k=1$ using the method of energy channels; see~\cite{DKMM}. 
Roughly, energy channels refer to measurements of the portion of energy that a linear or nonlinear wave radiates outside fattened light cones.
Such exterior energy estimates were introduced by Duyckaerts, Kenig, and Merle ~\cite{DKM4} in their proof of the soliton resolution conjecture
for the radial energy critical NLW in $3$ space dimensions; see also~\cite{DKM7, DKM8, DKM9} for the treatment of all odd dimensions. 
The approach we take to prove Theorem~\ref{thm:main} is independent of the method of energy channels. 
\end{rem}

\begin{rem}
The soliton resolution problem is inspired by the theory of completely integrable systems, e.g.,~\cite{Eck-Sch, Sch, SA},  motivated by numerical simulations,~\cite{FPU, ZK}, and by the bubbling theory of harmonic maps in the elliptic and parabolic settings~\cite{Struwe85, Qing, QT, Top-JDG, Topping04}; see also~\cite{DJKM1, DKM9, DKMM} for discussions on the history of the problem. 
\end{rem} 

\begin{rem}
Our method establishes the exact analog of Theorem~\ref{thm:main} in the case of the equivariant Yang-Mills equation, by making the usual analogy between equivariant Yang-Mills and $k=2$-equivariant wave maps; see Cazenave, Shatah, and Tahvildar-Zadeh~\cite{CST98} for the formulation.  There, the harmonic map is replaced by the first instanton. 
\end{rem}

\begin{rem} 
Theorem~\ref{thm:main} is a qualitative description of the dynamics of all finite energy solutions to~\eqref{eq:wmk}.  A natural, challenging question is to ask which types of configurations of solitons and radiation are realized in solutions. The first results of this nature were constructions of solutions blowing up in finite time by bubbling off a single harmonic map by Krieger, Schlag, and Tataru~\cite{KST}, Rodnianski and Sterbenz~\cite{RS}, and Rapha\"el and Rodnianski~\cite{RR}. In~\cite{JJ-AJM}, the first author constructed a solution exhibiting more than one bubble in the decomposition, showing the existence of a solution that forms a $2$-bubble in infinite time with zero radiation in equivariance classes $k \ge 2$. In~\cite{R19} Rodriguez showed that no such $2$-bubble occurs in the case $k=1$, proving that the only non-scattering solution with energy $= 2E(\bs Q)$ blows up by bubbling of a single harmonic map in finite time, while radiating $\bs u^*_0 = - \bs Q$.  It is not known if there are any solutions with more than one bubble in the decomposition when $k=1$.  
\end{rem}

It is natural to ask about the fate of solutions with more than one bubble in the decomposition in the opposite time direction.
 An answer to this question was given by the authors in~\cite{JL1} for the $2$-bubble solution $\bs u_{(2)}(t)$ constructed by the first author in~\cite{JJ-AJM}. We showed that any $2$-bubble in forward time must scatter freely in backwards time.  When the scales of the bubbles become comparable, this ‘collision’ completely annihilates the $2$-bubble structure and the entire solution becomes free radiation, i.e., the collision is inelastic. Viewing the evolution of $\bs u_{(2)}(t)$ in forward time, this means that the $2$-soliton structure emerges from pure radiation, and constitutes an orbit connecting two different dynamical behaviors. We later showed in~\cite{JL2-regularity, JL2-uniqueness} that $\bs u_{(2)}(t)$ is the unique $2$-bubble solution up to sign, translation, and scaling in equivariance classes $k \ge 4$. 
  
 Crucial to the proof of scattering after the collision in the case of two bubbles is the fact that the $2$-bubble configurations considered in~\cite{JL1} are minimal in the sense that any solution in $\calE_{0, 0}$ with energy $< 2 E(\bs Q)$ must scatter (see~\cite{CKLS1}). While inelasticity of collisions is still expected in the case of solutions with more than two bubbles in one time direction,  such a solution can still exhibit bubbling behavior even after a collision that produces radiation -- for example a solution in $\calE_{0, 1}$ with three bubbles and no radiation in one direction could have one bubble with non-zero radiation in the other direction. While we do not consider such refined two-directional analysis here, a relatively straightforward corollary of the proof of Theorem~\ref{thm:main} is that there can be \emph{no elastic collisions of pure multi-bubbles}, which we formulate as a proposition below. 


\begin{defn}
\label{def:pure}
With the notations from the statement of Theorem~\ref{thm:main}, we say that
$\bs u$ is a \emph{pure multi-bubble} in the forward time direction if
$\bs u\lin^* = 0$ in the case $T_+ = +\infty$, and $\bs u_0^* = 0$ in the case $T_+ < +\infty$.

We say that $\bs u$ is a pure multi-bubble in the backward time direction
if $t \mapsto \bs u(-t)$ is a pure multi-bubble in the forward time direction.
\end{defn}

\begin{prop}
\label{prop:inelastic}
Stationary solutions are the only pure multi-bubbles in both time directions. 
\end{prop}

\begin{rem}
  We note that Proposition~\ref{prop:inelastic} was also proved in the case $k=1$ for~\eqref{eq:wmk} in the recent preprint~\cite{DKMM}, as well as for the energy critical focusing NLW under radial symmetry and in odd space dimensions in~\cite{DKM4, DKM9},  all via a different approach based on energy channels. As mentioned above, the case of $N=2$ bubbles was already considered in~\cite{JL1}.   See ~\cite{MM11, MM11-2, MM18} for more regarding the inelastic soliton collision problem for non-integrable PDEs. 
\end{rem}

\subsection{History of progress on the problem}  

Our proof of Theorem~\ref{thm:main} is built on top of two significant partial results, namely (1) that the radiation term, $\bs u^*\lin$ in the global case and $\bs u^*_0$ in the blow-up setting,  can be identified continuously in time, and (2) that the resolution is known to hold along a  well-chosen sequence of times. The result (1) was established in~\cite{CKLS1, CKLS2, Cote15, JK} as a consequence of the classical work of Shatah and Tahvildar-Zadeh~\cite{STZ92}, and we make explicit use of this fact. The latter result (2) was proved by C\^ote~\cite{Cote15} and Jia and Kenig~\cite{JK} using Struwe's classical bubbling analysis~\cite{Struwe}, many ideas from Duyckaerts, Kenig, and Merle's seminal works~\cite{DKM1, DKM2, DKM3}, and several new insights particular to~\eqref{eq:wmk}. While the sequential resolution certainly inspires part of our argument, we cannot use it simply as a black box,  but rather we revisit the proof and derive more precise information from the analysis of C\^ote, and Jia and Kenig as we explain in the next section. 

We discuss these prior results in more detail. To unify the blow-up and global-in-time settings we make the following conventions. Consider a finite energy wave map  $\bs u(t) \in \E_{\ell, m}$. We assume that either $\bs u(t)$ blows up in backwards time at $T_-=0$ and is defined on an interval $I_*:=(0, T_0]$, or $\bs u(t)$ is global in forward time and defined on the  interval $I_* := [T_0, \infty)$ where in both cases $T_0>0$. We let $T_* := 0$ in the blow-up case and $T_* := \infty$ in the global case. 

\textbf{Extraction of the radiation.} 
%
Below we will use the notation $\cE(r_1, r_2)$ to denote the local energy norm
\begin{equation}
\|\bs g\|_{\cE(r_1, r_2)}^2 := \int_{r_1}^{r_2} \Big((\dot g)^2 + (\partial_r g)^2 + \frac{k^2}{r^2}g^2\Big)\,r\vd r,
\end{equation}
By convention, $\cE(r_0) := \cE(r_0, \infty)$ for $r_0 > 0$. The local nonlinear energy is denoted $E(\bs u_0; r_1, r_2)$. 
We adopt similar conventions as for $\cE$ regarding the omission
of $r_2$, or both $r_1$ and $r_2$.

\begin{thm}[Identification of the radiation] \emph{\cite[Propositions 5.1, 5.2]{Cote15}}  \label{thm:stz} 
Let $\bs u(t) \in \E_{\ell, m}$ be a finite energy wave map on an interval $I_*$ as above. Then, the limit $\pi \Z \ni m_\Delta \pi  :=  \lim_{t \to T_*}u(t, \frac 12 t)$  exists, and there is an integer $m_\infty \in \Z$,  a finite energy wave map $\bs u^*(t) \in \E_{0, m_\infty}$  called the radiation, and a function $\rho : I_* \to (0,\infty)$ that satisfies, 
\begin{equation} \label{eq:radiation} 
\lim_{t \to T_*} \big((\rho(t) / t)^k + \|\bs u(t) -  \bs u^*(t) - m_\Delta \bs\pi\|_{\cE(\rho(t))}^2\big) = 0.
\end{equation}
Moroever, for any $\al \in (0, 1)$,  
\EQ{ \label{eq:rad-rho} 
E( \bs u^*(t); 0, \al t) \to 0 \mas t \to T_*. 
}
\end{thm}

\begin{rem} 
In the global setting, i.e., $I_* = [T_0, \infty)$ we must have $m_\infty = 0$ and the linear wave $\bs u\lin^*(t) \in \E$ that appears in Theorem~\ref{thm:main} is the unique solution to the linear equation~\eqref{eq:lin} satisfying,
\EQ{
\| \bs u^*(t) - \bs u\lin^*(t) \|_{\E} \to 0 \mas n \to \infty, 
} 
which one obtains via the existence of wave operators; see Lemma~\ref{lem:Cauchy}. In the finite time blow-up setting the final radiation $\bs u^*_0 \in \E_{m_\De,m}$ that appears in Theorem~\ref{thm:main} is shifted initial data for $\bs u^*(t)$, i.e., the radiation $\bs u^*(t)$ in Theorem~\ref{thm:stz} satisfies $\bs u(t, r)  =  m_\De \bs \pi+  \bs u^*(t, r)$ for $r > t$. With this definition and energy conservation,  Theorem~\ref{thm:seq} implies the energy identity, 
\EQ{ \label{eq:en-ident} 
E( \bs u) = N E( \bs Q) + E( \bs u^*). 
} 
We remark that~\eqref{eq:rad-rho} in the case $T_* = \infty$ uses the estimates for the even dimensional free scalar wave equation proved by C\^ote, Kenig, and Schlag in~\cite{CKS}.

The identification of $\bs u^*(t)$ and the vanishing~\eqref{eq:radiation} uses fundamental technique of Shatah and Tahvildar-Zadeh~\cite{STZ92} (see also Christodoulou and Tahvildar-Zadeh~\cite{CTZduke} for the case of spherically symmetric wave maps);  in~\cite{STZ92} it is proved that every singular wave map has asymptotically no energy  in the self-similar region of the cone, i.e., 
\EQ{
E( \bs u(t); \al t, t)  \to 0 \mas t \to T_*
}
for each $\al \in (0, 1)$ in the case $T_*  = 0$, and 
\EQ{
\lim_{A \to \infty}\limsup_{t \to T_*} E( \bs u(t), \al t, t-A) = 0
}
in the case $T_* = \infty$. Note that the latter refined estimate for globally defined wave maps was proved in~\cite{CKLS2} using methods from~\cite{CTZduke, STZ92}.  
\end{rem} 

\begin{rem} 
The radiation field can be identified in several other contexts and by different means. For example, Tao accomplished this in~\cite{Tao-07} for certain high dimensional NLS. For critical nonlinear waves with power-type nonlinearities, the radiation field can be identified even outside radial symmetry; see the work of Duyckaerts, Kenig, and Merle~\cite{DKM19}. 
\end{rem} 

\textbf{Sequential soliton resolution}. 
The first result in this direction was Struwe's bubbling theorem~\cite{Struwe}, which showed that any smooth solution to~\eqref{eq:wmk} that develops a singularity in finite time must do so by bubbling off at least one harmonic map, locally in space, along some sequence of times. 

A deep insight of Duyckaerts, Kenig, and Merle, proved in~\cite{DKM3} for the energy critical NLW, is that once the linear radiation is subtracted from the solution, the entire remainder should exhibit strong sequential compactness -- it decomposes into a finite sum of asymptotically decoupled elliptic objects, in our case these are stationary harmonic maps, along at least one time sequence, up to an error that vanishes in the energy space. A crucial tool in proving such a compactness statement is the remarkable theory of profile decompositions for dispersive equations developed by Bahouri and G\'erard~\cite{BG}. However, after finding the profiles and their space-time concentration properties (in our case their scales) via the main result in~\cite{BG}, one must identify them as elliptic objects (solitons) by some means, and then prove that the error vanishes in the sense of energy, rather the weaker form of compactness (vanishing in the sense of a Strichartz norm) given by~\cite{BG}.   In the wave map case, this program was carried out by C\^ote, Kenig, the second author, and Schlag~\cite{CKLS1, CKLS2} (using the even dimensional exterior energy estimates proved by C\^ote, Kenig, and Schlag in~\cite{CKS}) for solutions to~\eqref{eq:wmk} with $k=1$ in $\calE_{0, 1}$ with $E < 3 E( \bs Q)$. The latter condition restricted the number of possible configurations to those with a single bubble, and in this special case the sequential resolution could easily be upgraded to a continuous one using the variational characterization of $\bs Q$ and the coercivity of the energy functional. 
 
In our setting, the sequential resolution was proved by C\^ote~\cite{Cote15} in the case $k=1$, and Jia and Kenig~\cite{JK} in the case $k =2$, namely that Theorem~\ref{thm:main} holds along a well-chosen sequence of times. These works used the bubbling theory of Struwe~\cite{Struwe} to identify the profiles as harmonic maps, and in the latter paper the authors used a novel nonlinear multiplier identity to obtain the convergence of the error in the energy space -- in fact, we make use of this same identity in this work, see Section~\ref{sec:compact}. A minor technical observation, which we explain in Remark~\ref{rem:seq}, yields their result in all equivariance classes $k \in \N$. Before stating it, we introduce some notation.

\begin{defn}[Multi-bubble configuration] \label{def:multi-bubble} 
Given $M \in \{0, 1, \ldots\}$, $m \in \bZ$, $\vec\iota = (\iota_1, \ldots, \iota_M) \in \{-1, 1\}^M$
and an increasing sequence $\vec\lambda = (\lambda_1, \ldots, \lambda_M) \in (0, \infty)^M$,
a \emph{multi-bubble configuration} is defined by the formula
\begin{equation}
\bs\calQ(m, \vec\iota, \vec\lambda; r) := m\bs\pi + \sum_{j=1}^M\iota_j\big(\bs Q_{\lambda_j}(r) - \bs\pi\big).
\end{equation}
\end{defn}
\begin{rem}
If $M = 0$, it should be understood that $\bs \calQ(m, \vec\iota, \vec\lambda; r) = m\bs\pi$
for all $r \in (0, \infty)$, where $\vec\iota$ and $\vec\lambda$ are $0$-element sequences,
that is the unique functions $\emptyset \to \{-1, 1\}$ and $\emptyset \to (0, \infty)$, respectively. 

\end{rem}

We state the main theorems from C\^ote~\cite{Cote15} and Jia, Kenig~\cite{JK} using this notation. 
\begin{thm}[Sequential soliton resolution] \emph{\cite[Theorem 1.1]{Cote15}, \cite[Theorem 1.2]{JK}}  \label{thm:seq}  
Let $k \in \N$, $\ell, m \in \Z$, and let $\bs u(t) \in \E_{\ell, m}$ be a finite energy wave map on an interval $I_*$ as above.  Let $m_\Delta,  m_\infty \in \Z$, and  the radiation $\bs u^*(t) \in \E_{0, m_\infty}$ be as in Theorem~\ref{thm:stz}. Then, there exists an integer $N \ge 0$, a sequence of times $t_n \to T_*$, signs $\vec \iota_n \in \{-1, 1\}^N$, and scales $\vec \lam_n \in (0, \infty)^N$  such that, 
 \EQ{
 \lim_{n \to \infty} \Big( \| \bs u(t_n) - \bs u^*(t_n)  - \bs \calQ( m, \vec \iota_n, \vec \lam_n) \|_{\E} + \sum_{j =1}^{N}  \frac{\lam_{n, j}}{\lam_{n, j+1}} \Big)  = 0, 
 }
 where above we use the convention $\lam_{n, N+1} := t_n$. 

\end{thm} 

\begin{rem} 
The Duyckaerts, Kenig, and Merle approach from~\cite{DKM3} to sequential soliton resolution has been successful in other settings. The same authors with Jia proved the sequential decomposition for the full energy critical NLW (i.e., not assuming radial symmetry) in~\cite{DJKM1} and for wave maps outside equivariant symmetry for data with energy slightly above the ground state~\cite{DJKM2}, where the perturbative regularity theory of Tao~\cite{Tao2} could be used; see also the bubbling theory of Grinis~\cite{Gri}. See also~\cite{CKLS3} for the radially symmetric energy critical NLW in four space dimensions, and~\cite{Rod16-adv} for the same equation in odd space dimensions. 
\end{rem}

\subsection{Summary of the proof: collision intervals and no-return analysis}

The challenging nature of bridging the  gap between Theorem~\ref{thm:seq}, which is the resolution along one sequence of times, and Theorem~\ref{thm:main} is apparent from the following consideration. The sequence $t_n \to T_*$ in Theorem~\ref{thm:seq} gives no relationship between the lengths of the time intervals $[t_n, t_{n+1}]$ and the concentration scales $\vec \lam_n$ of the various harmonic maps  in the decomposition. One immediate enemy 
is then the  possibility of \emph{elastic collisions}.
If colliding solitons could recover their shape after a collision, then one could potentially
encounter the following scenario: the solution approaches a multi-soliton configuration for a sequence of times,
but in between infinitely many collisions take place, so that there is no soliton resolution in continuous time.

We describe our approach.  Fix $\bs u(t) \in \E_{\ell, m}$,  a finite energy solution to~\eqref{eq:wmk} on the time interval $I_*$ as defined above. Let $N \ge 0$, $m_{\infty}, m_{\De} \in \Z$, and the radiation $\bs u^*(t) \in \E_{0, m_{\infty}}$ be as in Theorem~\ref{thm:seq}.  We define a 
\emph{multi-bubble proximity function} at each $t \in I_*$ by
\begin{equation} \label{eq:d-intro} 
\bfd(t) := \inf_{\vec \iota, \vec\lam}\bigg( \| \bs u(t) - \bs u^*(t) - \bs\calQ(m_\Delta, \vec\iota, \vec\lambda) \|_{\cE}^2 + \sum_{j=1}^{N}\Big(\frac{ \lam_{j}}{\lam_{j+1}}\Big)^{k} \bigg)^{\frac{1}{2}},
\end{equation}
where $\vec\iota := (\iota_{1}, \ldots, \iota_N) \in \{-1, 1\}^{N}$, $\vec\lambda := (\lambda_{1}, \ldots, \lambda_N) \in (0, \infty)^{N}$,  and $\lambda_{N+1} := t$. We note that $\bfd(t)$ is a continuous function on $I_*$.

With this notation, we see that Theorem~\ref{thm:seq} gives a monotone sequence of times $t_n \to T_*$ such that, 
\begin{equation} \label{eq:seq} 
\lim_{n \to \infty} \bfd(t_n) = 0.
\end{equation}
Theorem~\ref{thm:main} is an immediate consequence of showing that $\lim_{t \to T_*} \bfd(t) = 0$. We argue by contradiction, assuming that $\limsup_{t  \to T_*} \bfd(t) >0$.  This means that there is  some sequence of times where $\bs u(t) - \bs u^*(t)$ approaches an $N$-bubble and another sequence of times for which it stays bounded away from $N$-bubble configurations. It is natural to rule out this  behavior  by proving what is called a \emph{no-return} lemma. In this generality, our approach  is inspired by no-return results for one soliton by Duyckaerts and Merle~\cite{DM08, DM09}, Nakanishi and Schlag~\cite{NaSc11-1, NaSc11-2}, and Krieger, Nakanishi and Schlag~\cite{KNS13, KNS15}. The exponential instability considered in those works is absent here, but is replaced by attractive nonlinear interactions between the solitons. This latter consideration, and indeed the overall scheme of the proof is based on our previous work \cite{JL1},
where modulation analysis  of bubble interactions was used
for the first time in the context of the soliton resolution problem (in fact, we recently showed that the collision analysis in~\cite{JL1} yielded a quick proof of Theorem~\ref{thm:main} in the special cases when at most two bubbles appear in the decomposition; see~\cite{JL5}). 

The basic tool we use is the standard virial functional 
\EQ{
\fv(t) := \int_0^\infty\partial_t u(t) r\partial_r u(t) \chi_{\rho(t)}\,r\vd r, 
}
where the cut-off $\chi$ is placed along a Lipschitz curve $r = \rho(t)$ that will be carefully chosen (note that a time-dependent cut-off of the virial functional was also used in~\cite{NaSc11-1, NaSc11-2}).  Differentiating $\fv(t)$ in time we have, 
\EQ{ \label{eq:v'} 
\fv'(t) = - \int_0^\infty \abs{\p_t u(t, r)}^2\chi_{\rho(t)}(r) \, \rdr + \Om_{\rho(t)}(\bs u(t)), 
}
where $\Om_{\rho(t)}(\bs u(t))$ is the error created by the cut-off. Importantly, this error has structure, see Lemmas~\ref{lem:vir} and \ref{lem:virial-error}, and satisfies the estimates, 
\EQ{
\Om_{\rho(t)}(\bs u(t)) \lesssim (1+ \abs{ \rho'(t)}) \min \{ E(\bs u(t); \rho(t), 2 \rho(t)), \bfd(t)\} . 
}
Roughly, this allows us to think of $\fv(t)$ as a Lyapunov functional for our problem, localized to scale $\rho(t)$, with ``almost'' critical points given by multi-bubbles $\bs\calQ(m, \vec \iota, \vec \lam)$. Indeed, if $\bs u(t)$ is close to a multi-bubble up to scale $\rho(t)$, and $\abs{\rho'(t)} \lesssim 1$,  then $\abs{\fv'(t)} \lesssim \bfd(t)$. 

Our first result is a localized compactness lemma. In Section~\ref{sec:compact} we prove the following: given a sequence of wave maps $\bs u_n(t) \in\E_{\ell,m}$  on time intervals $[0, \tau_n]$ with bounded energy, and a sequence $R_n \to \infty$ such that 
\EQ{
\lim_{n \to \infty} \frac{1}{\tau_n} \int_0^{\tau_n} \int_0^{R_n\tau_n} \abs{ \p_t u_n(t, r)}^2 \, r \, \ud r \, \ud t  = 0, 
}
one can find a new sequence $1 \ll r_n \ll R_n$ and a sequence of times $s_n \in [0, \tau_n]$,  so that up to passing to a subsequence of the $\bs u_n$,  we have $\lim_{n \to \infty} \bs \de_{r_n \tau_n}( \bs u_n(s_n)) = 0$. Here  $\bs \de_{R}(\bs u)$ is a local (up to scale $R$) version of the distance function $\bfd$. We note that the sequential decomposition Theorem~\ref{thm:seq} is an almost immediate consequence of the localized compactness lemma along with the Shatah and Tahvildar-Zadeh theory; see Remark~\ref{rem:seq}. The proof of the compactness lemma is very similar in spirit to the analysis of C\^ote~\cite{Cote15} and Jia and Kenig~\cite{JK}.   

We give a caricature of the no-return analysis, pointing the reader to the technical arguments in Sections~\ref{sec:decomposition},~\ref{sec:conclusion} for the actual arguments.  
We would like to integrate~\eqref{eq:v'} over intervals $[a_n, b_n]$ with $a_n, b_n \to T_*$ such that $\bfd(a_n), \bfd(b_n) \ll 1$ but contain some subinterval $[c_n, d_n] \subset [a_n, b_n]$ on which $\bfd(t) \simeq 1$; such intervals exist under the contradiction hypothesis.  From~\eqref{eq:v'} we obtain, 
\EQ{ \label{eq:vir-ineq} 
\int_{a_n}^{b_n} \int_0^{\rho(t)}\abs{\p_t u(t, r)}^2\,  \rdr \, \ud t \lesssim \rho(a_n) \bfd(a_n) + \rho(b_n) \bfd(b_n) + \int_{a_n}^{b_n}\abs{\Om_{\rho(t)}(\bs u(t))} \, \ud t. 
}
 We consider the choice of $\rho(t)$. One can use the sequential compactness lemma so that choosing $\rho(t)/(d_n - c_n) \gg 1$  we have, 
\EQ{ \label{eq:compact} 
\int_{c_n}^{d_n}\int_0^{\rho(t)}\abs{\p_t u(t, r)}^2\chi_{\rho(t)}(r) \, \rdr \, \ud t \gtrsim d_n-c_n , 
}
and one can expect that the integral of the error $\int_{c_n}^{d_n} \abs{\Om_{\rho(t)}(\bs u(t))} \, \ud t \ll \abs{d_n-c_n}$ absorbs into the left-hand side by choosing $\rho(t)$ to lie in a region where $\bs u(t)$ has negligible energy.  

To complete the proof one would need to show that the error generated on the intervals $[a_n, c_n]$ and $[d_n, b_n]$ can also be absorbed into the left-hand side, and moreover that the terms $\rho(a_n) \bfd(a_n),  \rho(b_n) \bfd(b_n) \ll d_n - c_n$. To accomplish this, we require a more careful choice of the intervals $[a_n, b_n]$ and placement of the cut-off $\rho(t)$, which 
motivates the notion of \emph{collision intervals} introduced  in Section~\ref{ssec:proximity}. These allow us to distinguish between ``interior'' bubbles that come into collision, and ``exterior'' bubbles,  which stay coherent throughout the intervals $[a_n, b_n]$, and to ensure we place the cutoff in the region between the interior and exterior bubbles. 

Given $K \in \{1, \dots, N\}$, 
we say that an interval $[a, b]$ is a collision interval with parameters $0<\eps< \eta$ and $N-K$ exterior bubbles for some $1 \le K \le N$, if $\bfd(a),  \bfd(b) \le \eps$, there exists a $c \in [a, b]$ with $\bfd(c) \ge \eta$, and a curve $r = \rho_K(t)$ outside of which $\bs u(t) - \bs u^*(t)$ is within $\eps$ of an $N-K$-bubble in the sense of~\eqref{eq:d-intro} (a localized version of $\bfd(t)$); see Defintion~\ref{def:collision}.  We now define $K$ to be the \emph{smallest} non-negative integer for which there exists $\eta>0$, a sequence $\eps_n \to 0$,  and sequences $a_n, b_n \to T_*$, so that $[a_n, b_n]$ are collision intervals with parameters $\eps_n, \eta$ and $N-K$ exterior bubbles, and we write $[a_n, b_n] \in \calC_K( \eps_n, \eta)$; see Section~\ref{ssec:proximity} for the proof that $K$ is well-defined and $\ge 1$, under the contradiction hypothesis.  

We revisit~\eqref{eq:vir-ineq} on a sequence of collision intervals $[a_n, b_n] \in \calC_K( \eps_n, \eta)$. Near the endpoints $a_n, b_n$,  $\bs u(t) - \bs u^*(t)$ is close to an $N$-bubble configuration and we denote the interior scales, which will come into collision, by $\vec \lam = ( \lam_1, \dots, \lam_K)$ and the exterior scales, which stay coherent, by $\vec \mu = ( \vec \mu_{K+1}, \dots, \vec \mu_N)$. We assume for simplicity in this discussion that the collision intervals have only a single subinterval $[c_n, d_n]$ as above, and that $\bfd(t)$ is sufficiently small on the intervals $[a_n, c_n]$ and $[d_n, b_n]$ so that the interior scales are well defined (via modulation theory) there. We call $[a_n, c_n], [d_n, b_n]$ \emph{modulation intervals} and $[c_n, d_n]$ \emph{compactness intervals}. 

 The scale of the $K$th bubble $\lam_K(t)$ plays an important role and must be carefully tracked. We will need to also make sense of this scale on the compactness intervals, where the bubble itself may lose its shape from time to time. We do this by energy considerations; see Definition~\ref{def:mu}. Crucially, the minimality of $K$ can be used to ensure that the intervals $[c_n, d_n]$ as above satisfy $d_n -c_n \simeq \max\{\lam_K(c_n), \lam_K(d_n)\}$; see Lemma~\ref{lem:cd-length}. Thus the first terms on the right-hand-side of~\eqref{eq:vir-ineq} can be absorbed using~\eqref{eq:compact} by ensuring $\rho(a_n) = o(\eps_n^{-1}) \lam_K(a_n), \rho(b_n) = o(\eps_n^{-1}) \lam_K(b_n)$ if we can additionally prove that  the scale $\lam_K(t)$ does not change much on the modulation intervals. Note that our choice of cut-off  will  satisfy $\lam_K(t) \ll \rho(t) \ll \mu_{K+1}(t)$. 
 
 We  must also absorb the errors $(\int_{a_n}^{c_n}+ \int_{d_n}^{b_n})|\Om_{\rho(t)}(\bs u(t))| \, \ud t  \lesssim (\int_{a_n}^{c_n}+ \int_{d_n}^{b_n}) \bfd (t) \, \ud t$ on the modulation intervals. Here we perform a refined modulation analysis on the interior bubbles, which allows us to track the growth of $\bfd(t)$ through a collision of (possibly) many bubbles. Roughly, up to scale $\rho(t)$, $\bs u(t)$ looks like a $K$-bubble, and using the implicit function theorem we define modulation parameters $\vec \iota$,  $\vec \lam(t)$, and error $\bs g(t)$ with 
 \EQ{
 \bs u(t, r) = \bs\calQ(m_n, \vec \iota, \vec \lam(t); r) + \bs g(t, r), \mif r \le \rho(t),\quad \La\Lam Q_{\lam_j(t)} \mid g(t) \Ra = 0, \mfor \, \,  j = 1,\dots, K, 
 }
 where $\Lam:= r\p_r$ is the generator of the $H$-invariant scaling (note that for $k=1, 2$ the decomposition is slightly different due to the slow decay of $\Lam Q$) and
 \EQ{ \label{eq:inner-prod}
\ang{\phi \mid g}  := \int_0^\infty f(r) g(r) \, r \ud r, \qquad \text{for }\phi, g : (0, \infty) \to \bR.
} The orthogonality conditions and an expansion of the nonlinear energy of $\bs u(t)$ up to scale $\rho(t)$  lead to the coercivity estimate, 
 \EQ{
 \| \bs g(t) \|_{\E} + \sum_{j \neq \calA} \Big( \frac{\lam_{j}}{\lam_{j+1}} \Big)^{\frac{k}{2}} \lesssim \max_{i \in \calA} \Big( \frac{\lam_i(t)}{\lam_{i+1}(t)} \Big)^{\frac{k}{2}} + o_n(1)  \simeq \bfd(t) + o_n(1), 
 }
where $\calA = \{ j \in 1, \dots, K-1 \, : \, \iota_j \neq \iota_{j+1}\}$ captures the alternating bubbles (which experience an attractive interaction force) and  the $o_n(1)$ term comes from errors due to the presence  of the radiation $\bs u^*$ in the region $r \lesssim \rho(t) \ll t$. In fact, since $\bfd(t)$ grows out of the modulation intervals we can absorb these errors into $\bfd(t)$ by enlarging the parameter $\eps_n$ and requiring the lower bound $\bfd(t) \ge \eps_n$ on the modulation intervals. 

The growth of $\bfd(t)$ is then captured by the dynamics of the alternating bubbles, which, since~\eqref{eq:wmk} is second order, enter at the level of $\lam_{j}''(t)$. However, it is not clear how to derive useful estimates from the equation for $\lam''(t)$ obtained by twice differentiating the orthogonality conditions. To cancel terms with critical size, but indeterminate sign, we introduce a localized virial correction to $\lam_j'\simeq - \iota_j  \| \Lam Q \|_{L^2}^{-2} \lam_j^{-1}\La \Lam Q_{\lam_j} \mid \dot g\Ra$, defining 
\EQ{
\beta_j'(t) = - \iota_j \| \Lam Q \|_{L^2}^{-2}\La \Lam Q_{\U{\lam_j(t)}} \mid \dot g(t)\Ra  -   \| \Lam Q \|_{L^2}^{-2}\ang{ \uln A( \lam_j(t)) g(t) \mid \dot g(t)}, 
}
where $\U A(\lam)$ is a truncated (to scale $\lam$) version of  $\U{\Lam} = \Lam +1$, the generator of $L^2$ scaling. Roughly, we show in Sections~\ref{ssec:ref-mod} and~\ref{ssec:bub-dem},  that if  the distance $\bfd(t)$ is dominated at a local minimum $t_0$ by the ratio between the $j$-th bubble and its larger neighbor with opposite sign, then we can control dynamics of $\beta_j(t)$ near $t_0$,  showing that $\bfd(t)$ grows in a controlled way until some other bubble ratio becomes dominant, and so on, until we exit the modulation interval.  All the while we can ensure that the $K$th scale does not move much, and we obtain bounds of the form $(\int_{a_n}^{c_n}+ \int_{d_n}^{b_n}) \bfd (t) \, \ud t \lesssim \bfd(c_n)^{\frac{2}{k}}\lam_K(a_n) + \bfd(d_n)^{\frac{2}{k}}\lam_K(b_n)$ (see the ``ejection'' Lemma~\ref{lem:ejection}). Thus the errors can be absorbed into the left-hand side of~\eqref{eq:vir-ineq} and we obtain a contradiction. 

A similar, but simpler 
refined  modulation  
analysis was performed in~\cite{JL1}.  
The use of such refinements to modulation parameters to obtain dynamical control was introduced by the first author in the context of a two-bubble construction for $NLS$  in~\cite{JJ-APDE}. The notion of localized virial corrections in the context of energy/Morawetz-type estimates was developed by Rapha\"el and Szeftel in~\cite{RaSz11}.

\subsection{Notational conventions}
%
%
The energy is denoted $E$, $\cE$ is the energy space, $\cE_{\ell, m}$ are the finite energy sectors.

Given a function $\phi(r)$ and $\lambda>0$, we denote by $\phi_{\lam}(r) = \phi(r/ \lam)$, the $H$-invariant re-scaling, and by $\phi_{\U{\lam}}(r) = \lam^{-1} \phi(r/ \lambda)$ the $L^2$-invariant re-scaling. We denote by $\Lam :=r \p_r$ and $\ULam := r \partial r +1$ the infinitesimal generators of these scalings. We denote $\ang{\cdot\mid\cdot}$
the radial $L^2(\bR^2)$ inner product given by \eqref{eq:inner-prod}.

We denote $k$ the equivariance degree and $f(u) := \frac{1}{2} \sin 2u$ the nonlinearity in \eqref{eq:wmk}.
We let $\chi$ be a smooth cut-off function, supported in $r \leq 2$ and equal $1$ for $r \le 1$.

The general rules we follow giving names to various objects are:
\begin{itemize}
\item index of an infinite sequence: $n$
\item sequences of small numbers: $\gamma, \delta, \epsilon, \zeta, \eta, \theta$
\item scales of bubbles and quantities describing the spatial scales: $\lambda, \mu, \nu, \xi, \rho$;
in general we call $\lambda$ the scale of the interior bubbles and $\mu$ the exterior ones
(once these notions are defined)
\item moment in time: $t, s, \tau, a, b, c, d, e, f$
\item indices in summations: $ i, j, \ell$
\item time intervals: $I, J$
\item number of bubbles: $K, M, N$
\item signs are denoted $\iota$ and $\sigma$
\item boldface is used for pairs of elements related to the Hamiltonian structure; an arrow is used for vectors (finite sequences) in other contexts.
\end{itemize}
We call a ``constant'' a number which depends only on the equivariance degree $k$ and the number of bubbles $N$.
Constants are denoted $C, C_0, C_1, c, c_0, c_1$. We write $A \lesssim B$ if $A \leq CB$ and $A \gtrsim B$ if $A \geq cB$.
We write $A \ll B$ if $\lim_{n\to \infty} A / B = 0$.

For any sets $X, Y, Z$ we identify $Z^{X\times Y}$ with $(Z^Y)^X$, which means that
if $f: X\times Y \to Z$ is a function, then for any $x \in X$ we can view $\phi(x)$ as a function $Y \to Z$
given by $(\phi(x))(y) := \phi(x, y)$.

\section{Preliminaries}

\subsection{Basic properties of finite energy maps} We aggregate here several  well known results. 


\begin{lem}  \label{lem:pi} 
Fix integers $\ell, m$. For every $\eps>0$ and $R_0 >1$,  there exists a $\de>0$ with the following property. Let $0 \le R_1 < R_2\le \infty$ with $R_2/ R_1 \ge R_0$, and $\bs u \in \E_{\ell, m}$ be such that $E( (u, 0); R_1, R_2)  < \de$. Then, there exists $\ell_0 \in \Z$  such that $| u(r) - \ell_0 \pi|<\eps$ for almost all $r \in (R_1, R_2)$. 

Moreover, there exist  constants $C=C(R_0), \al= \al(R_0)>0$ such that if $E( (u, 0); R_1, R_2)< \al$,  then 
\EQ{ \label{eq:H-E-comp} 
 \| \bs u - \ell_0 \bs \pi \|_{\E(R_1, R_2)} \le C E( \bs u; R_1, R_2). 
 }
\end{lem} 
\begin{proof} 
By an approximation argument we can assume $(u, 0) \in \E_{\ell, m}$ is smooth. First, we show that for any $\eps_0>0$ exists $r_0 \in [R_1, R_2]$ such that $| u(r_0) - \ell_0 \pi|< \eps_0$ for some $\ell_0 \in \Z$ as long as $E( (u, 0); R_1, R_2)$ is sufficiently small. 
If not, one could find $\eps_1>0$,  $0< R_1< R_2$, and a sequence $(u_n, 0) \in \E_{\ell, m}$ so that $E( (u_n, 0); R_1; R_2) \to 0$ as $n \to \infty$ but such that $\inf_{r \in[R_1, R_2], \ell \in \Z}   | u_n(r) - \ell \pi| \ge \eps_1$. The latter condition gives a constant $c( \eps_1)>0$ such that $\inf_{r \in [R_1, R_2]} |\sin( u_n(r))| \ge c(\eps_1)$. But then
\EQ{
E( (u_n, 0); R_1; R_2)  \ge \frac{k^2}{2}  \int_{R_1}^{R_2} \sin^2( u_n(r)) \,  \frac{\ud r}{r} \ge \frac{k^2}{2} c(\eps_1)^2 \log (R_2/R_1), 
}
which is a contradiction. Next define the function, 
$
G(u) = \int_0^u \abs{ \sin \rho} \, \ud \rho, 
$
and for $r_1 \in (R_1, R_2)$  note the inequality, 
\EQ{
\abs{G(u(r_0)) - G( u(r_1))} = \Big|\int_{u(r_1)}^{u(r_0)}  \abs{ \sin \rho} \, \ud \rho \Big| = \Big|\int_{r_1}^{r_0} \abs{ \sin u(r) } \abs{ \p_r u(r)} \, \ud r \Big| \lesssim E((u, 0); R_1, R_2). 
}
We conclude using that $G$ is continuous and increasing that $| u(r) - \ell_0 \pi|<\eps$ for all $r \in (R_1, R_2)$. As long as $\eps>0$ is small enough we see that in fact, $\sin^2(u(r)) \ge \frac{1}{2}| u(r) - \ell_0 \pi|^2$ for all $r \in (R_1, R_2)$ and~\eqref{eq:H-E-comp} follows. 
\end{proof}

  We have the following version of the principle of finite speed of propagation.
\begin{lem}
Let $\bs u(t)$ be a solution to \eqref{eq:wmk} on the time interval $[0, T]$. Then
\begin{equation}
\label{eq:energy-monoton}
E(\bs u(T); 0, R - T) \leq E(\bs u(0); 0, R), \qquad \text{for all }R \geq T.
\end{equation}
\end{lem}
\begin{proof}
It suffices to consider the case of a smooth solution
and then approximate a finite energy solution by smooth ones.
For a proof in the smooth case, see \cite[Section 2]{STZ92}.
\end{proof}
\begin{rem}
The energy conservation yields the following equivalent formulation:
\begin{equation}
\label{eq:ext-energy-monoton}
E(\bs u(T); R + T) \leq E(\bs u(0); R), \qquad\text{for all }R \geq 0.
\end{equation}
\end{rem}

 We have the following virial identity. 
 \begin{lem}[Virial identity] \label{lem:vir} 
 Let $\bs u(t)$ be a solution to~\eqref{eq:wmk} on an open time interval $I$ and $\rho: I \to (0, \infty)$
 a Lipschitz function. Then for almost all $t \in I$, 
 \EQ{\label{eq:vir}
 \frac{\ud}{\ud t} \ang{  \p_t u(t) \mid \chi_{\rho(t)}^2 \, r \p_r u(t)}  = - \int_0^\infty (\p_t u(t, r)\chi_{\rho(t)}(r))^2 \, \rdr + \Om_{\rho(t)}(\bs u(t)), 
 }
 where 
 \EQ{ \label{eq:Om-gamma-def} 
 \Om_{\rho(t)}(\bs u(t)) :=  &- 2\frac{\rho'(t)}{\rho(t)} \int_0^\infty \p_t u(t, r) r \p_r u(t, r) \chi_{\rho(t)}(r)  \Lam \chi_{\rho(t)}(r) \, \rdr \\
 & - \int_0^\infty  \Big( (\p_t u(t, r))^2 + (\p_r u(t, r))^2 - k^2 \frac{\sin^2 u(t, r)}{r^2} \Big)  \chi_{\rho(t)}(r)  \Lam \chi_{\rho(t)}(r) \rdr.
 }
 \end{lem}
 \begin{proof}
The proof is a direct computation along with an approximation argument for fixed $t \in I$, assuming $\rho$ is differentiable at $t$.
  \end{proof}

\subsection{Local Cauchy theory} \label{sec:Cauchy}

The following theorem was proved by Shatah and Tahvildar-Zadeh in~\cite{STZ92, STZ94}.  

\begin{lem}[Local well-posedness]\label{lem:lwp} \emph{\cite[Theorem 1.1]{STZ94},\cite[Theorem 8.1]{SSbook}\,  \cite{STZ92}} Let $\ell, m \in \Z$ and let $\bs u_0 \in \E_{\ell, m}$. Then, there exist a maximal time interval of existence $(T_-, T_+) = I_{\max}(\bs u_0) \ni 0$ on which~\eqref{eq:wmk} admits a unique solution $\bs u(t)$ in the space $ C^0(I_{\max}; \E_{\ell, m})$ with $\bs u(0) = \bs u_0$. 

In fact, there exists $\eps_0 >0$ with the following property. Let $\bs u_0 \in \E_{\ell, m}$, $\tau>0$ and  suppose the solution $\bs u(t)$ to~\eqref{eq:wmk} with data $\bs u(0) = \bs u_0$ is defined on the interval $[0, \tau)$, i.e., in $C^0([0, \tau); \E_{\ell, m})$. Suppose that there exists a time $0\le t< \tau$ and a number $R> \tau - t$ such that, 
\EQ{
E( \bs u(t), 0; R) < \eps_0. 
}
Then,  $T_+( \bs u) > \tau$. 
\end{lem} 

%
%

See Struwe~\cite[p. 817]{Struwe} for the continuation criterion in the second paragraph of Lemma~\ref{lem:lwp} in the case of smooth initial data, and see~\cite[Theorem 8.1]{SSbook} for the global well-posedness theorem for energy class equivariant wave maps with sufficiently small energy. 
 Key to the proof are Strichartz estimates for the wave equation (see, e.g., Lindblad, Sogge~\cite{LinS}, and Ginibre, Velo~\cite{GiVe95}), after noticing that the linearization of \eqref{eq:wmk} about the zero solution is equivalent, in the energy space, to the free scalar wave equation in dimension $d = 2k +2$. Indeed, 
the linearization of~\eqref{eq:wmk} about the zero solution is given by the linear wave equation, 
\EQ{ \label{eq:lin-2d} 
\p_t^2 v - \p_r^2 v -\frac{1}{r} \p_r v + \frac{k^2}{r^2} v = 0. 
}
We will sometimes use the notation $v_{\Lin}(t) = S\lin(t) \bs v_0$ as the unique solution to~\eqref{eq:lin-2d} with initial data $\bs v_{\Lin}(0) = \bs v_0 \in \E$.  The mapping $\E \ni \bs v(t) \mapsto \bs W(t) \in ( \dot H^{1}\times L^2)_{\textrm{rad}} (\R^{2k+2})$ defined by 
\EQ{ \label{eq:2k+2} 
\bs W( t, r) := (r^{-k} v( t, r), r^{-k} \p_t v(t, r))
} satisfies $\| \bs v(t)\|_{\E} \simeq \| \bs W(t) \|_{(\dot H^1 \times L^2)_{\textrm{rad}} (\R^{2k+2})}$ and $\bs v(t) \in\E$ solves~\eqref{eq:lin-2d} if and only if $\bs W(t) \in ( \dot H^1 \times  L^2)_{\textrm{rad}} $ solves 
\EQ{\label{eq:lin-2k+2}
\p_t^2 W - \De_{2k+2} W = 0 ,
}
where $\De_{2k+2} = \p_r^2 + \frac{2k+1}{r} \p_r$ is the radial Laplacian in dimension $d = 2k +2$.

For equivariance classes $k >2$, this leads to a spatial dimension $d>6$ and inconvenient technical complications.  However, we observed in~\cite{JL1} that one may give a unified local Cauchy theory for~\eqref{eq:wmk} for all equivariance classes $k \in\N$ based on Strichartz estimates for linear waves with a critical repulsive potential proved by Planchon, Stalker, Tahvildar-Zadeh~\cite{PST03b}. For this purpose, consider the mapping,  
\EQ{\label{eq:2-to-4-map} 
v_0(r) \mapsto V_0( r):= r^{-1} v_0( r), \quad \dot v_0( r) \mapsto   \dot V_0( r) := r^{-1} \dot v_0( r) .
}
We see that $\bs v(t) = ( v(t), \p_t v(t))$ solves~\eqref{eq:lin-2d} if and only if $\bs V(t) = (v(t), \p_t v(t))$ solves
\EQ{\label{eq:lin-4d} 
\p_t^2 V - \p_r^2 V - \frac{3}{r} \p_r V + \frac{k^2 -1}{r^2} V = 0 . 
}

For each $k\ge 1$, define the norm $H_k$ for radially symmetric functions $V$ on $\R^4$ by 
\EQ{
\| V \|_{H_k(\R^4)}^2:= \int_0^\infty \left[(\p_r V)^2 + \frac{(k^2-1)}{r^2} V^2 \right] \, r^3 \, \ud r. 
}
Solutions $\bs V(t)$ to~\eqref{eq:lin-4d} conserve the $H_k \times L^2$ norm and by Hardy's inequality we have 
\EQ{ \label{eq:hardy} 
\|V \|_{H_k(\R^4)} \simeq \| V \|_{\dot{H}^1(\R^4)}.
} 
Thus the mapping~\eqref{eq:2-to-4-map} 
 satisfies 
 \EQ{ \label{eq:2-4}
\|(V_0, \dot V_0) \|_{\dot{H}^1 \times L^2(\R^4)} \simeq  \| (V_0, \dot V_0) \|_{H_k \times L^2 (\R^4)}  = \| (v_0, \dot v_0) \|_{H \times L^2 (\R^2)}. 
 }
We conclude that the Cauchy problem for~\eqref{eq:lin-4d} with initial data in $\dot{H}^1 \times L^2(\R^4)$ is equivalent to the Cauchy problem for~\eqref{eq:lin-2d} for initial data $(v_0, \dot v_0) \in H \times L^2$. As a consequence, Strichartz estimates for solutions to~\eqref{eq:lin-2d} are inherited from Strichartz estimates for~\eqref{eq:lin-4d} proved by Planchon, Stalker, and Tahvildar-Zadeh~\cite{PST03b}.

\begin{lem}[Strichartz estimates for~\eqref{eq:lin-4d}]\emph{ \cite[Corollary 3.9]{PST03b}} \label{l:strich} Fix $k \ge 1$ and let $\bs V(t)$ be a radial solution to the linear equation 
\EQ{
\p_t^2 V - \p_r^2 V - \frac{3}{r} \p_r V + \frac{k^2 -1}{r^2} V  = F(t, r), \quad \bs V(0) = (V_0, \dot V_0) \in \dot{H}^1 \times L^2 (\R^4).
}
Then, for any time interval $0 \in I \subset \R$ we have 
\EQ{ \label{eq:strich-4d} 
\| V \|_{(L^{3}_t L^6_x \cap L^5_{t, x})(I \times \R^4)} + \sup_{t \in I}\| \bs V(t) \|_{\dot{H}^1 \times L^2(\R^4)} \lesssim   \| \bs V(0) \|_{\dot{H}^1 \times L^2(\R^4)} + \| F \|_{L^1_t L^2_x(I \times \R^4)}, 
} 
where the implicit constant above is independent of $I$. 
\end{lem}  

We define the Strichartz norm, 
\EQ{
\| v \|_{\calS(I) }:=  \| r^{-\frac{3}{5}}v  \|_{L^{5}_{t, r}(I)} + \| r^{-\frac{2}{3}}v  \|_{ L^{3}_tL^6_r(I)}
} 
and recall that the notation $L^p_r$ refers to the Lebesgue space on $(0, \infty)$ with respect to the measure $r \, \ud r$.

\begin{cor}[Strichartz estimates for~\eqref{eq:lin-2d}] \label{cor:strich} 
Fix $k \ge 1$ and let $\bs v(t)$ be a radial solution to the linear equation 
\EQ{
\p_t^2 v - \p_r^2 v -\frac{1}{r} \p_r v + \frac{k^2}{r^2} v = F(t, r), \quad \bs v(0) = (v_0, \dot v_0) \in \E = H \times L^2 . 
}
Then, for any time interval $0 \in I \subset \R$ we have 
\EQ{ \label{eq:strich-2d} 
\| v  \|_{\calS(I)}  + \| \bs v(t) \|_{L^{\infty}_t (H \times L^2)(I)} \lesssim   \| \bs v(0) \|_{H \times L^2} + \| F \|_{L^1_t L^2_r(I)}, 
} 
where the implicit constant above is independent of $I$. 
\end{cor} 

Writing the Cauchy problem for~\eqref{eq:wmk} in the class $ \E = \E_{0, 0}$ as 
\EQ{ \label{eq:Cauchy} 
\p_t^2 u - \De u + \frac{k^2}{r^2} u &= \frac{k^2}{2r^2}(2 u - \sin 2u)\\
\bs u(0) &= (u_0, u_1) \in  \E = H \times L^2. 
}
a standard argument based on the contraction mapping principle yields the following result; see for example~\cite{CKM}.

\begin{lem}[Cauchy theory in $\E_{0, 0}$] \label{lem:Cauchy}  There exist functions $\de_0, C_0: [0, \infty) \to (0, \infty)$ with the following properties. Let $A\ge 0$ and let $\bs u_0 = (u_0, u_1) \in \E$ with $\|\bs u_0 \|_{\E} \le A$. Let $I \ni 0$ be an open interval such that 
\EQ{
\| S\lin(t) \bs u_0 \|_{\calS(I)}   = \de \le  \de_0(A). 
}
Then there exists a unique solution $\bs u(t)$ to~\eqref{eq:Cauchy} in the space $C^0(I; \E) \cap \calS(I)$ with initial data $\bs u(0) = \bs u_0$. Moreover, $\bs u(t)$ satisfies the bounds $\| u \|_{\calS(I)} \le C(A) \de$, and $ \| \bs u \|_{L^\infty_t(I; \E)} \le C(A)$. To each solution $\bs u(t)$ to~\eqref{eq:Cauchy} we can associate a maximal interval of existence $I_{\max}(\bs u)$ such that for each compact subinterval $I' \subset I_{\max}$ we have $\| u \|_{\calS(I')} < \infty$. 

Moreover, the completeness of wave operators holds: there exists $\eps_0$ small enough so that if $\bs u_0 \in \E$ satisfies $E(\bs u_0) < \eps_0$, the solution $\bs u(t)$ given above is defined globally in time, satisfies the bound, 
\EQ{
\sup_{t \in \R}\| \bs u(t) \|_{\E} + \| u \|_{\calS(\R)} \lesssim \|\bs u_0\|_{\E}
}
and scatters in the following sense: there exist solutions $\bs u_\Lin^{\pm}(t) \in \E$ to~\eqref{eq:lin-2d} such that 
\EQ{ \label{eq:scattering} 
\| \bs u - \bs u\lin^{\pm}(t) \|_{\E} \to 0 \mas  t \to \pm \infty
}
Conversely, the existence of wave operators holds, i.e.,  for any solution $\bs v_{\Lin}(t) \in \E$ to the linear equation~\eqref{eq:lin-2d}, there exists a unique, global-in-forward time solution $\bs u(t) \in \E$ to~\eqref{eq:Cauchy} such that~\eqref{eq:scattering} holds as $t \to \infty$. An analogous statement holds for negative times.

\end{lem}

%

We make note of the following estimate proved in~\cite{Cote15}, which is relevant for the vanishing of the error in the linear profile decomposition stated in the next section. 

\begin{lem}\emph{\cite[Lemma 2.11]{Cote15}} 
There exists  a uniform constant $C>0$ such that every solution $\bs v(t) \in \E$ to~\eqref{eq:lin-2d} satisfies, 
\EQ{
\| v \|_{L^\infty_{t, r}(\R)} \le C  \| \bs v(0) \|_{\E}^{\frac{3}{8}} \| r^{-\frac{3}{5}} v \|_{L^{5}_{t, r}(\R)}^{\frac{5}{8}}.
}
\end{lem}

\subsection{Profile decomposition}

Bahouri-G\'erard-type linear profile decompositions~\cite{BG} are an essential ingredient in the study of solutions to~\eqref{eq:wmk}; see also~\cite{BrezisCoron, Gerard, Lions1, Lions2,  MeVe98}. We make explicit use of a version adapted to sequences of functions in the affine spaces $\E_{\ell, m}$ proved by Jia and Kenig in~\cite{JK}, which synthesized C\^ote's analysis in~\cite{Cote15}; see also~\cite{CKLS1} which treats sequences in $\E_{0, 0}$. 

\begin{lem}[Linear profile decomposition]\emph{\cite[Lemma 5.5]{JK}\cite{BG}}  \label{lem:pd}  Let $\ell, m \in\Z$ and let $\bs u_n$ be a sequence in $\calE_{\ell, m}$ with $\limsup_{n \to \infty} E( \bs u_n) <\infty$. Then, there exists $K_0 \in \{0, 1, 2, \dots\}$, sequences $ \lam_{n, j} \in (0, \infty)$ for $j \in \{1, \dots, K_0\}$,  $\sigma_{n, i} \in (0, \infty)$, and  $t_{n, i} \in \R$, as well as mappings $\bs \psi^j \in \E_{\ell_j, m_j}$ with $E( \bs \psi^j) < \infty$ and finite energy solutions $\bs v\lin^{i}$ to~\eqref{eq:lin-2d} such that for each $J \ge 1$, 
\EQ{
\bs u_n &= m \bs \pi +  \sum_{j = 1}^{K_0} \Big( \psi^j \big(  \frac{ \cdot}{\lam_{n, j}} \big), \frac{1}{\lam_{n, j}} \dot \psi^j \big(\frac{ \cdot}{\lam_{n, j}} \big) \Big) - m_j \bs \pi)   \\
&\quad + \sum_{i =1}^J \Big( v_{\Lin}^i \big(  \frac{-t_{n, i}}{\s_{n, i}}, \frac{ \cdot}{\s_{n, i}} \big) , \frac{1}{\s_{n, i}} \p_t v_{\Lin}^i \big( \frac{-t_{n, i}}{\s_{n, i}}, \frac{ \cdot}{\s_{n, i}} \big)\Big) + \bs w_{n, 0}^J( \cdot)
}
where, denoting by $\bs w_{n, \Lin}^J(t)$ the solution to the linear wave equation~\eqref{eq:lin-2d} with initial data $\bs w_{n, 0}^J$, the following hold: 
\begin{itemize} 
\item  the parameters $\lam_{n, j}$ satisfy
\EQ{
\lam_{n, 1} \ll \lam_{n, 2} \ll \dots  \ll  \lam_{n, K_0} \mas n \to \infty; 
}
and for each $j$  one of  $\lam_{n, j} \to 0$, $\lam_{n, j} = 1$ for all $n$, or $\lam_{n, j} \to \infty$ as $n \to \infty$, holds;    
\item for each $i$,  either $t_{n, i} = 0$ for all $n$ or $\lim_{n\to \infty} \frac{-t_{n, i}}{\s_{n, i}}  = \pm \infty$. If $t_{n, i} = 0$ for all $n$, then one of $\s_{n, i} \to 0$, $\s_{n, i} = 1$ for all $n$, or $\s_{n, i} \to \infty$ as $n \to \infty$, holds; 
\item for each $i \in \N$, 
\EQ{
\frac{\lam_{n, j}}{\s_{n, i}} + \frac{\s_{n, i}}{\lam_{n, j}}  + \frac{\abs{t_{n, i}}}{\lam_{n, j}} \to \infty \mas n \to \infty \quad \forall j = 1, \dots, K_0; 
}
\item the scales $\s_{n, i}$ and times $t_{n, i}$ satisfy, 
\EQ{
\frac{\s_{n, i}}{\s_{n, i'}} + \frac{ \s_{n, i'}}{\s_{n, i}} + \frac{ \abs{ t_{n, i} - t_{n, i'}}}{\s_{n, i}} \to \infty \mas n \to \infty; 
}
\item the integers $\ell_j$ and $m_j$ satisfy, $\abs{\ell_j - m_j} \ge 1$, and, 
\EQ{
\ell = m - \sum_{j=1}^{K_0} ( \ell_j - m_j) ; 
}
\item the error term $\bs w_{n}^J$ satisfies, 
\EQ{
&( w_{n, 0}^J( \lam_{n, j} \cdot), \lam_{n, j} \dot w_{n, 0}^J( \lam_{n, j} \cdot)) \rightharpoonup 0 \in \E \mas n \to \infty\\
&( w_{n, \Lin}^J( t_{n, i}, \s_{n, i} \cdot), \s_{n, i} \p_t w_{n, \Lin}^J( t_{n, i}, \s_{n, i} \cdot)) \rightharpoonup 0\in \E \mas n \to \infty
}
for each $J \ge 1$, each $j = 1, \dots, K_0$, and $i \in \N$,  and vanishes strongly in the sense that 
\EQ{
\lim_{J \to \infty} \limsup_{n \to \infty} \Big(  \| w_{n, \Lin}^J \|_{L^\infty_{t, r}(\R)} + \| w_{n, \Lin} \|_{\calS(\R)} \Big)  = 0; 
}
\item the following  pythagorean decomposition of the nonlinear energy holds:  for each $J \ge 1$, 
\EQ{ \label{eq:pyth} 
E( \bs u_n) &= \sum_{j =1}^{K_0} E( \bs \psi^j)  + \sum_{i =1}^J E\big(  (v^j\lin (-t_{n,i}/ \s_{n, i}), \s_{n, i} \p_t v^j\lin  (-t_{n,i}/ \s_{n, i}))\big) + E( \bs w_{n}^J)  + o_n(1) 
} 
as $n \to \infty$. 
 
\end{itemize} 

\end{lem} 
\begin{rem}
The pythagorean expansion of the nonlinear energy in the case $K_0 =0$ was treated in~\cite[Lemma 2.16]{CKLS1}. The case with $K_0 \ge 1$ was treated in the recent preprint~\cite[Appendix B.2]{DKMM}.  
\end{rem} 

\begin{rem} 
We call the pairs $(\bs \psi^j, \lam_{n, j})$ and the triplets $(\bs v\lin^i, \s_{n, i}, t_{n, i})$ \emph{profiles}. Following Bahouri and G\'erard~\cite{BG} we refer to  the profiles $(\bs \psi^j, \lam_{n, j})$ and the profiles $(\bs v\lin^i, \s_{n, i}, 0)$ as \emph{centered}, to  the profiles $(\bs v\lin^i, \s_{n, i}, t_{n, i})$ with $- t_{n, i}/ \s_{n, i} \to \infty$ as $n \to \infty$ as \emph{outgoing},  and those with $- t_{n, i}/ \s_{n, i} \to -\infty$ as \emph{incoming}. 
\end{rem} 

In Section~\ref{sec:compact} we will need to evolve the linear profiles via the flow for~\eqref{eq:wmk} in the special case when all of the centered profiles are given by harmonic maps. In this setting we define \emph{nonlinear profiles} as follows.  Given a profile $(\bs v_{\Lin}^i, \s_{n, i}, t_{n, i})$ as in Lemma~\ref{lem:pd} we define the corresponding nonlinear profile, $( \bs v_{\nl}^i, \s_{n, i}, t_{n, i})$ as the unique solution to~\eqref{eq:wmk} such that for all $-t_{n, i}/ \s_{n, i} \in I_{\max}(\bs v_{\nl})$ we have, 
\EQ{
\lim_{n \to \infty} \| \bs v_{\nl}^i(-\frac{t_{n, i}}{\s_{n, i}} ) - \bs v_{\Lin}^i(-\frac{t_{n, i}}{\s_{n, i}} ) \|_{\E}  = 0
}
The existence of nonlinear profiles follows from the local Cauchy theory in Lemma~\ref{lem:Cauchy} in the case of a centered linear profile, i.e., $t_{n, i} = 0$, and from the existence of wave operators statement in Lemma~\ref{lem:Cauchy} in the case of outgoing/incoming profiles, i.e., $-t_{n, i}/ \s_{n, i} \to \pm \infty$. 

\begin{lem}[Nonlinear profile decomposition] \label{lem:nlpd} Let $\ell, m \in\Z$ and let $\bs u_n$ be a sequence in $\calE_{\ell, m}$ with $\limsup_{n \to \infty} E( \bs u_n) <\infty$. Assume the linear profile decomposition for $\bs u_n$ given by the Lemma~\ref{lem:pd} takes the form
\EQ{
\bs u_n  = m \bs \pi + \sum_{j = 1}^{K_0} ( Q \big(  \frac{ \cdot}{\lam_{n, j}} \big)  , 0) - \bs \pi)  + \sum_{i =1}^J \Big( v_{\Lin}^i \big(  \frac{-t_{n, i}}{\s_{n, i}}, \frac{ \cdot}{\s_{n, i}} \big) , \frac{1}{\s_{n, i}} \p_t v_{\Lin}^i \big( \frac{-t_{n, i}}{\s_{n, i}}, \frac{ \cdot}{\s_{n, i}} \big)\Big) + \bs w_{n, 0}^J( \cdot), 
}
that is, all of the profiles $( \bs \psi^j, \lam_{n, j})$ for $1 \le j \le K_0$ as in Lemma~\ref{lem:pd} are given by harmonic maps $(\bs Q, \lam_{n, j})$. 
There exists a constant $\de_0>0$ sufficiently small with the following properties. 
Let $i_0 \in\N$, $\tau_0>0$  and assume that for each $i  \in \N$, and for each $1 \le j \le K_0$, 
\EQ{ \label{eq:times} 
\frac{\tau_0 \s_{n, i_0} - t_{n, i}}{\s_{n, i}} &< T_{+, i}( \bs v_{\nl}^i), \quad    \limsup_{ n \to \infty} \| v_{\nl}^i \|_{\calS((- \frac{t_{n, i}}{ \s_{n, i}}, \frac{ \tau_0 \s_{n, i_0} - t_{n, i}}{\s_{n, i}}))} < \infty,\\
 &\mand  \tau_0 \frac{\s_{n,i_0} }{\lam_{n, j}}   \le  \de_0,
}
for all $n$. Then for  each $n$ sufficiently large, the wave map evolution $\bs u_n(t)$ of the data $\bs u_n(0) = \bs u_n$ is defined on the interval $[0, \tau_0 \s_{n, i_0}]$  and  the  following \emph{nonlinear profile decomposition} holds: for each $t \in [0, \tau_0 \s_{n, i_0}]$ the sequence $\bs z_{n}^J(t)$  defined by
\EQ{
\bs u_n(t) &=m \bs \pi + \sum_{j = 1}^{K_0} ( Q \big(  \frac{ \cdot}{\lam_{n, j}} \big)  , 0) - \bs \pi)  + \sum_{i =1}^J \Big( v_{\nl}^i \big(  \frac{t-t_{n, i}}{\s_{n, i}}, \frac{ \cdot}{\s_{n, i}} \big) , \frac{1}{\s_{n, i}} \p_t v_{\nl}^i \big( \frac{t-t_{n, i}}{\s_{n, i}}, \frac{ \cdot}{\s_{n, i}} \big)\Big) \\
&\quad + \bs w_{n, \Lin}^J(t) + \bs z_{n}^J(t), 
}
satisfies, 
\EQ{
\lim_{J \to \infty} \limsup_{n \to \infty} \Big( \sup_{t \in[0, \tau_0 \s_{n, i_0}]}\| \bs z_{n}^J(t) \|_{\E}  + \|  z_{n}^J\|_{\calS([0, \tau_0 \s_{n, i_0}])} \Big) = 0. 
}
\end{lem} 

\begin{rem} 
We note that the harmonic maps in the nonlinear profile decomposition are static solutions to~\eqref{eq:wmk}, but their presence in a linear profile decomposition may lead to eventual singularities in the nonlinear flow. This leads us to the hypothesis in the second line of~\eqref{eq:times}, which ensures that we are only considering the nonlinear evolution of $\bs u_n(t)$ on time intervals shorter than the length scales of the harmonic maps, thus avoiding the possibility of a singularity. 
\end{rem} 


The key ingredient in the proof of Lemma~\ref{lem:nlpd} is the following  modification of the now standard nonlinear perturbation lemma~\cite[Theorem 2.20]{KM08}; see also~\cite[Section IV]{BG}. 

\begin{lem}[Nonlinear perturbation lemma] \label{lem:nl-pert}  Fix integers $\ell, m$. There are continuous functions $\eps_0, C_0: (0, \infty) \to (0, \infty)$ with the following properties. Let  $I$ be an open interval and let $\bs u, \bs v \in C^0(I;  \E_{\ell, m})$ such that for some $A \ge 0$, 
\EQ{
\| \bs u - \bs v \|_{L^\infty_t(I; \E)} + \| r^{-\frac{2}{3}}\sin v \|_{L^3_t(I; L^6_r)} \le A
}
and 
\EQ{ 
\| \glei(u)   \|_{L^1_t(I; L^2_r)} + \| \glei(v) \|_{L^1_t(I; L^2_r)} + \| w_0 \|_{\calS(I)}  \le \eps \le \eps_0(A)
}
where $\glei(u) := \p_t^2 u - \De u + k^2 r^{-2} f(u)$ in the sense of distributions, and $\bs w_0(t):= S(t-t_0) (\bs u(t_0) - \bs v(t_0))$ is the linear evolution of the difference, i.e., the solution to~\eqref{eq:lin-2d}, where $t_0 \in I$ is arbitrary, but fixed. Then, 
\EQ{
\| \bs u(t) - \bs v(t) - \bs w_0(t) \|_{L^\infty(I; \E)} + \| u - v \|_{\calS(I)}  \le C_0(A) \eps. 
}
\end{lem} 

\begin{proof}[Proof of Lemma~\ref{lem:nl-pert}] 
Let $X(I)$ denote the space $L^3_t(I;  L^6_r)$ in this proof. Define 
$\bs w(t):= \bs u(t) - \bs v(t)$ and let $ e:=  \p_t^2 u - \De u + k^2 r^{-2} \sin(u)\cos(u)-  (\p_t^2 v - \De v + k^2 r^{-2} \sin(v)\cos(v)) = \glei(u) - \glei(v)$. 
Let $t_0 \in I$, fix a small constant $\de_0$ to be determined below and partition the right-half of $I$ as follows, 
\EQ{
&t_0 < t_1 < t_2 < \dots < t_n \le \infty, \quad I_j := (t_j, t_{j+1}) , \quad I \cap (t_0, \infty) = (t_0, t_n),  \\
& \| r^{-\frac{2}{3}}\sin v \|_{L^3_t( I_j; L^6_r)} \le \de_0 \mfor j = 0, \dots, n-1, \mand  n \le C(A; \de_0). 
}
We omit the estimate on $I\cap(-\I,t_0)$ since it is the same by symmetry. 
Let $\bs w_j(t):=S\lin(t-t_j)\bs w(t_j)$, where $S\lin$ is the linear propagator for~\eqref{eq:lin-2d}, for  all $0\le j <n$.  Then
\EQ{\label{eq:w-Duhamel}
\bs w(t)  
&= \bs w_{0}(t) + \int_{t_0}^{t} S\lin(t-s) \big(0,e -k^2 r^{-2} (f(v+w) - f(v) -w)\big)(s)\, \ud s
}
which implies that, for some absolute constant $C_{1}\ge1$, 
\EQ{ \label{eq:ww0}
\pn \| w-w_{0}\|_{X(I_0)} 
 \pt\lec  \|e -k^2 r^{-2} (f(v+w) - f(v) -w)\|_{L^1_tL^2_r(I_0)}
 \pr\le C_{1}\eps + C_1(\de_{0}^{2}+\| w\|_{X(I_0)}^{2})\| w\|_{X(I_0)}
}
In the second estimate above we have used the expansion, 
\EQ{
f(v+w) - f(v) -w &= \frac{1}{2}( \sin(2v + 2w) - \sin2v - w) \\
& =- w \sin^2 v   -2 \sin v \cos v \sin^2 w + O( \abs{w}^3), 
}
to estimate the terms on the right. Note that in~\eqref{eq:ww0} we are using in an essential way the divisibility of the $X(I)$ norm. 
Note that $\|w\|_{X(I_{0})}<\I$ provided $I_{0}$ is a finite interval. If $I_{0}$ is half-infinite, then we first
need to replace it with an interval of the form $[t_{0},N)$, and  let $N\to\I$ after performing  estimates which are
uniform in~$N$.  Now assume that $C_{1}\delta_{0}^{2}\le \frac{1}{4}$ and fix $\delta_{0}$ in this fashion.
By means of the continuity method,  
\eqref{eq:ww0} implies that $\| w\|_{X(I_{0})}\le 8C_{1}\eps$. 
Next, Duhamel's formula gives
\EQ{
\vec w_{1}(t)- \vec w_{0}(t) = \int_{t_{0}}^{t_{1}} S\lin(t-s) \big(0,e -k^2 r^{-2} (f(v+w) - f(v) -w)\big)(s)\, \ud s
}
from which we obtain 
\EQ{\label{eq:w1w0}
\| w_{1}-w_{0}\|_{X(\R)} \lesssim \int_{t_0}^{t_1} \| \big(e -k^2 r^{-2} (f(v+w) - f(v) -w)\big)(s)\|_{2}\, \ud s
}
which is estimated as in~\eqref{eq:ww0}.  We conclude that $ \| w_{1}\|_{X(\R)}\le 8C_{1}\eps$. 
In a similar fashion one verifies that for all $0\le j<n$ 
\EQ{ \label{eq:estS}
 \pn\| w- w_j\|_{X(I_j)} + \| w_{j+1}-w_j\|_{X(\R)}
 \pt\lec  \|  e -k^2 r^{-2} (f(v+w) - f(v) -w)  \|_{L^1_tL^2_r(I_j)}
  \pr\le C_{1}\eps + C_1(\de_{0}^{2}+\| w\|_{X(I_j)}^{2})\| w\|_{X(I_j)}
  }
 where $C_{1}\ge1$ is as above. By induction in $j$ we have 
 \EQ{
 \| w\|_{X(I_{j})} + \| w_{j}\|_{X(\R)} \le C(j)\, \eps\quad \forall \; 1\le j<n. 
 }
 This requires that  $\eps<\eps_{0}(n)$ which can be achieved as long as  $\eps_0(A)$ is chosen small enough.
Repeating the estimate~\eqref{eq:estS},  but with the full $\calS(I)$ norm and the  energy piece $L^{\infty}_{t}\E$ included on the left-hand side completes the proof.  
\end{proof}

\begin{proof}[Sketch of the proof of Lemma~\ref{lem:nlpd}]
The proof is very similar to~\cite[Proof of Proposition 2.8]{DKM1} or~\cite[Proof of Proposition 2.17]{CKLS1} and we give a brief sketch below, mainly to address how the nonlinear profiles given by harmonic maps are handled. 

Let $I_{n} = [0, \tau_n)  \subset [0, \tau_0 \s_{n, i_0}]$ be any half-open subinterval on which the wave map evolution $\bs u_n(t)$ is defined. By~\eqref{eq:times}, the sequence 
\EQ{
\bs v_{n}^J(t) := m \bs \pi + \sum_{j = 1}^{K_0} ( Q \big(  \frac{ \cdot}{\lam_{n, j}} \big)  , 0) - \bs \pi)  + \sum_{i =1}^J \Big( v_{\nl}^i \big(  \frac{t-t_{n, i}}{\s_{n, i}}, \frac{ \cdot}{\s_{n, i}} \big) , \frac{1}{\s_{n, i}} \p_t v_{\nl}^i \big( \frac{t-t_{n, i}}{\s_{n, i}}, \frac{ \cdot}{\s_{n, i}} \big)\Big)
} 
is well defined on the time intervals $ [0, \tau_{0} \s_{n, i_0}]$. The idea is to apply Lemma~\ref{lem:nl-pert} to the sequences $\bs u_n$ and $\bs v_{n}^J$ on $I_n$ for large $n$ and so we need to check that the hypothesis of Lemma~\ref{lem:nl-pert} are satisfied. First, $\bs u_n(t)$ solves~\eqref{eq:wmk} so $\glei(u_n) = 0$. Next we claim that 
\EQ{ \label{eq:rhs-vanishing} 
 \lim_{n \to \infty} \| \glei(v_{n}^J) \|_{L^1_t L^2_{r}([0, \tau_0 \s_{n, i_0}])}  = 0. 
} 
for any \emph{fixed} $J$. Denoting $v_{\nl, n}^i(t):= v_{\nl}^i \big(  \frac{t-t_{n, i}}{\s_{n, i}}, \frac{ \cdot}{\s_{n, i}} \big)$ we have,  
\EQ{
\abs{\glei(u_n^J)(t)} &= \frac{k^2}{2r^2} \Big|\sin( 2v_{n}^J(t)) - \sum_{j = 1}^{K_0} \sin 2Q_{\lam_{n, j}} - \sum_{i=1}^J \sin 2 (v_{\nl, n}^i(t)) \Big| 
}
And hence~\eqref{eq:rhs-vanishing} follows from an argument based the pseudo-orthogonality of the parameters, the hypothesis~\eqref{eq:times},  and repeated use of the identity, 
\EQ{
\sin(A+B) - \sin A - \sin B = -2 \sin A \sin^2 B - 2 \sin B \sin^2 A. 
}
Next,  note  that the last condition in~\eqref{eq:times} implies that 
\EQ{
\lim\sup_{n \to \infty} \| r^{-\frac{2}{3}} \sin Q_{\lam_{n, j}} \|_{L^3_t([0, \tau_0 \s_{n, i_0}]; L^6_r)} \lesssim 1, 
}
for each $j \in 1, \dots, K_0$. 
In fact,  is crucial that, 
\EQ{
\limsup_{n \to \infty} \| r^{-\frac{2}{3}}\sin( v_{n}^J)\|_{L^3_t L^6_r(I_{n})} \lesssim 1
}
uniformly in $J$. This is possible thanks to the small data theory from Lemma~\ref{lem:Cauchy} together with the pythagorean expansion of the energy~\eqref{eq:pyth}. Indeed, there exists $J_1$ such that for each $i \ge J_1$, we must have $\limsup_{n \to \infty}E( \bs v_{\Lin, n}^j(0))  < \eps_0$ where $\eps_0$ is as in Lemma~\ref{lem:Cauchy} and $v_{\Lin, n}^i(t):= v_{\Lin}^i \big(  \frac{t-t_{n, i}}{\s_{n, i}}, \frac{ \cdot}{\s_{n, i}} \big)$. Using again the pseudo-orthogonality of the parameters and~\eqref{eq:times}  along with Lemma~\ref{lem:Cauchy} we obtain, 
\EQ{
\limsup_{n \to \infty} \|r^{-\frac{2}{3}}  \sin( \sum_{i \ge J_1} v_{\nl, n}^i) \|_{L^3_t L^6_r(I_{n})}^3 &\lesssim \limsup_{n \to \infty} \sum_{i \ge J_1} \| r^{-\frac{2}{3}}  v_{\nl, n}^i \|_{L^3_t L^6_r(I_{n})}^3 \\
&\lesssim \limsup_{n \to \infty}  \sum_{i \ge J_1} \|  v_{\Lin, n}^i \|_{\E}^3  < \infty, 
}
where the last inequality implicitly uses the fact that for all $\bs v \in \E$ with  $E( \bs v) \le \eps$ sufficiently small we have $\| \bs v \|_{\E} \simeq E( \bs v)$. One may now apply Lemma~\ref{lem:nl-pert} and conclude, for instance, that 
\EQ{ \label{eq:un-vn-wn} 
\limsup_{n \to \infty} \sup_{t \in I_n}\| \bs u_n(t) - \bs v_n^J(t) - \bs w_{\Lin, n}^J(t) \|_{\E}   = 0, 
} 
for each interval $I_n \subset [0, \tau_0 \s_{i_0, n}]$ on which $\bs u_n(t)$ is defined. In fact, by Lemma~\ref{lem:lwp} this is sufficient  to deduce that $T_+( \bs u_n) > \tau_0 \s_{i_0, n}$ for all sufficiently large $n$ as long as $\de_0$ as in~\eqref{eq:times} is chosen small enough. To see this, suppose for contradiction there is some subsequence $\bs u_n(t)$ and a sequence $\tau_n \to 0$ for which $\bs u_n(t)$ has maximal forward interval existence given by $I_n = [0, \tau_n \s_{n, i_0})$. Fix $J>i_0$ and let $\eps>0$ be a constant to be determined below. Since each of the profiles is well-defined up till time $\tau_0 \s_{n, i_0}$, and using crucially the second line in~\eqref{eq:times} (in particular  that $\lam_{n, j} \gtrsim \s_{n, i_0}$ for each $j$), we can find $A_n = A_n( \eps)>0$ such that 
\EQ{
\sum_{j =1}^{K_0} E( Q_{\lam_{n, j}}; 0, A_n)  + \sum_{i=1} E( \bs v_{\nl, n}^i( \tau_n \s_{n, i_0}); 0, A_n) + E( \bs w_{ \Lin, n}^J( \tau_n \s_{n, i_0}); 0, A_n) < \eps 
}
and such that $s_n:= \tau_n \s_{n, i_0} - \frac{A_n}{4} >0$. By finite speed of propagation and the above  we have 
\EQ{
\sum_{j =1}^{K_0} E( Q_{\lam_{n, j}}; 0, A_n/2)  + \sum_{i=1} E( \bs v_{\nl, n}^i( s_n); 0, A_n/2) + E( \bs w_{ \Lin, n}^J(s_n); 0, A_n/2) < \eps 
}
Combing the above with~\eqref{eq:un-vn-wn}, we obtain, 
\EQ{
E( \bs u_n( s_n); 0, A_n/2)  \lesssim \eps
}
as long as $n$ is taken sufficiently large. Since $\tau_n\s_{n, i_0} - s_n = A_n/4 < A_n/2$, we see by Lemma~\ref{lem:lwp} that $\tau_n \s_{n, i_0}$ cannot be a maximal time for $\bs u_n$ as long as $\eps>0$ is chosen small enough, a contradiction. This completes the proof. 
\end{proof}

In Section~\ref{sec:compact} we need an additional fact about profile decompositions satisfying additional hypothesis proved in~\cite{DKM3erratum}. First, a preliminary lemma. 

\begin{lem} \emph{\cite[Claim 2]{DKM3erratum}} \label{lem:sym}
Let $(f_n, g_n) \in \E$ be a sequence of functions, bounded in $\E$ and assume that there exists a sequence $\al_n>0$ of positive numbers such that 
\EQ{
\| g_n \|_{L^2(r \ge \al_n)}  \to 0 \mas n \to \infty
}
Let $\{s_n\} \subset \R$ be any sequence such that $\lim_{n \to \infty} \abs{s_n}/ \al_n = \infty$,  and denote by $\bs v_0 = ( v_0, \dot v_0) \in \E$ the following weak limit, 
\EQ{
S\lin( -s_n) (f_n, g_n)   \rightharpoonup \bs v_0  \in \E. 
}
Then, $S\lin( s_n) (f_n, g_n)   \rightharpoonup \bs (v_0, - \dot v_0)  \in \E$.
\end{lem} 

As a consequence one has the following lemma.
 
\begin{lem}\emph{\cite[Claim 3]{DKM3erratum}} \label{lem:sym-profile}  Let $\ell, m \in\Z$ and let $\bs u_n$ be a sequence in $\calE_{\ell, m}$ with $\limsup_{n \to \infty} E( \bs u_n) <\infty$. Assume the sequence $\bs u_n$ admits a profile decomposition of the form, 
\EQ{
\bs u_n  = m \bs \pi + \sum_{j = 1}^{K_0} ( Q \big(  \frac{ \cdot}{\lam_{n, j}} \big)  , 0) - \bs \pi)  + \sum_{i =1}^J \Big( v_{\Lin}^i \big(  \frac{-t_{n, i}}{\s_{n, i}}, \frac{ \cdot}{\s_{n, i}} \big) , \frac{1}{\s_{n, i}} \p_t v_{\Lin}^i \big( \frac{-t_{n, i}}{\s_{n, i}}, \frac{ \cdot}{\s_{n, i}} \big)\Big) + \bs w_{n, 0}^J( \cdot), 
}
that is, all of the profiles $( \bs \psi^j, \lam_{n, j})$ for $1 \le j \le K_0$ as in Lemma~\ref{lem:pd} are given by harmonic maps $(\bs Q, \lam_{n, j})$. Assume in addition that, 
\EQ{
\| \dot u_n \|_{L^2} \to 0 \mas n \to \infty. 
}
Then, after passing to a subsequence,  for each profile $( \bs v_{\Lin}^i, \s_{n, i}, t_{n, i})$ we can ensure that either, 
\EQ{
t_{n, i} = 0\, \,  \forall \, \, n \mand \dot v_{\Lin}^i(0) = 0 
}
or, 
\EQ{
-\frac{t_{n, i}}{\s_{n, i}} \to \pm \infty \mand \exists i' \neq i \, \, \textrm{such that} \,\, v_{\Lin}^i(t) = v_{\Lin}^{i'}(-t) \, \,  \forall t, \, \, t_{n, i} = - t_{n, i'} \, \, \s_{n, i} = - \s_{n, i'} \, \, \forall n. 
}
\end{lem}

\subsection{Multi-bubble configurations}
In this section we study properties of finite energy maps near a multi-bubble configuration.

The operator $\LL_{\calQ}$ obtained by linearization of~\eqref{eq:wmk} about an $M$-bubble configuration $\bs \calQ(m, \vec \iota, \vec \lam)$ is given by, 
\EQ{  \label{eq:LQ-def} 
\LL_{\calQ} \, g := \uD^2 E_{\bfp}(\calQ(m, \vec\iota, \vec \lam)) g = - \De g + \frac{k^2}{r^2} f'(\calQ(m, \vec\iota, \vec \lam) )g, 
}
where $f'( z) = \cos 2 z$. Given $\bs g = (g,\dot g) \in \E$, 
\EQ{
\La \uD^2 E(\bs\calQ(m, \vec \iota, \vec \lam)) \bs g \mid \bs g \Ra =  \int_0^\infty \Big(\dot g(r)^2 + (\p_r g(r))^2 + \frac{k^2}{r^2} f'(\calQ(m, \vec \iota, \vec \lam)) g(r)^2 \, \Big) r \ud r. 
}
An important instance of the operator $\LL_{\calQ}$ is given by linearizing \eqref{eq:wmk} about a single harmonic map $ \calQ(m, M, \vec\iota, \vec \lam) = Q_{\lam}$. In this case we use the short-hand notation, 
\EQ{ \label{eq:LL-def} 
\LL_{\lam} := (-\De + \frac{k^2}{r^2}) + \frac{1}{r^2} ( f'(Q_{\lam}) - k^2) 
}
We write $\calL := \calL_1$.
For each $k \ge 1$,  
\EQ{
\Lam Q(r):= r \p_r Q(r)  = k \sin Q = 2 k  \frac{ r^k}{1 + r^{2k}}
}
 When $k \ge 2$, $\Lam Q$ is a zero energy eigenfunction for $\calL$, i.e.,  
\EQ{
\calL \Lam Q = 0, \mand \Lambda Q  \in L^2_{\textrm{rad}}(\R^2).
}
 When $k=1$, $\calL \Lam Q = 0$ holds but  $\Lam Q \not \in L^2$ due to slow decay as $r \to \infty$  and $0$ is called a threshold resonance. 
 Indeed, for $R>0$, 
\EQ{ \label{eq:LamQL2} 
	\int_0^R (\Lambda Q(r))^2 \, r\, \ud r  =  -\frac{2 R^2}{1 + R^2} +2 \log(1+ R^2)  = 4 \log R + O(1) \mas R \to \infty.
}
On the other hand when $k=1$,  $\U\Lam \Lam Q$ has an important cancellation which leads to improved decay, 
\EQ{ \label{eq:Lam0LamQ} 
	\ULam \Lam Q = \frac{4r}{(1+ r^2)^2} , 
}
so $\ULam \Lam Q \in L^1 \cap L^\infty$ and $\ang{ \ULam \Lam Q \mid \Lam Q} = 2$, whereas for $k \ge 2$, $\ang{ \ULam \Lam Q \mid \Lam Q} = 0$. 
 
We define a smooth non-negative function $\calZ \in C^{\infty}(0, \infty) \cap L^1((0, \infty), r\, \ud r)$ by  
 \EQ{ \label{eq:Z-def} 
 \calZ(r) := \begin{cases}   \chi(r) \Lam Q(r) \mif k =1, 2 \\ \Lam Q(r)  \mif k \ge 3 \end{cases}
 }
and note that  
 \EQ{
 \ang{ \calZ \mid \Lam Q} >0  \label{eq:ZQ} . 
 }
 In fact the precise form of $\calZ$ is not so important, rather only that it is not perpendicular to $\Lam Q$ and has sufficient decay and regularity. We fix it as above because of the convenience of setting $\calZ = \Lam Q$  if $k\ge 3$. 
We record the following localized coercivity lemma proved in~\cite{JJ-AJM}.

\begin{lem}[Localized coercivity for $\LL$] \emph{\cite[Lemma 5.4]{JJ-AJM}}  \label{l:loc-coerce} 
Fix $k \ge 1$.  There exist uniform constants $c< 1/2, C>0$ with the following properties. Let $g \in H$. Then, 
\EQ{ \label{eq:L-coerce}
\ang{ \LL g \mid g} \ge c  \| g \|_{H}^2  - C\ang{ \calZ \mid g}^2 
}
If $R>0$ is large enough then,  
\EQ{ \label{eq:L-loc-R} 
(1-2c)&\int_0^{R} \Big((\p_r g)^2 + k^2 \frac{g^2}{r^2} \Big) \, r \ud r +  c \int_{R}^\infty \Big((\p_r g)^2 + k^2 \frac{g^2}{r^2} \Big) \, r \ud r  - \La\frac{ k^2}{r^2}(f'(Q) - 1) g\mid g\Ra \\
& \ge  - C\ang{ \calZ \mid g}^2. 
}
If $r>0$ is small enough, then
\EQ{ \label{eq:L-loc-r} 
(1-2c)&\int_r^{\infty} \Big((\p_r g)^2 + k^2 \frac{g^2}{r^2} \Big) \, r \ud r +  c \int_{0}^r \Big((\p_r g)^2 + k^2 \frac{g^2}{r^2} \Big) \, r \ud r  - \La\frac{ k^2}{r^2}(f'(Q) - 1) g\mid g\Ra \\
& \ge  - C\ang{ \calZ \mid g}^2. 
}
\end{lem} 

As a consequence, (see for example~\cite[Proof of Lemma 2.4]{JKL1} for an analogous argument) one obtains the following coercivity property of the operator $\LL_{\calQ}$. 

\begin{lem} \label{lem:D2E-coerce}  Fix $k \ge 1$, $M \in \N$. There exist $\eta, c_0>0$ with the following properties. Consider the subset of $M$-bubble configurations $\bs \calQ(m, \vec\iota, \vec \lam)$ for $\vec \iota \in \{-1, 1\}^M$, $\vec \lam \in (0, \infty)^M$ such that, 
\EQ{ \label{eq:lam-ratio} 
\sum_{j =1}^{M-1} \Big( \frac{\lam_j}{\lam_{j+1}} \Big)^k \le \eta^2. 
}
Let $g \in H$ be such that 
\EQ{
0 = \ang{ \calZ_{\U{\lam_j}} \mid g}  \mfor j = 1, \dots M. 
}
for some $\vec \lam$ as in~\eqref{eq:lam-ratio}. Then, 
\EQ{
\ang{ \uD^2 E_{\bfp}( \calQ( m, \vec \iota, \vec \lam)) g \mid g} \ge c_0 \| g \|_{H}^2. 
} 
\end{lem} 

The following technical lemma is useful when computing interactions between bubbles at different scales. 
\begin{lem}
\label{lem:cross-term}
For any $\lambda \leq \mu$ and $\alpha, \beta > 0$ with $\alpha \neq \beta$ the following bound holds:
\begin{equation}
\int_0^{\infty} \Big[\Big(\frac{r}{\lambda}\Big)_{\geq 1}\Big]^{-\alpha}\Big[\Big(\frac{\mu}{r}\Big)_{\geq 1}\Big]^{-\beta} \frac{\vd r}{r}
\lesssim_{\alpha, \beta} \Big(\frac{\lambda}{\mu}\Big)^{\min(\alpha, \beta)}.
\end{equation}
For any $\alpha > 0$ the following bound holds:
\begin{equation}
\int_0^\infty\Big[\Big(\frac{r}{\lambda}\Big)_{\geq 1}\Big]^{-\alpha}\Big[\Big(\frac{\mu}{r}\Big)_{\geq 1}\Big]^{-\alpha} \frac{\vd r}{r}
\lesssim_{\alpha} \Big(\frac{\lambda}{\mu}\Big)^{\alpha}\log\Big(\frac{\mu}{\lambda}\Big).
\end{equation}
\end{lem}
\begin{proof}
This is a straightforward computation, considering separately the regions $0 < r \leq \lambda$, $\lambda \leq r \leq \mu$, and $r \geq \mu$.
\end{proof}

Using the above, along with the formula for $\calZ$ in~\eqref{eq:Z-def} we obtain the following. 
\begin{cor}  \label{cor:ZQ} 
Let $\calZ$ be as in~\eqref{eq:Z-def} and suppose that $\lam, \mu>0$ satisfy $\lam/ \mu \le 1$. Then, 
\EQ{
\ang{ \calZ_{\U \lam} \mid \Lam Q_{\U \mu}}  \lesssim  \begin{cases} (\lam/\mu)^{k+1} \mif k=1, 2 \\ (\lam/ \mu)^{k-1} \mif k \ge 3 \end{cases} ,  \quad \ang{ \calZ_{\U \mu} \mid \Lam Q_{\U \lam}}  \lesssim \begin{cases} 1 \mif k =1 \\  (\lam/\mu)^{k-1} \mif  k \ge 2\end{cases} 
}
\end{cor} 

Another use of Lemma~\ref{lem:cross-term} is to extract the leading order terms in a Taylor expansion of the nonlinear energy functional about an $M$-bubble configuration.

\begin{lem}  \label{lem:M-bub-energy} Fix $k\ge1,  M \in \N$. 
For any $\te>0$, there exists $\eta>0$ with the following property. Consider the subset of $M$-bubble $\bs \calQ(m,\iota, \vec \lam)$ configurations 
such that 
\EQ{
\sum_{j =1}^{M-1} \Big( \frac{ \lam_{j}}{\lam_{j+1}} \Big)^k \le \eta.
}
Then, 
\EQ{
  \Big|  E( \bs\calQ( m, \vec \iota, \vec \lam))  - M E( \bs Q) -  16 k \pi \sum_{j =1}^{M-1} \iota_j \iota_{j+1}  \Big( \frac{ \lam_{j}}{\lam_{j+1}} \Big)^k  \Big| \le \te \sum_{j =1}^{M-1} \Big( \frac{ \lam_{j}}{\lam_{j+1}} \Big)^k .
}
Moreover, there exists a uniform constant $C>0$ such that for any $g \in H$, 
\EQ{
\abs{\ang{ D E_{\bfp}( \calQ(m, \vec \iota, \vec \lam)) \mid g} } \le C \| g \|_{H} \sum_{j =1}^M \Big( \frac{\lam_{j}}{\lam_{j+1}} \Big)^k . 
}
\end{lem} 

\begin{proof} 
This is an explicit computation; see~\cite[Proof of Lemma 3.1, p. 1283-1286]{JL1} where the leading order term in computed in the case of two bubbles. The error terms are computed using Lemma~\ref{lem:cross-term}.
\end{proof}

%
%

The following modulation lemma plays an important role in our analysis. Before stating it, we define a proximity function to $M$-bubble configurations. Fixing $m, M$ we observe that $\bs\calQ(m, \vec\iota, \vec\lambda; r)$ is an element of $\E_{\ell, m}$, where 
\EQ{ \label{eq:ell-def} 
\ell = \ell(m, M, \vec \iota): = m - \sum_{j=1}^M  \iota_j
}

 \begin{defn} \label{def-d} Fix $m, M$ as in Definition~\ref{def:multi-bubble} and let $\bs v \in \E_{\ell, m}$ for some $\ell \in \Z$.  Define, 
\EQ{ \label{eq:d-def} 
\bfd( \bs v) = \bfd_{ m, M}( \bs v) := \inf_{\vec \iota, \vec \lam}  \Big( \| \bs v - \bs \calQ( m, \vec \iota, \vec \lam) \|_{\E}^2 + \sum_{j =1}^{M-1} \Big( \frac{\lam_{j}}{\lam_{j+1}} \Big)^k \Big)^{\frac{1}{2}}.
}
where the infimum is taken over all vectors $\vec \lam = (\lam_1, \dots, \lam_M) \in (0, \infty)^M$ and all $\vec \iota = \{ \iota_1, \dots, \iota_M\} \in \{-1, 1\}^M$ satisfying~\eqref{eq:ell-def}. 
\end{defn}


\begin{lem}[Static modulation lemma] \label{lem:mod-static} Fix $k \ge 1$ and $M \in \N$. 
There exists $\eta, C>0$ with the following properties.  
Let  $m$ be as in Definition~\ref{def:multi-bubble} and $\bfd_{m, M}$ as in Definition~\ref{def-d}. Let $\te>0$, $\ell \in \Z$,  and let $\bs v \in  \calE_{\ell,  m}$   be such that 
\EQ{ \label{eq:v-M-bub} 
\bfd_{ m, M}( \bs v)  \le \eta, \mand E( \bs v) \le ME( \bs Q) + \te^2, 
}
Then, there exists a unique choice of $\vec \lam = ( \lam_1, \dots, \lam_M) \in  (0, \infty)^M$, $\vec\iota \in \{-1, 1\}^M$, and $g \in  H$, such that setting $\bs g = (g, \dot v)$, we have 
\EQ{ \label{eq:v-decomp} 
   \bs v &=  \bs \calQ( m, \vec \iota, \vec \lam) + \bs g, \\
   0 & = \ang{ \calZ_{\U{\lam_j}} \mid g} , \quad \forall j = 1, \dots, M,
   }
   along with the estimates, 
\EQ{  \label{eq:g-bound-0}
\bfd_{ m, M}( \bs v)^2 &\le \|  \bs g \|_{\E}^2  + \sum_{j =1}^{M-1} \Big( \frac{\lam_{j}}{\lam_{j+1}} \Big)^k  \le C \bfd_{ m, M}( \bs v)^2,
}
and, 
\EQ{\label{eq:g-bound-A} 
   \|  \bs g \|_{\E}^2 + \sum_{j \not \in \calA}   \Big( \frac{ \lam_j}{ \lam_{j+1} }\Big)^k & \le C  \max_{ j \in \calA} \Big( \frac{ \lam_j}{ \lam_{j+1} }\Big)^k + \te^2, 
}
where $\calA  := \{ j \in \{ 1, \dots, M-1\} \, : \, \iota_j  \neq \iota_{j+1} \}$. 
\end{lem}

\begin{rem} \label{rem:IFT}  We  use the following, less standard, version of the implicit function theorem in the proof of Lemma~\ref{lem:mod-static}. 

 \emph{
Let $X, Y, Z$ be Banach spaces,  $(x_0, y_0) \in X \times Y$,  and $\de_1, \de_2>0$. Consider a mapping  $G: B(x_0, \de_1) \times B(y_0, \de_2) \to Z$,  continuous in $x$ and $C^1$ in $y$. Assume $G(x_0, y_0) = 0$,  $(D_y G)(x_0, y_0)=: L_0$ has bounded inverse $L_0^{-1}$, and  
\EQ{ \label{eq:IFT-cond} 
&\| L_0 - D_y G(x, y) \|_{\LL(Y, Z)} \le \frac{1}{3 \| L_0^{-1} \|_{\LL(Z, Y)}}  \\ 
& \| G(x, y_0) \|_Z \le \frac{\de_2}{ 3 \| L_0^{-1} \|_{\LL(Z, Y)}},  
}
for all $\| x - x_0 \|_{X} \le \de_1$ and $\| y - y_0 \|_Y \le \de_2$. 
Then, there exists a continuous function $\si: B(x_0, \de_1)  \to B(y_0, \de_2)$ such that for all $x \in B(x_0, \de_1)$, $y = \si(x)$ is the unique solution of $G(x, \si(x)) = 0$ in $B(y_0, \de_2)$. 
}

This is proved in the same way as the usual implicit function theorem, see, e.g.,~\cite[Section 2.2]{ChowHale}. The essential point is that the bounds~\eqref{eq:IFT-cond} give uniform control on the size of the open set where the Banach contraction mapping theorem is applied. 
\end{rem} 

\begin{proof}[Proof of Lemma~\ref{lem:mod-static}] The argument is very similar to~\cite[Proof of Lemma 3.1]{JL1} and we only give a brief sketch. Let $\eta_0 := \bfd_{m, M}( \bs v)$. 
By~\eqref{eq:v-M-bub} there exists some choice of $ \vec{ \iota} \in \{-1, 1\}^M$ and $\vec{\ti \lam} \in (0, \infty)^M$ such that 
\EQ{
\ti g :=  v -  \calQ(m, \vec{\iota}, \vec{\ti \lam})\, \,\, \,   \textrm{satisfies} \, \, \, \, 
\eta_0^2 \le \| \ti g \|_{H}^2 + \sum_{ j =1}^{M-1} \Big( \frac{\ti \lam_j}{\ti \lam_{j+1}} \Big)^k   \le 4 \eta_0^2
}
Define $F:H \times (0, \infty)^M \to H$, by  
\EQ{
F(g, \vec \lam) := g + \calQ(m, \vec{\iota}, \vec{\ti \lam}) - \calQ(m, \vec{\iota}, \vec{ \lam})
}
Note that, $F(0, \vec{\ti \lam}) = 0$ and 
\EQ{ \label{eq:F-cont} 
\| F(g, \vec \lam) \|_{H} \le \| g \|_H + \sum_{j=1}^M\Big| \frac{\lam_j}{\ti \lam_j} - 1\Big| 
}
Next, define $G: H \times (0, \infty)^M \to \R^M$ by, 
\EQ{
\vec G( g, \vec \lam) := \Big( \frac{1}{\lam_1} \ang{ \calZ_{\U{\lam_1}} \mid F( g, \vec \lam)}, \dots, \frac{1}{\lam_M} \ang{ \calZ_{\U{\lam_M}} \mid F( g, \vec \lam)} \Big) 
}
 note that $\vec G(0, \vec{\ti \lam}) = \vec 0$, and we record the computation,  
\EQ{ \label{eq:partial-G} 
\lam_j \p_{\lam_j} G_j( g, \vec \lam) &=  - \frac{1}{\lam_j} \ang{ [(\ULam+1)\calZ]_{\U{\lam_j}} \mid F(g, \vec \lam)} - \iota_j \ang{ \calZ \mid \Lam Q} \\
\lam_i \p_{\lam_i} G_j(g, \vec \lam)& = - \iota_i \frac{ \lam_i}{\lam_j} \ang{ \calZ_{\U{\lam_j}} \mid \Lam Q_{\U{\lam_i}}} \mif i \neq j
}
 At this point, it is convenient to change variables, letting $\ell_j:= \log \lam_j$ and 
$
\ti G( g, \vec \ell) = G( g, \vec \lam)
$. 
Note that  $\p_{\ell_j} = \lam_j \p_{\lam_j}$. From~\eqref{eq:F-cont} we see that $\ti G( \cdot, \cdot)$ is continuous near $0 \in H$ in the first slot and is $C^1$ near $\vec{\ti \ell}= (\log \ti \lam_1, \dots, \log  \ti \lam_M)$ in the last $M$ variables.  We compute, 
\EQ{  
L_0 := \uD_{\ell_1, \dots \ell_M} \ti G(g, \vec \ell)  \rest_{g =0, \vec \ell = \vec{\ti \ell}} \, = ( A_{i j})_{1 \le i, j \le M} 
}
where $(A_{ij})$ is the $M\times M$ matrix with entries, 
\EQ{\label{eq:L_0-def}
A_{jj} &= - \iota_j \ang{ \calZ \mid \Lam Q }, \quad A_{ij} = -\iota_j \frac{\lam_j}{\lam_i} \ang{ \calZ_{\U{\lam_i}} \mid \Lam Q_{\U {\lam_j}}} \mif i \neq j
}
which one may check, using~\eqref{eq:ZQ} and Corollary~\ref{cor:ZQ}  is invertible and $\| L_0\|^{-1} = O(1)$. The conditions in~\eqref{eq:IFT-cond} are readily verified, and one may take $\de_1 = C_1 \eta$ and $\de_2 = C_2 \eta$ in the notation of Remark~\ref{rem:IFT} in that case for uniform constants $C_1, C_2$. 
Indeed, 
\EQ{
| G( g, \vec{\ti \lam})|  \lesssim  \|g \|_H
}
and thus the second condition in~\eqref{eq:IFT-cond} is verified.  One may verify the first condition in~\eqref{eq:IFT-cond} using~\eqref{eq:partial-G} and~\eqref{eq:L_0-def}. 

An application of Remark~\ref{rem:IFT} yields a continuous mapping  $\varsigma: B_H(0; \de_1) \to B_{\R^M}(0; \de_2)$  such that
\EQ{
\ti G( g_0, \vec \ell) = \vec 0 \Longleftrightarrow \vec \ell = \varsigma(g_0). 
}
We define
\EQ{
g:= F(\ti g, \varsigma(\ti g)), \quad \vec \ell:= \varsigma(\ti g). 
}
Setting $\lam_j = e^{\ell_j}$, and $\bs g = (g, \dot v)$,  by construction we then have, 
\EQ{
\bs v =  \bs \calQ( m, \vec \iota, \vec \lam) + \bs g, 
}
and $\bs g$ satisfies~\eqref{eq:v-decomp} and~\eqref{eq:g-bound-0}. 

To prove the remaining estimates we expand the nonlinear energy of $\bs v$, 
\EQ{
M E( \bs Q) &+ \te^2 \ge E( \bs v) = E( \bs \calQ( m, \vec \iota, \vec \lam) + \bs g) \\
& = E( \bs \calQ( m, \vec \iota, \vec \lam) ) + \ang{ \uD E( \bs \calQ(m, \vec \iota, \vec \lam)) \mid \bs g}  + \frac{1}{2} \ang{ \uD^2 E( \bs \calQ(m, \vec \iota, \vec \lam))\bs g  \mid \bs g} + O(\| \bs g \|_{\E}^3) 
}
and apply the conclusions of Lemma~\ref{lem:D2E-coerce} and Lemma~\ref{lem:M-bub-energy}. This completes the proof. 
\end{proof} 

\begin{lem}  \label{lem:bub-config} Let $k \ge 1$. 
 There exists $\eta>0$ sufficiently small with the following property. Let $m, \ell \in \Z$, $M, L \in \N$,  $\vec\iota \in \{-1, 1\}^M, \vec \sigma \in \{-1, 1\}^L$, $\vec \lam \in (0, \infty)^M, \vec \mu \in (0, \infty)^L$,   and $w$ be such that $E_{\bfp}( w) < \infty$ and, 
 \begin{align} 
 \|w  - \calQ(m,  \vec \iota, \vec \lam)\|_{H}^2  + \sum_{j =1}^{M-1} \Big(\frac{\lam_j}{\lam_{j+1}} \Big)^{k} &\le \eta,  \label{eq:M-bub} \\
 \|w  - \calQ(\ell, \vec \sigma , \vec \mu)\|_{H}^2 +  \sum_{j =1}^{L-1} \Big(\frac{\mu_j}{\mu_{j+1}} \Big)^{k} &\le \eta. \label{eq:L-bub} 
 \end{align} 
 Then, $m = \ell$, $M = L$, $\vec \iota = \vec \sigma$. Moreover, for every $\te>0$ the number $\eta>0$ above can be chosen small enough so that 
 \EQ{ \label{eq:lam-mu-close} 
\max_{j = 1, \dots M} | \frac{\lam_j}{\mu_j} - 1 | \le  \te.
 }
\end{lem}

\begin{proof}[Proof of Lemma~\ref{lem:bub-config}] From~\eqref{eq:M-bub} we see that $\lim_{r \to \infty} w(r) = m \pi$, and from~\eqref{eq:L-bub} we see that $\lim_{r\to \infty} w(r) =  \ell\pi$. Hence,  $m = \ell$. 

Next, let $g_{\lam} := w - \calQ(m,  \vec \iota, \vec \lam)$ and $g_\mu := w - \calQ(\ell,  \vec \s, \vec \mu)$. By expanding the nonlinear potential energy we have, 
\EQ{
E_{\bfp}(w) = E_{\bfp}( \calQ(m,  \vec \iota, \vec \lam) ) + \ang{ DE_{\bfp}(\calQ(m,  \vec \iota, \vec \lam)) \mid g_\lam}  + O(\| g_\lam \|_H^2) . 
}
Choosing $\eta>0$ small enough so that Lemma~\ref{lem:M-bub-energy} applies, we see that 
\EQ{
M E( \bs Q) - C \eta \le E_{\bfp}( w) \le M E( \bs Q) +C \eta,
}
for some $C>0$. By an identical argument, 
\EQ{
L E( \bs Q) - C \eta \le E_{\bfp}( w) \le L E( \bs Q) +C \eta.
}
It follows that $M=L$. Next, we prove that $\eta>0$ can be chosen small enough to ensure that $\vec \iota = \vec \sigma$. Suppose not, then we can find a sequence $w_n$ with $E_{\bfp}( w_n) \le C$, and sequences $\vec \iota_n, \vec \s_n, \vec \lam_n, \vec \mu_n$ so that, 
\EQ{
\| w_n  - \calQ(m,  \vec \iota_n, \vec \lam_n)\|_{H}^2  + \sum_{j =1}^{M-1} \Big(\frac{\lam_{n,j}}{\lam_{n, j+1}} \Big)^{k} & = o_n(1)  \mas n \to \infty,\\
 \|w_n  - \calQ(m, \vec \sigma_n , \vec \mu_n)\|_{H}^2 +  \sum_{j =1}^{M-1} \Big(\frac{\mu_{n,j}}{\mu_{n, j+1}} \Big)^{k} &= o_n(1) \mas n \to \infty,
}
but with $\vec \iota_n \neq \vec  \s_n$ for every $n$.  We may assume without loss of generality that  $$0 = \lim_{r \to 0} w_n(r) = \lim_{r \to 0} \calQ( m, \vec \iota_n, \vec \lam_n; r) = \lim_{r \to 0} \calQ( m, \vec \s_n, \vec \mu_n; r)$$ and we note that above limits agree mean that we must have $\sum_{j =1}^M \iota_{n, j}  = \sum_{j =1}^M \s_{n, j}$ for each $n$. Passing to a subsequence we may assume that there exists an index $j_0 \ge 1$ such that $ \iota_{j, n} = \sigma_{j, n}$ for every $j < j_0$ and every $n$ and $\iota_{j_0, n} \neq \s_{j_0, n}$ for every $n$. We have, 
\EQ{ \label{eq:diff-sign} 
\| \calQ(m,  \vec \iota_n, \vec \lam_n) - \calQ(m, \vec \sigma_n , \vec \mu_n) \|_H \le \|w_n  - \calQ(m,  \vec \iota_n, \vec \lam_n)\|_{H} +   \|w_n  - \calQ(m, \vec \sigma_n , \vec \mu_n)\|_{H}  = o_n(1) .
}
First we show that $j_0>1$ 
Assume for contradiction that $j_0 = 1$. Then, we  may assume that $\iota_{n, 1} = 1$,  $\sigma_{n, 1} = -1$ and $\lam_{n, 1} < \mu_{n, 1}$ for all $n$. It follows that 
\EQ{
\calQ(m,  \vec \iota_n, \vec \lam_n) - \calQ(m, \vec \sigma_n , \vec \mu_n)  \ge \frac{\pi}{4}  \quad \forall r \in [\lam_{n, 1}, 2 \lam_{n, 1}],
}
for all $n$ large enough. 
But then, 
\EQ{
\| \calQ(m,  \vec \iota_n, \vec \lam_n) - \calQ(m, \vec \sigma_n , \vec \mu_n) \|_H^2 \ge  \int_{\lam_{n, 1}}^{2 \lam_{n, 1}}  (\pi/4)^2 \, \frac{\ud r}{r}  \ge (\pi/4)^2 \log 2,
}
for all sufficiently large $n$, which contradicts~\eqref{eq:diff-sign}. So $\iota_{1, n} = \sigma_{n, 1}$ for all $n$. Thus $j_0>1$. But then by a nearly identical argument we can show that we must have $\lam_{n, j} \simeq \mu_{n, j}$ uniformly in $n$ for all $j < j_0$.  Again we may assume (after passing to a subsequence) that  $\lam_{n, j_0} < \mu_{n, j_0}$.  It follows again that for all sufficiently large $n$ we have, 
\EQ{
\abs{\calQ(m,  \vec \iota_n, \vec \lam_n) - \calQ(m, \vec \sigma_n , \vec \mu_n) } \ge \frac{\pi}{4}  \quad \forall r \in [\lam_{n, j_0}, 2 \lam_{n, j_0}],
}
which again yields a contradiction. Hence we must have $\vec \iota = \vec \sigma$. 

Finally, we prove~\eqref{eq:lam-mu-close}. Suppose ~\eqref{eq:lam-mu-close} fails. Then there exists $\te_0>0$ and sequences $\vec \lam_{n}, \vec \mu_n$ such that 
\EQ{
\| \calQ(m,  \vec \iota_n, \vec \lam_n) - \calQ(m, \vec \iota_n , \vec \mu_n) \|_H = o_n(1),
} 
but 
\EQ{ \label{eq:diff-scale} 
\sup_{j =1, \dots, M} | \lam_{n, j}/ \mu_{n, j} - 1|\ge \te_0, 
}
 for all $n$. Following  the same logic as before we note that we must have $\lam_{n, j} \simeq \mu_{n, j}$ uniformly in $n$. But then we have, 
\EQ{
\| \calQ(m,  \vec \iota_n, \vec \lam_n) - \calQ(m, \vec \iota_n , \vec \mu_n) \|_H^2  =  \sum_{j = 1}^{M} \| Q_{\lam_{n, j}} - Q_{\mu_{n, j}} \|_H^2 + o_n(1) , 
}
which implies that $\| Q_{\lam_{n, j}} - Q_{\mu_{n, j}} \|_H = o_n(1)$ for every $j$, yielding a contradiction with~\eqref{eq:diff-scale}. This completes the proof. 
\end{proof} 

Later in the paper we require the following lemma, which gives the nonlinear interaction force between bubbles.  Given an $M$-bubble configuration, $\calQ(m, \vec \iota, \vec \lam)$ we set 
\EQ{ \label{eq:fi-def} 
f_{\bfi}( m, \vec \iota, \vec \lam) := -\frac{k^2}{r^2}  \Big( f( \calQ(m, \vec \iota, \vec \lam)) - \sum_{j=1}^M \iota_j f( Q_{\lam_j}) \Big) 
}

\begin{lem} \label{lem:interaction} Let $k \ge 1$, $M \in\N$. For any $\theta>0$ there exists $\eta>0$ with the following property. Let $\bs \calQ(m, \vec \iota, \vec \lam)$ be an $M$-bubble  configuration with 
\EQ{
 \sum_{j =0}^{M} \Big( \frac{ \lam_{j}}{\lam_{j+1}} \Big)^k \le \eta, 
 }
under the convention that $\lam_0 = 0$, $\lam_{M+1} = \infty$. Then, 
we have, 
\EQ{
\Big|  \ang{ \Lam Q_{\lam_j} \mid f_{\bfi}( m, \vec \iota, \vec \lam)}  + \iota_{j-1} 8k^2 \Big( \frac{ \lam_{j-1}}{\lam_j} \Big)^k  - \iota_{j+1} 8k^2 \Big( \frac{ \lam_{j}}{\lam_{j+1}} \Big)^k  \Big| \le \theta \Big( \Big( \frac{ \lam_{j-1}}{\lam_j} \Big)^k  +   \Big( \frac{ \lam_{j}}{\lam_{j+1}} \Big)^k\Big) 
}
where here $f_{\bfi}( m, \vec \iota, \vec \lam) $ is defined in~\eqref{eq:fi-def}. 
\end{lem} 
\begin{proof}
Letting $\ell = m - \sum_{j=1}^M \iota_j$ we have 
\EQ{
f(\calQ(m, \vec \iota, \vec\lam) = \frac{1}{2} \sin( 2 \ell \pi + 2\sum_{j=1}^M  \iota_j Q_{\lam_j}) = \frac{1}{2} \sin (2\sum_{j=1}^M  \iota_j Q_{\lam_j})
}
Fixing $j \in \{1, \dots, M\}$, we expand, 
\EQ{ \label{eq:fi-exp} 
\frac{1}{2} \sin (2\sum_{ i \neq j} & \iota_i Q_{\lam_i} + 2 \iota_j Q_{\lam_j})  - \frac{1}{2}\sum_{ i = 1}^M \iota_i \sin 2 Q_{\lam_i} \\
&= \frac{1}{2} \sin (2\sum_{ i \neq j}  \iota_i Q_{\lam_i})\Big(  \cos 2 Q_{\lam_j} -1) + \frac{1}{2} \Big(\cos (2\sum_{ i \neq j}  \iota_i Q_{\lam_i}) -1\Big) \iota_j \sin 2 Q_{\lam_j}   \\
&\quad + \frac{1}{2} \sin (2\sum_{ i \neq j}  \iota_i Q_{\lam_i})  - \frac{1}{2} \sum_{i \neq j} \iota_i \sin 2 Q_{\lam_i}  \\
& = - \iota_{j+1} \sin 2 Q_{\lam_{j+1}} \sin^2 Q_{\lam_j}- \iota_{j-1} \sin 2 Q_{\lam_{j-1}} \sin^2 Q_{\lam_j}  + \Psi_j( \vec \iota, \vec \lam_j)
}
where via an explicit computation using Lemma~\ref{lem:cross-term} the function $\Psi_j( \vec \iota, \vec \lam_j)$ above satisfies, 
\EQ{
\Big| \ang{ \Lam Q_{\lam_j} \mid r^{-2}\Psi_j( \vec \iota, \vec \lam_j)}\Big| \le  \te(\eta)  \Big(  \Big( \frac{ \lam_{j-1}}{\lam_j} \Big)^k  +   \Big( \frac{ \lam_{j}}{\lam_{j+1}} \Big)^k \Big), 
}
where $\te(\eta) >0$ is a function that tends to zero as $\eta \to 0$.  It follows that 
\EQ{
\ang{ \Lam Q_{\lam_j} \mid f_{\bfi}( m, \vec \iota, \vec \lam)} & \simeq + \iota_{j+1} \ang{ r^{-2} \Lam Q_{\lam_j}^3 \mid \sin 2 Q_{\lam_{j+1}}} +  \iota_{j-1}  \ang{ r^{-2}\Lam Q_{\lam_j}^3 \mid \sin 2 Q_{\lam_{j-1}}}  \\
 & =  \iota_{j+1} \ang{ r^{-2} \Lam Q_{\lam_j/ \lam_{j+1}}^3 \mid \sin 2 Q} + \iota_{j-1}  \ang{ r^{-2}\Lam Q_{\lam_j/ \lam_{j-1}}^3 \mid \sin 2 Q} 
}
where ``$\simeq$'' above means up to negligible terms. 
Note that,    
\EQ{
\sin 2Q= 4 r^k \frac{1- r^{2k}}{(1+ r^{2k})^2}  &= 4r^k + O(r^{3k})  \mif r \ll1  \\
& = -4 r^{-k} + O( r^{-3k}) \mif r \gg 1. 
}
Via residue calculus we compute, 
\EQ{ \label{eq:Q3}
 \int_0^\infty \Lam Q(r)^3 4 r^{k} \, \frac{\ud r}{r}  =  32k^3 \int_0^\infty \frac{ r^k}{(r^{k} + r^{-k})^3} \, \frac{\ud r}{r}  &= 8k^2  \\
 \int_0^\infty \Lam Q(r)^3 4 r^{-k} \, \frac{\ud r}{r}  =  32k^3 \int_0^\infty \frac{ r^{-k}}{(r^{k} + r^{-k})^3} \, \frac{\ud r}{r}  & = 8k^2 
}
And thus, 
\EQ{
+ \iota_{j+1} \ang{ r^{-2} \Lam Q_{\lam_j/ \lam_{j+1}}^3 \mid \sin 2 Q}  &=  + \iota_{j+1}8k^2\Big( \frac{\lam_j}{\lam_{j+1}} \Big)^k + \theta(\eta) \Big(\frac{\lam_j}{\lam_{j+1}} \Big)^k\\
+ \iota_{j-1}  \ang{ r^{-2}\Lam Q_{\lam_j/ \lam_{j-1}}^3 \mid \sin 2 Q}  & = - \iota_{j-1}  8k^2 \Big( \frac{\lam_{j-1}}{\lam_{j}} \Big)^k + \theta(\eta) \Big(\frac{\lam_{j-1}}{\lam_{j}} \Big)^k
}
where $\theta(\eta) \to 0$ as $\eta \to 0$, which completes the proof; see~\cite[Proof of Claim 3.14]{JL1} for more details of this computation. 
\end{proof}


\section{Localized sequential bubbling} \label{sec:compact} 

The goal of this section is to prove a localized sequential bubbling lemma for sequences of wave maps with vanishing averaged kinetic energy on an expanding region of space. The main result, and the arguments used to prove it are in the spirit of the main theorems in C\^ote~\cite{Cote15} and Jia and Kenig~\cite{JK}, and also use many ideas from Struwe~\cite{Struwe} and Duyckaerts, Kenig, and Merle~\cite{DKM3}. 

To state the compactness lemma, we define a localized distance function, 
\EQ{ \label{eq:delta-def} 
\bs \de_R( \bs u) :=  \inf_{m, M,  \vec \iota, \vec \lam}  \Big( \| u - \calQ( m, \vec \iota, \vec \lam) \|_{H( r \le R)}^2 + \| \dot u \|_{L^2(r \le R)}^2 + \sum_{j = 1}^{M} \Big(\frac{ \lam_j}{ \lam_{j+1}}\Big)^k \Big)^{\frac{1}{2}}. 
}
where  the infimum above is taken over all $m \in \Z$, $M \in\{0, 1, 2, \dots\}$,  and all  vectors $\vec \iota\in \{-1, 1\}^M, \vec \lam \in (0, \infty)^M$, and here we use the convention that the last scale $\lam_{M+1} = R$.

\begin{lem}[Compactness Lemma] \label{lem:compact}  Let $\ell, m \in \Z$. Let $\rho_n>0$ be a sequence of positive numbers and let $\bs u_n(t) \in\E_{\ell, m}$ be a sequence of wave maps  on the time intervals $[0, \rho_n]$ such that $\limsup_{n \to \infty} E(\bs u_n) < \infty$. 

Suppose there exists a sequence $R_n \to \infty$  such that,
\EQ{
\lim_{n \to \infty} \frac{1}{ \rho_n} \int_0^{\rho_n} \int_0^{ \rho_n R_n} \abs{\p_t u_n(t, r)}^2 \, r \, \ud r\,  \ud t
 = 0.
}
Then, up to passing to a subsequence of the $\bs u_n$,  there exists a time sequence $t_n \in [0, \rho_n]$ and a sequence $r_n \le R_n$ with $r_n \to \infty$ such that 
\EQ{
\lim_{n \to \infty} \bs \de_{r_n\rho_n}( \bs u_n(t_n))  = 0.
}

\end{lem} 

\begin{rem} \label{rem:seq} 
We note that Theorem~\ref{thm:seq} in the blow-up case is a quick consequence of Lemma~\ref{lem:compact} together with the fundamental result of Shatah and Tahvildar-Zadeh~\cite{STZ92}, that for wave map developing a singularity at $T_- = 0$ one has, 
\EQ{
\lim_{t \to 0} \frac{1}{t} \int_0^t \int_0^\tau \abs{ \p_t u(\tau, r)}^2 \, r \, \ud r \, \ud t = 0. 
}
 In the global case $T_+ = \infty$  one uses,  
\EQ{
\lim_{A \to \infty} \limsup_{T\to \infty} \frac{1}{T} \int_{A}^T \int_0^{t-A} \abs{ \p_t u(t, r)}^2 \, r \, \ud r \, \ud t = 0, 
}
proved in~\cite{CKLS2} using the analysis of~\cite{STZ92}. 
\end{rem} 

\subsection{Prior results on bubbling} 

The proof of Lemma~\ref{lem:compact} requires several preliminary lemmas, including two Real Analysis results, which we address first.  
\begin{lem}
\label{lem:sequences}
If $a_{k, n}$ are positive numbers such that $\lim_{n\to \infty}a_{k, n} = \infty$ for all $k \in \bN$,
then there exists a sequence of positive numbers $b_n$ such that $\lim_{n\to \infty} b_n = \infty$
and $\lim_{n\to \infty} a_{k, n} / b_n = \infty$ for all $k \in \bN$.
\end{lem}
\begin{proof}
For each $k$ and each $n$ define $\ti a_{k, n} = \min\{ a_{1, n}, \dots, a_{k, n}\}$. Then the sequences $\ti a_{k, n} \to \infty$ as $n \to \infty$ for each $k$, but also satisfy $\ti a_{k, n} \le a_{k, n}$ for each $k, n$, as well as $\ti a_{j, n} \le \ti a_{k, n}$ if $j>k$. Next, choose a strictly increasing sequence $\{n_k \}_k \subset \N$ such that $\ti a_{k, n} \ge k^2$ as long as $n \ge n_k$. For $n$ large enough, let $b_n \in \bN$ be determined by the condition
$n_{b_n} \leq n < n_{b_n + 1}$. Observe that $b_n \to \infty$ as $n \to \infty$. Now fix any $\ell \in \N$ and let $n$ be such that $b_n > \ell$. We then have
\begin{equation}
a_{\ell, n} \geq \ti a_{\ell, n} \ge \ti a_{b_n, n} \ge  b_n^2  \gg b_n.
\end{equation}
Thus the sequence $b_n$ has the desired properties. 
\end{proof}
If $f: [0, 1] \to [0, +\infty]$ is a measurable function, we denote by 
\EQ{
Mf(\tau) := \sup_{I \ni \tau; I \subset [0, 1]} \frac{1}{\abs{I}} \int_I f(t) \, \ud t 
}
its Hardy-Littlewood maximal function. Recall the weak-$L^1$ boundedness estimate
\begin{equation}
\label{eq:weak-max}
|\{\tau \in [0, 1]: Mf(\tau) > \alpha\}| \leq \frac{3}{\alpha}\int_0^1f(t)\ud t, \qquad\text{for all }\alpha > 0,
\end{equation}
see \cite[Section 2.3]{Muscalu-Schlag}.

\begin{lem}
\label{lem:maximal}
Let $f_n$ be a sequence of continuous positive functions defined on $[0, 1]$ such that $\lim_{n \to \infty}\int_0^1 f_n(t)\ud t = 0$ and let $g_n$ be a uniformly bounded sequence of real-valued continuous functions on $[0, 1]$ such that
$\limsup_{n\to\infty} \int_0^1 g_n(t)\ud t \leq 0$.
Then there exists a sequence $t_n \in [0, 1]$ such that
\begin{equation}
\lim_{n\to\infty} Mf_n(t_n) = 0, \qquad \limsup_{n\to \infty}g_n(t_n) \leq 0.
\end{equation}
\end{lem}
\begin{proof}
Let $\alpha_n$ be a sequence such that $\int_0^1 f_n(t)\ud t \ll \alpha_n \ll 1$. Let $A_n :=  \{t \in [0, 1]: Mf_n(t) \le \alpha_n\}$. 
By \eqref{eq:weak-max}, $\lim_{n\to\infty}|A_n| = 1$. Since $g_n$ is uniformly bounded, we have
\begin{equation}
\int_{[0, 1]\setminus A_n}|g_n(t)|\ud t \lesssim |[0, 1] \setminus A_n| \to 0,
\end{equation}
which implies
\begin{equation}
\limsup_{n\to\infty}\int_{A_n}g_n(t)\ud t \leq 0.
\end{equation}
It suffices to take $t_n \in A_n$ such that $g_n(t_n) \leq |A_n|^{-1}\int_{A_n}g_n(t)\ud t$.
\end{proof}

%
%

A key ingredient of the proof of Lemma~\ref{lem:compact} is a Struwe-type bubbling lemma~\cite{Struwe}. We require the version proved in~\cite{Cote15, JK}. 

\begin{lem} [Bubbling]\emph{ \cite{Struwe}, \cite[Proposition 3.1]{Cote15},\cite[Lemma 5.6]{JK}} \label{lem:bubbling} 
Let $\s>0$ and let $\alpha_n \to 0$ and $\beta_n \to \infty$ be two sequences.  Let $\bs v_n$ be a sequence of wave maps, i.e., solutions to~\eqref{eq:wmk}, on the time interval $[0, \s]$ such that $\limsup_{n \to \infty} E(\bs v_n) < \infty$. Suppose that 
\EQ{
\lim_{n \to \infty} \frac{1}{\s} \int_0^\s \int_{\alpha_n}^{\beta_n}  | \p_t v_n(t,r)|^2 \, r \, \ud r \, \ud t  = 0
}
Then, there exists an integer $m_0$, $\iota_0 \in \{-1, 0, + 1\}$, and  a scale $\lam_0>0$ such that, up to passing to a subsequence, we have 
\EQ{
\bs v_n  \to   m_0 \bs \pi + \iota_0 \bs  Q_{\lam_0}
}
in the space $(L^2_t(\E))_{\loc}([0, \s]\times (0, \infty))$. In addition $\bs v_n \to m_0 \bs \pi + \iota_0 \bs  Q_{\lam_0}$ locally uniformly in $[0, \s] \times (0, \infty)$. And finally, $\bs v_n(0) \to m_0 \bs \pi + \iota_0 \bs Q_{\lam_0}$ in the space $\E_{\loc}((0,\infty))$.  
\end{lem} 


The lengthy proof of the Compactness Lemma  will consist of several steps, which are designed to reduce the proof to the exact scenarios already considered by C\^ote in~\cite[Proof of Lemma 3.5]{Cote15} and then by Jia-Kenig in~\cite[Proof of Theorem 3.2]{JK}. In particular, we will seek to apply the following result from~\cite{JK}.  

\begin{lem} \emph{\cite[Theorem 3.2]{JK}} \label{lem:jk} 
Let $\bs v_n$ be a sequence of wave maps, i.e., solutions to~\eqref{eq:wmk}, on the time interval $[0, 1]$ such that $\limsup_{n \to \infty} E(\bs v_n) < \infty$. Suppose that there exists a sequence $t_n \in [0, 1]$, and integer $K_0 \ge 0$,  and scales $\lam_{n, 1} \ll  \dots \ll \lam_{n, K_0} \lesssim 1$ such that 
\EQ{
\bs{v}_n(t_n) &= m_1 \bs \pi + \sum_{j =1}^{K_0} (  \iota_j Q \big(  \frac{ \cdot}{\lam_{n, j}} \big)  , 0) -  \bs \pi)  + \bs{w}_{n, 0}, 
}
where $\|\bs{w}_{n, 0} \|_{L^\infty \times L^2} \to 0$ and $\| \bs{ w}_{n, 0} \|_{\E(r \ge  r_n^{-1})} \to 0$ as $n \to \infty$ for some sequence $r_{n} \to \infty$. Suppose in addition that,   $\| \bs{w}_{n, 0} \|_{\E(A^{-1} \lam_n \le r \le A \lam_n)} \to 0$ as $n \to \infty$ for any sequence $\lam_n \lesssim 1$ and any $A>1$, 
and finally, that 
\EQ{ \label{eq:jk0} 
\limsup_{n\to \infty}\int_0^\infty \bigg( k^2 \frac{\sin^2(2 v_n(t_n))}{2r^2} + (\p_r v_n(t_n))^2 2 \cos (2 v_n(t_n))  \bigg) \,r \,  \ud r    \le 0 . 
}
Then, 
\EQ{
\| \bs{w}_{n, 0} \|_{\E} \to 0 \mas n \to \infty. 
}
\end{lem} 

\begin{rem} 
 Lemma~\ref{lem:jk} is not stated in~\cite{JK} exactly as given above. However, an examination of~\cite[Proof of Theorem 3.2]{JK} shows that this is precisely what is established.   The heart of the matter lies in the fact that the Jia-Kenig virial functional~\eqref{eq:jk0} vanishes at $Q$, i.e., 
\EQ{
\int_0^\infty \bigg( k^2 \frac{\sin^2(2Q)}{2r^2} + (\p_r Q)^2 2 \cos (2 Q)  \bigg) \,r \,  \ud r    = 0 , 
} 
but gives coercive control of the energy in regions where $v_n(t, r)$ is near integer multiples of~$\pi$.
\end{rem} 

\subsection{Proof of the compactness lemma} 

\begin{proof}[Proof of Lemma~\ref{lem:compact}] 
Rescaling we may assume that $\rho_n = 1$ for each $n$.

\textbf{Step 1.}
We claim that there exist $\sigma_n \in [0, \frac 13]$, $\tau_n \in [\frac 23, 1]$ , and a sequence $R_{1,n} \le R_n$ with $R_{1, n} \to \infty$ as $n \to \infty$  such that
\begin{equation}
\label{eq:jia-kenig-0}
\lim_{n\to \infty}\int_{\sigma_n}^{\tau_n}\int_0^\infty\bigg( k^2 \frac{\sin^2(2u_n)}{2r^2}\chi(\cdot/ R_{1,n})- (\partial_r^2 u_n + \frac 1r \partial_r u_n)\sin(2u_n)\chi(\cdot /  R_{1, n})\bigg)\,r\vd r\ud t = 0,
\end{equation}
where $\chi$ is a smooth cut-off function equal $1$ on $[0, \frac 12]$, with support in $[0, 1]$. Here and later in the argument the second term in the integrand in~\eqref{eq:jia-kenig-0} is to be interpreted as the expression obtained after integration by parts, which is well defined due to the finiteness of the energy. 

Since 
\EQ{
\lim_{n\to \infty}\int_0^\frac 13 \int_0^{R_n}(\partial_t u_n)^2\,r\vd r = 0 \mand \lim_{n\to \infty}\int_\frac 23^1 \int_0^{R_n}(\partial_t u_n)^2\,r\vd r = 0
}
 there exist $\sigma_n \in [0, \frac 13]$,  $\tau_n \in [\frac 23, 1]$ and a sequence $R_{1, n} \to \infty$ such that, 
\EQ{ \label{eq:sigma-tau} 
\lim_{n\to\infty}R_{1, n}  \int_0^{R_n}(\partial_t u(\sigma_n))^2\,r\vd r = 0 \mand 
\lim_{n\to\infty}  R_{1, n} \int_0^{R_n}(\partial_t u(\tau_n))^2\,r\vd r = 0
}
For $t \in [\sigma_n, \tau_n]$, we have the following Jia-Kenig virial identity; see~\cite[Lemma 3.5 and Lemma 3.10]{JK}.
\EQ{
\label{eq:jia-kenig-1} 
\dd t\int_0^\infty \partial_t u_n\sin(2u_n)\chi(\cdot /  R_{1, n})&\,r\vd r = \int_0^\infty 2\cos(2u_n)(\partial_t u_n)^2\chi(\cdot/ R_{ 1, n})\,r\vd r \\ &+ \int_0^\infty (\partial_r^2 u_n + \frac 1r \partial_r u_n - k^2 \frac{\sin(2u_n)}{2r^2})\sin(2u_n)\chi(\cdot /  R_{ 1, n})\,r\vd r.
}
By the Cauchy-Schwarz inequality, the boundedness of the nonlinear energy and~\eqref{eq:sigma-tau}, we see that 
\begin{equation}
\lim_{n\to \infty}\int_0^\infty \big(|\partial_t u_n(\sigma_n)||\sin(2u_n(\sigma_n))|
+ |\partial_t u_n(\tau_n)||\sin(2u_n(\tau_n))|\big)\chi(\cdot/R_{1, n})\,r\vd r = 0.
\end{equation}
Integrating~\eqref{eq:jia-kenig-1} between $\sigma_n$ and $\tau_n$, and using the above, we obtain \eqref{eq:jia-kenig-0}.

\textbf{Step 2.}
We rescale again so that $[\sigma_n, \tau_n]$ becomes $[0, 1]$. We apply 
Lemma~\ref{lem:maximal}, to 
\EQ{
f_n(t)&:= \int_0^{R_n}\abs{\p_t u_n(t, r)}^2 \, r \, \ud r, \\
g_n(t) &:= \int_0^\infty\bigg( k^2 \frac{\sin^2(2u_n)}{2r^2}- (\partial_r^2 u_n + \frac 1r \partial_r u_n)\sin(2u_n)\bigg)\chi(\cdot/ R_{1, n})\,r\vd r
}
(integrating by parts the second term in $g_n$, we see that this is a uniformly bounded sequence of continuous functions)
and we find a sequence $\{t_n\} \in [0, 1]$ such that we have vanishing of the maximal function of the local kinetic energy,
\EQ{ \label{eq:maximal} 
&\lim_{ n \to \infty} \sup_{I \ni t_n;  I \subset [0, 1]}  \frac{1}{\abs{I}} \int_I \int_0^{R_n} | \p_t u_n(t, r)|^2 \, r \, \ud r \, \ud t = 0,  
\\
\textrm{and} \, \,\,  & \lim_{n \to \infty} \int_0^{R_n} | \p_t u_n(t_n, r)|^2 \, r \, \ud r \, \ud t = 0, 
}
and also pointwise vanishing of a localized Jia-Kenig virial functional, 
\begin{equation}
\label{eq:jia-kenig}
\limsup_{n\to \infty}\int_0^\infty\bigg( k^2 \frac{\sin^2(2u_n(t_n))}{2r^2}- (\partial_r^2 u_n(t_n) + \frac 1r \partial_r u_n(t_n))\sin(2u_n(t_n))\bigg)\chi(\cdot/\ti R_{ n})\,r\vd r \leq 0. 
\end{equation}
for any sequence $\ti R_{n} \le R_{1, n} \le R_n$ with $\ti R_n \to \infty$ as $n \to \infty$. 
We emphasize the  conclusion from the first steps is the existence of the sequence $t_n$ such that~\eqref{eq:maximal} and~\eqref{eq:jia-kenig} hold.   

\textbf{Step 3.} 
Now that we have chosen the sequence $t_n \in[0, 1]$, we may, after passing to a subsequence, assume that $t_n \to t_0 \in [0, 1]$. 


We apply Lemma~\ref{lem:pd} to the sequence $\bs u_n(t_n)$, obtaining profiles $( \bs \psi^j, \lam_{n, j})$ and $(\bs v^i\lin, t_{n, i}, \s_{n, i})$, and $\bs w_{n, 0}^J$,  so that,  using the notation, 
\EQ{
\bs \psi_n^j&:= \big( \psi^j( \cdot/ \lam_{n, j}), \lam_{n, j}^{-1} \dot\psi^j( \cdot/ \lam_{n, j}) \big) , \quad  
\bs  v_{\Lin, n}^i(0) :=  \big( v_{\Lin}^i (  \frac{-t_{n, i}}{\s_{n, i}}, \frac{ \cdot}{\s_{n, i}} ) , \s_{n, i}^{-1} \p_t v_{\Lin}^i (  \frac{-t_{n, i}}{\s_{n, i}}, \frac{ \cdot}{\s_{n, i}} )\big), 
}
 we have
\EQ{ \label{eq:profiles1} 
\bs u_n(t_n) &= m_0 \bs \pi +  \sum_{j = 1}^{K_0} ( \bs \psi^j_n - m_j \bs \pi)   + \sum_{i =1}^J \bs  v_{\Lin, n}^i(0) + \bs w_{n, 0}^J
}
satisfying the conclusions of Lemma~\ref{lem:pd}. We refer to the profiles $(\bs \psi^j, \lam_{n, j})$ as well as the profiles $(\bs v\lin^i(0), t_{n, i}, \s_{n, i})$ with $t_{n, i} = 0$ for all $n$ as \emph{centered} profiles (here the subscript $\lin$ on $\bs v\lin^i$ is superfluous).  We refer to  the profiles $(\bs v\lin^i(0), t_{n, i}, \s_{n, i})$ with $-t_{n, i}/ \s_{n, i} \to \pm \infty$ as \emph{outgoing/incoming} profiles. 

\textbf{Step 4.}(Centered profiles at large scales) 
At each step, we will impose conditions on the choice of the ultimate choice of sequence $r_n \to \infty$. Consider the set of indices 
\EQ{
\calJ_{c, \infty}&:= \{ j  \in \{1, \dots, K_0\} \mid  \lim_{n \to \infty}\lam_{n, j} =  \infty  \}   \cup \{ i \in \N \mid t_{n, i} = 0 \, \, \forall n, \mand \lim_{n \to \infty}\s_{n, i} = \infty \}
}
Using Lemma~\ref{lem:sequences} we choose a sequence $ r_{ 0, n} \to \infty$ so that $r_{ 0, n} \ll R_n, \lam_{n, j}, \s_{n, i}$ for each $\lam_{n, j}$ with $j \in \calJ_{c, \infty}$ and each $\s_{n, i}$ with $i \in \calJ_{c, \infty}$. 
We note that by construction we have, 
\EQ{ \label{eq:Jc-infty} 
E( \bs \psi_n^j;  0,  r_{0, n}) &\to 0 \mas n \to \infty, \\
E(  (v\lin^i( 0, \cdot/\s_{n, i}), \s_{n, i}^{-1} \dot v^i \lin (\cdot/ \s_{n, i})); 0, r_{0, n}) &\to 0 \mas n \to \infty, 
}
for any of the indices $j, i \in \calJ_{c, \infty}$. 

%
%


\textbf{Step 5.}(Centered profiles at bounded  scales)  Consider the set of indices 
\EQ{
\calJ_{c, 0}&:= \{ j  \in \{1, \dots, K_0\} \mid \lim_{n \to \infty}\lam_{n, j} < \infty \}  \cup \{ i \in \N \mid t_{n, i} = 0 \, \, \forall n, \mand \lim_{n \to \infty} \s_{n, i} < \infty \}
}
We use  Lemma~\ref{lem:bubbling} to show that each of the associated profiles must be a harmonic map. 

Consider first the case of a profile $( \bs \psi^j, \lam_{n, j})$ with $j \in \calJ_{c, 0}$. Define, 
\EQ{
\bs u_{n}^j(t, r) = (u_n^j(t, r),\p_t u_n^j(t, r) )  := \big( u_n( t_n + \lam_{n, j} t, \lam_{n, j} r), \, \lam_{n, j} \p_t u_n( t_n + \lam_{n, j} t, \lam_{n, j} r)\big) 
}
and note that $\bs u_n^j$ is a wave map on the interval $t \in [ -t_{n}/ \lam_{n, j}, (\rho- t_{n})/ \lam_{n, j}]$.  Consider the case $t_n \to t_0 < 1$, (the other possible limits have nearly identical arguments).  
Recall that we have the weak convergence $\bs u_{n}^j(0) \rightharpoonup \bs \psi^j(0)$. 
Moreover, 
\EQ{
\frac{1}{\s} \int_{0}^\s \int_{0}^{\frac{R_n}{\lam_{n, j}}} | \p_t u_n^j(t, r) |^2 \, r \, \ud r \, \ud t &= 
\frac{1}{\s} \int_{0}^\s \int_{0}^{\frac{R_n}{\lam_{n, j}}} | \lam_{n, j} \p_t u_n( t_n + \lam_{n, j} t, \lam_{n, j} r) |^2 \, r \, \ud r \, \ud t \\
& = \frac{1}{\s \lam_{n, j}} \int_{t_n}^{t_n+ \lam_{n, j} \s} \int_0^{R_n} | \p_t u_n( s, y) |^2 \, y \, \ud y \, \ud s \to 0 \mas n \to \infty 
}
where the last line follows from~\eqref{eq:maximal} after fixing $\s >0$ small enough so that $t_n + \lam_{n, j} \s \le 1$ for all $n$ large enough. Thus by Lemma~\ref{lem:bubbling} we conclude that there exists $\ti m_j, \iota_j, \lam_{0, j}$ so that $ \bs \psi^j = \ti m_j \bs \pi +\iota_j \bs Q_{\lam_{0, j}}$. 


The cases of profiles $(\bs v\lin^i(0), t_{n, i}, \s_{n, i})$ with $i \in \calJ_{c, 0}$ are completely analogous. And we conclude that each of these profiles must satisfy 
\EQ{
\bs v\lin^i(0, r) = (0, 0).
}
since each $\bs v\lin^i(0) \in \calE$ and thus can only be a constant harmonic map.

\textbf{Step 6.}(Incoming/outgoing profiles  with $\lim_{n \to \infty}\abs{t_{n, i}} = \infty$) We next treat profiles $(\bs v^i\lin, t_{n, i}, \s_{n,i})$ that satisfy, 
\EQ{
-\frac{t_{n, i}}{\s_{n, i}} \to \pm \infty. 
}
Up to passing to a subsequence of $\bs u_n(t_n)$ we may assume that $-t_{n, i}  \to t_{\infty} \in [- \infty, \infty]$. Consider the set of indices, 
\EQ{
\calJ_{\Lin, \infty} := \{ i \in \N \mid -\frac{t_{n, i}}{\s_{n, i}} \to \pm \infty \mand \abs{ t_{n, i}} \to \infty \}.
}
We impose additional restrictions on the sequence $r_n$. We require that $r_n \le \frac{1}{2} \abs{t_{n, i}}$ for each sequence $t_{n, i}$ in $\calJ_{\Lin, \infty}$. So at this stage, we again use Lemma~\ref{lem:sequences} to choose a sequence $r_{1, n} \to \infty$ such that $r_{1, n} \le r_{0, n}$ and $r_{1, n} \le \frac{1}{2} \abs{t_{n, i}}$ for each sequence $t_{n, i}$ in $\calJ_{\Lin, \infty}$. 

Since $\bs v^i\lin$ is a solution to~\eqref{eq:lin-2d} we know that it asymptotically concentrates all of its energy near the light-cone. In fact, a direct consequence of~\cite[Theorem 4]{CKS} is that 
\EQ{
\lim_{s \to \pm \infty}  \| \bs v^i\lin(s) \|_{\E(r \le \frac{1}{2} \abs{s})} = 0 .
} 
Thus, if $i\in \calJ_{\Lin, \infty}$ and as long as $ r_{1, n} \le \frac{1}{2} \abs{t_{n, i}}$ for $n$ large enough,  we see that $\s_{n, i}^{-1} r_{1, n} \le \frac{1}{2} \s_{n,i}^{-1} \abs{t_{n, i}}$ and thus 
\EQ{ \label{eq:Jlin-infty} 
\| \bs v^i_{\Lin}(-t_{n, i}/ \s_{n, i}) \|_{\E( r \le r_{1, n} \s_{n, i}^{-1})} \to 0 \mas n \to \infty. 
}
by the above and we conclude that any such profile does not contribute to the asymptotic size of $\bs \de_{r_{1, n}}( \bs u_n(t_n))$.

\textbf{Step 7.}(Incoming/outgoing profiles with $\lim_{n \to \infty}\abs{t_{n, i}} < \infty$) Next, we consider profiles $(\bs v^i\lin, t_{n, i}, \s_{n,i})$ such that 
\EQ{
-\frac{t_{n, i}}{\s_{n, i}} \to \pm \infty \mand  -t_{n, i} \to t_{\infty, i} \in \R
}
and we denote by $\calJ_{\Lin, 0}$ the set indices labeling all such profiles, and  note that $\s_{n, i} \to 0$ as $n \to \infty$ for each $i \in \calJ_{\Lin, 0}$. We claim that any such profile must satisfy $\bs v^i\lin  \equiv 0$. The argument we use follows closely the argument given in~\cite[Erratum]{DKM3erratum}. As there are few technical changes due to setting of the current problem, we reproduce the argument here.

We claim that there exists a new sequence $\sqrt{r_{1, n}} \leq  r_{2, n} \leq r_{1, n}$ such that
\begin{equation}\label{eq:r2} 
\lim_{n\to\infty}\sup_{t \in [0, 1]}E(\bs u_n(t); A_n^{-1}r_{2, n}, A_n  r_{2, n}) = 0
\end{equation}
for some $1 \ll A_n \ll  r_{2, n}$. 
By the finite speed of propagation, it suffices to have
\begin{equation}
\lim_{n\to\infty}E(\bs u_n(0); A_n^{-1} r_{2, n}, A_n  r_{2, n}) = 0, 
\end{equation}
and then replace $A_n$ by its half, for example.

Let $A_n$ be the largest integer such that $A_n^{2A_n} \leq \sqrt{r_{1,n}}$. Obviously, $1 \ll A_n \ll \sqrt{r_{1,n}}$.
For $l \in \{0, 1, \ldots, A_n - 1\}$, set $R_n^{(l)} := A_n^{2l}\sqrt{r_{1, n}}$, so that $A_n^{-1}R_n^{(l+1)} = A_nR_n^{(l)}$, thus
\begin{equation}
\sum_{l=0}^{A_n - 1}E(\bs u_n(0); A_n^{-1}R_n^{(l)}, A_nR_n^{(l)}) \leq E(\bs u_n(0)).
\end{equation}
Since all the terms of the sum are positive, there exists $l_0 \in \{0, 1, \ldots, A_n - 1\}$ such that $r_{2, n} := R_n^{(l_0)}$ satisfies
\begin{equation}
E(\bs u_n(0); A_n^{-1}r_{2, n}, A_n r_{2, n}) \leq A_n^{-1}E(\bs u_n(0)) \to 0.
\end{equation}
proving~\eqref{eq:r2} 

Next, using the finite speed of propagation along with~\eqref{eq:r2},  we pass to a new sequence of maps $\bs{\ti u}_n$ with vanishing average kinetic energy on the whole space. To see this, first we use Lemma~\ref{lem:pi} to find a sequence $y_n \in [2 r_{2, n}, 4 r_{2, n}]$ and  integers $m_n \in \Z$ such that 
\EQ{
| u_n( 0, y_n) - m_n \pi|   \to 0 \mas n \to \infty
}
Since $\ell, m$ are fixed and $\limsup_{n \to \infty} E( \bs u_n) < \infty$, the integers $m_n \in [-L, L]$ for all $n$ for some $L >0$. Hence, after  passing to a subsequence, we may assume that $m_n = m_1$ is a fixed integer for each $n$. We define a sequence of truncated initial data $\bs{\ti u}_n(0)$ as follows, 
\EQ{
\bs {\ti u}_n(t_n, r) = \chi_{2 r_{2,n}}(r) \bs u_n (t_n, r) + (1- \chi_{2 r_{2, n} } (r)) m_1 \bs \pi 
}
Using~\eqref{eq:r2}, we have  $E( \bs {\ti u}_n(t_n); \frac{1}{8} r_{2, n}, 8r_{2, n}) \to 0$ as $n \to \infty$. Let $\bs {\ti u}_n(t)$ denote the wave map evolution of the data $\bs{\ti u}_n(t_n)$, which we observe, using the vanishing of the energy of the data on the region $[r_{2, n}/8,  8r_{2, n}]$ is well defined on the interval $[0, 1]$ for large $n$. In fact, using the finite speed of propagation and the monotonicity of the energy on truncated cones, we see that  $\bs{\ti u}_n(t)$ satisfies, 
\EQ{ \label{eq:tiu-fsp} 
\bs{ \ti u}_n(t, r)  = \bs u_n(t, r)  \mif r \le r_{2, n}, \mand \sup_{t \in [0, 1]}E(\bs {\ti u}_n(t); r_{2, n}, \infty) \to 0 \mas n \to \infty.
}
Next, from the decomposition ~\eqref{eq:profiles1}  we have, 
\EQ{ \label{eq:profiles02} 
\bs{ \ti u}_n(t_n) &= m_1 \bs \pi + \sum_{ j \in \calJ_{c, 0}} (  \iota_j Q \big(  \frac{ \cdot}{\lam_{n, j}} \big)  , 0) -  \bs \pi)  + \sum_{i \le J, \,  \, i \in \calJ_{\Lin, 0}} \bs v_{\Lin, n}^i(0)  + +\chi_{2r_{2, n}} \bs w_{n, 0}^J\\
&\quad - \chi_{2r_{2, n}}m_1 \bs \pi + \chi_{2 r_{2, n}} m_0 \bs \pi \\
& \quad + (\chi_{2r_{2, n}} - 1)  \sum_{ j \in \calJ_{c, 0}} (  \iota_j Q \big(  \frac{ \cdot}{\lam_{n, j}} \big)  , 0) -  \bs \pi)   + (\chi_{2r_{2, n}} - 1) \sum_{i \le J, \,  \, i \in \calJ_{\Lin, 0}} \bs v_{\Lin, n}^i(0) \\
&\quad +\chi_{2r_{2, n}}\sum_{ j \in \calJ_{c, \infty}} \bs \psi_n^j(0)  +\chi_{2r_{2, n}}\sum_{i \le J, \,  \, i \in \calJ_{\Lin, 0}} \bs v_{\Lin, n}^i(0) 
}
where above we have allowed the abuse of notation, $\lam_{n, j} \leftrightarrow  \lam_{n, j} \lam_{0, j}$, for the profiles with indices in $\calJ_{c, 0}$. Using the same logic used to deduce ~\eqref{eq:Jc-infty} and~\eqref{eq:Jlin-infty} we have, 
\EQ{
E( Q_{\lam_{n, j}} ; r_{2, n}, \infty)  \to 0 \mas n \to \infty, \quad \| \bs v^i_{\Lin}(-t_{n, i}/ \s_{n, i}) \|_{\E( r \ge r_{2, n} \s_{n, i}^{-1})} \to 0 \mas n \to \infty. 
}
for any fixed $j \in \calJ_{c, 0}$  or $i \in \calJ_{\Lin, 0}$.  Thus, using ~\eqref{eq:pyth} and the above along with  ~\eqref{eq:Jc-infty} and~\eqref{eq:Jlin-infty} we see that the last three lines in~\eqref{eq:profiles2} can effectively be absorbed into the error and writing 
\EQ{ \label{eq:ti-w} 
\bs{\ti w}_{n, 0}^J(r) :=  \chi_{2 r_{2, n}}(r)\bs w_{n, 0}^J(r) + o_n(1)
} 
we obtain the decomposition, 
\EQ{ \label{eq:profiles2} 
\bs{ \ti u}_n(t_n) &= m_1 \bs \pi + \sum_{ j \in \calJ_{c, 0}} (  \iota_j Q \big(  \frac{ \cdot}{\lam_{n, j}} \big)  , 0) -  \bs \pi)  + \sum_{i \le J, \,  \, i \in \calJ_{\Lin, 0}} \bs v_{\Lin, n}^i(0)  + \bs {\ti w}_{n, 0}^J . 
}
We claim that that above is a profile decomposition for $\bs{\ti u}_{n}(t_n)$ in that it satisfies the conclusions of Lemma~\ref{lem:pd}. Indeed, it remains to check the vanishing properties of the error $\bs {\ti w}_{n, 0}^J$, but these follow from, e.g., \cite[Lemmas 10 and 11]{CKS} after noting the correspondence between the linear wave equation~\eqref{eq:lin-4d} and the $2k+2$-dimensional radially symmetric free wave equation (see also~\cite[Claim A.1 and Claim 2.11]{DKM1} for the treatment of the wave equation in odd dimensions). 

Assume for the sake of contradiction that there exists a nonzero profile $(\bs v\lin^{i_0}, \s_{n, i_0}, t_{n, i_0})$ with index $i_0 \in \calJ_{\Lin, 0}$, and assume without loss of generality that 
\EQ{
\frac{-t_{n, i_0}}{\s_{n, i_0}} \to + \infty \mas n \to \infty. 
}
Using~\eqref{eq:maximal} and~\eqref{eq:tiu-fsp} we have 
\EQ{\label{eq:maximal-ti}
&\lim_{ n \to \infty} \sup_{I \ni t_n;  I \subset [0,  1]}  \frac{1}{\abs{I}} \int_I \int_0^{\infty} | \p_t \ti u_n(t, r)|^2 \, r \, \ud r \, \ud t = 0,\\  
&\| \p_t {\ti u}_n(t_n) \|_{L^2} \to 0 \mas n \to \infty, 
}
and we can apply (after passing to a subsequence) Lemma~\ref{lem:sym-profile} to deduce the existence of a matching profile $(\bs v\lin^{i_1}, \s_{n, i_1}, t_{n, i_1})$ such that for all $s \in \R$, 
\EQ{
 v_{\Lin}^{i_1}(s) =  v_{\Lin}^{i_0}(-s), \quad  \s_{n, i_1} = \s_{n, i_0}, \mand t_{n, i_1} = - t_{n, i_0} \, \, \forall n. 
}
After relabeling we may assume that $i_1 = i_0 + 1$. 

We claim that there exists $\tau_0>0$ so that, in addition to~\eqref{eq:maximal-ti}, we also have, 
\EQ{ \label{eq:tau_0}
\lim_{n \to \infty} \| \p_t{ \ti u}_n(t_n +  \tau_0 \s_{n, i_0}) \|_{L^2} = 0
} 
To see this, assume for simplicity that $t_n \to t_0 <1$ (the other possible scenarios are similar). Passing to a subsequence, we may assume that 
\EQ{
2^{-2n-4} \ge   \sup_{I \ni t_n;  I \subset [0,  1]}  \frac{1}{\abs{I}} \int_I \int_0^{\infty} | \p_t \ti u_n(t, r)|^2 \, r \, \ud r \, \ud t , 
}
and define sets $E_n$ (for all large $n$) via, 
\EQ{
E_n := \{ \tau \in [0, 1] \,\,: \, \,   \| \p_t { \ti u}_n(t_n +  \tau \s_{n, i_0}) \|_{L^2}^2 \ge 2^{-n-2} \}.
}
Thus,  
\EQ{
2^{-2n-4}  \ge \frac{1}{\s_{n, i_0}} \int_{t_n}^{t_n + \s_{n, i_0}}  \| \p_t \ti u_n( t ) \|_{L^2}^2 \, \ud t = \int_0^1 \| \p_t{ \ti u}_n(t_n +  \tau \s_{n, i_0}) \|_{L^2 }^2 \, \ud \tau  \ge \abs{E_n} 2^{-n-2}
}
which means that $\abs{E_n} \le 2^{-n-2}$ for all $n$ large enough. Hence $\abs{  \cup_{n \ge 0} E_n } \le \frac{1}{2}$, and thus any $\tau_0 \in [0, 1] \setminus \cup_{n \ge 0} E_n $ satisfies~\eqref{eq:tau_0}. 

Next, we will need to evolve the profiles for  time $ =  \tau_0\s_{n, i_0}$. To get in the setting of Lemma~\ref{lem:nlpd} we first need to truncate the sequence again, removing all profiles concentrating at a scales $\ll \s_{n, i_0}$. To this end, and following~\cite[Erratum]{DKM3erratum}, we denote by $\calK  = \calK_{s, 0} \cup \calK_{\Lin, 0}$, where $\calK_{c, 0} \subset \calJ_{c, 0}$ is the set of indices $j$ such that, 
\EQ{
\exists C_j >0 \quad \textrm{such that} \quad \lam_{n, j} \le C_k \s_{n, i_0} , 
}
and letting $\eps_0>0$ be as in Lemma~\ref{lem:Cauchy}, $\calK_{\Lin, 0} \subset \calJ_{\Lin, 0}$ is the set of indices $i$ such that both 
\EQ{
E( \bs v_{\nl}^i) \ge \eps_0 \mand \exists C_i >0 \quad \textrm{such that} \quad \max( \s_{n, i},  \abs{t_{n, i}} ) \le C_i \s_{n, i_0}
}
Observe that $i_0, i_0 +1 \not \in \calK_{\Lin, 0}$ and that by the pythagorean expansion of the nonlinear energy, $\calK$ is a finite set. 

Since $\s_{n, i_0} \ll \abs{t_{n, i_0}}$ we can, arguing as in~\eqref{eq:r2},  find a scale $\s_n$ such that $\s_{n, i_0} \ll \s_n \ll \abs{t_{n, i_0}}$ and such that $E( \bs{\ti u}_n; \s_n/4, 4\s_n) \to 0$ as $n \to\infty$. Using Lemma~\ref{lem:pi}, and arguing as above, after passing to a subsequence we can find a sequence $y_n \in [\frac{3}{4} \s_n, \frac{5}{4}\s_n]$ and an integer $\ell_1$ with 
$\abs{ \bs{\ti u}_n(t_n, y_n)-\ell_1 \bs \pi} \to 0$. We then define a sequence $\bs{\check u}_n(t_n) \in \E_{\ell_1, m_1}$ by 
\EQ{
\bs {\check u}_n(t_n) := \chi_{\sigma_n} \ell_1 \bs \pi + (1- \chi_{\sigma_n}) \bs {\ti u}_n(t_n) 
}
It follows  that for any $J \ge \max(i; i\in \calK_{\Lin, 0})+ 1$ we have
\EQ{
&\bs{\check{ u}}_n(t_n) = m_1 \bs \pi +  \sum_{j \in\calJ_{c, 0}\setminus \calK_{c, 0}} (  \iota_j Q \big(  \frac{ \cdot}{\lam_{n, j}} \big)  , 0) -  \bs \pi)   + \sum_{i \le J, \,  \, i \in \calJ_{\Lin, 0}\setminus \calK_{\Lin, 0}}  \bs v_{\Lin, n}^j(0) +  \bs{\check w}_{n, 0}^J  +o_n(1) 
}
where we define $\check w_{n, 0}^J(r) =  (1- \chi_{\sigma_n})\ti w_{n, 0}^J(r)$. 
We need to justify the $o_n(1)$ term above. First, it is clear the harmonic maps with indices $j \in \calK_{c, 0}$ satisfy $\|(1- \chi_{\s_n/2})(  \iota_j Q \big(  \frac{ \cdot}{\lam_{n, j}} \big)  , 0) -  \bs \pi)\|_{\E} = o_n(1)$ since $j \in \calK_{c, 0}$ implies $\lam_{n, j} \le C_j \s_{n, i_0} \ll \s_{n}$. Next for those indices $i \in \calK_{\Lin, 0}$ we claim that,  
\EQ{ \label{eq:exterior-vanishing} 
\| (1- \chi_{\s_n/2}) \bs v_{\Lin}^i( - t_{n, i}/ \s_{n, i}) \|_{\E} \lesssim \| \bs v_{\Lin}^i( - t_{n, i}/ \s_{n, i}) \|_{\E( r \ge \frac{1}{2} \s_n/ \s_{n, i})} = o_n(1)
}
To prove the last inequality above note that since  $i \in\calK_{\Lin, 0}$ we have 
\EQ{
\frac{1}{2} \frac{\s_n}{\s_{n, i}} = \frac{1}{2} \frac{\s_n}{\s_{n, i_0}} \frac{\s_{n, i_0}}{\s_{n, i}}  \ge \frac{1}{2C_i}  \frac{\s_n}{\s_{n, i_0}}\frac{\abs{t_{n, i}}}{\s_{n, i}} 
}
and now~\eqref{eq:exterior-vanishing} follows from~\cite[Lemma 9]{CKS} after noting again that $\s_n/ \s_{n, i_0} \to \infty$ as $n\to \infty$, and using the equivalence between~\eqref{eq:lin-2d} and~\eqref{eq:lin-2k+2} outlined in Section~\ref{sec:Cauchy}.  

Note that 
\EQ{
\bs { \check u}_n(t_n, r) = \bs{\ti u}_n(t_n, r) = \bs u_n( t_n, r) \mif 4 \s_n \le r \le  r_{2, n}
}
and thus, denoting by $\bs{ \check u}_n(t)$ the wave map evolution of $\bs{ \check u}_n$ we have by finite speed of propagation that for $s>0$, 
\EQ{ \label{eq:checku-fsp} 
\bs{ \check u}_n(t_n + s, r) = \bs{\ti u}_n(t_n + s,  r) = \bs u_n(t_n + s, r) \mif 4\s_n + s \le r \le  r_{2, n} - s.
}
 The point of these truncations is that we can now apply the nonlinear profile decomposition Lemma~\ref{lem:nlpd} to $\bs{\check u}_n(0)$ up to time $\tau_0 \s_{n, i_0}$,  obtaining an error term $\bs z_{n}^J(t)$ satisfying for all $s\in [0, \tau_0 \s_{n, i_0}]$, 
\EQ{
&\bs{\check u}_n(t_n + s)  = m_1 \bs \pi +  \sum_{j \in\calJ_{c, 0}\setminus \calK_{c, 0}} (  \iota_j Q \big(  \frac{ \cdot}{\lam_{n, j}} \big)  , 0) -  \bs \pi)   + \sum_{i \le J, \,  \, i \in \calJ_{\Lin, 0}\setminus \calK_{\Lin, 0}}  \bs v_{\nl, n}^j(s) +  \bs{\check w}_{n}^J(t) + \bs z_{n}^J(t) \\
&\lim_{J \to \infty} \limsup_{n \to \infty} \Big( \sup_{t \in[0, \tau_0 \s_{n, i_0}]}\| \bs z_{n}^J(t) \|_{\E}  + \|  z_{n}^J\|_{\calS([0, \tau_0 \s_{n, i_0}])} \Big) = 0. 
}
Next observe that plugging in $s=\tau_0\s_{n, i_0}$ above gives rise to  \emph{linear} profile decomposition for $\bs{\check u}_n( t_n + \tau_0 \s_{n, i_0})$ in the sense of Lemma~\ref{lem:pd}, where the profiles are given by $(\bs Q, \lam_{n, j})$ and $(\bs {\ti v}_{\Lin}^i, \ti \s_{n, i}, \ti t_{n, i}) = ( \bs v_{\Lin}^i, \s_{n, i}, t_{n, i}- \tau_0 \s_{n, i_0})$. In particular $\ti v_{\Lin}^{i_0}(t) =  \ti v_{\Lin}^{i_0 +1}(-t)$. 

We apply Lemma~\ref{lem:sym} to the sequence, 
\EQ{
(f_n, g_n)&:= \big( \check u_n( t_n + \tau_0 \s_{n, i_0}, \s_{n, i_0 }\cdot), \s_{n, i_0} \p_t \check u_n( t_n + \tau_0 \s_{n, i_0}, \s_{n, i_0} \cdot) \big) \\
&\quad  - m_1 \bs \pi -  \sum_{j \in\calJ_{c, 0}\setminus \calK_{c, 0}} (  \iota_j Q \big(  \frac{ \s_{n, i_0}}{\lam_{n, j}}  \cdot \big)  , 0) -  \bs \pi) 
}
with $\al_n= 4 \frac{\s_n}{\s_{n, i_0}} + \tau_0$ and $s_n = \frac{t_{n, i_0}}{\s_{n, i_0}}$. By~\eqref{eq:checku-fsp} and~\eqref{eq:tau_0} we have $\| g_n \|_{L^2(r \ge \al_n)} \to 0$ as $n \to \infty$. Since $\s_n \ll \abs{t_{n, i_0}}$ we also have $\frac{\abs{s_n}}{\al_n} \to \infty$ as $n \to \infty$. Hence we may apply Lemma~\ref{lem:sym}. On the one hand, by the way the profiles are obtained, 
\EQ{
S\lin(s_n)(f_n, g_n) = S\lin(\frac{t_{n, i_0}}{\s_{n, i_0}} )(f_n, g_n)  = S\lin( \tau_0) S\lin( \frac{\ti t_{n, i_0}}{\ti \s_{n, i_0}}) (f_n, g_n) \rightharpoonup (v_{\Lin}^{i_0}(\tau_0), \p_t v\lin^{i_0}( \tau_0) )\in \E
}
but on the other hand, since $\ti t_{n, i_0+1} = -t_{n, i_0} -\tau_0 \s_{n, i_0}$ and since $\s_{n, i_0} = \s_{n, i_0 +1} = \ti \s_{n, i_0 +1}$ we have 
\EQ{
&S\lin(-s_n)(f_n, g_n) = S\lin(-\frac{t_{n, i_0}}{\s_{n, i_0}} )(f_n, g_n)  \\
&= S\lin(\tau_0)S\lin(-\frac{\ti t_{n, i_0+1}}{\ti \s_{n, i_0+1}} )(f_n, g_n)  \rightharpoonup (v_{\Lin}^{i_0+1}(\tau_0), \p_t v\lin^{i_0+1}( \tau_0) ) = (v_{\Lin}^{i_0}(-\tau_0), -\p_t v\lin^{i_0}(- \tau_0) ) \in \E
}
An application of Lemma~\ref{lem:sym} then gives $v_{\Lin}^{i_0}(\tau_0) = v_{\Lin}^{i_0}(-\tau_0)$ and $\p_t v\lin^{i_0}( \tau_0) = \p_t v\lin^{i_0}(- \tau_0) $, or in other words $ v_{\Lin}^{i_0}(t) = v\lin^{i_0}( t+ 2 \tau_0)$, is periodic with period $2 \tau_0$, which is impossible since $\bs v\lin^{i_0}(t)$ is a finite energy solution to~\eqref{eq:lin-2d}, unless $\bs v\lin^{i_0} \equiv 0$, which contradicts our assumption. Thus, there are no nonzero profiles with indices in the set $\calJ_{\Lin, 0}$.


%

\textbf{Step 8.}(Vanishing properties of the error $\bs {w}_{n, 0}^J$) We summarize where the argument stands after all of the previous steps. With $\bs{\ti u}_n$ defined in~\eqref{eq:profiles2}, we may 
relabel the indices in $\calJ_{c, 0}$,   so that  $\vec \lam_n= (\lam_{n, 1} , \dots \lam_{n, K_0})$ with  $0 \le K_0 \le K_1$, and with $\lam_{n, 1} \ll \lam_{n, 2} \ll \dots \ll \lam_{n, K_0} \lesssim 1$ and signs $\vec \iota = ( \iota_1, \dots, \iota_{K_0})$, so that 
\EQ{ \label{eq:profiles-2} 
\bs{\ti u}_n(t_n) &= m_1 \bs \pi + \sum_{j =1}^{K_0} (  \iota_j Q \big(  \frac{ \cdot}{\lam_{n, j}} \big)  , 0) -  \bs \pi)  + \bs{ \ti w}_{n, 0}. 
}
where we have removed the index  $J$ in $\bs{ \ti w}_{n, 0}^J$, using the previous step since there are no nonzero outgoing/incoming profiles relevant to the region $r \le r_{2, n}$. 
It will suffice to show the existence of a sequence $r_n  \to \infty$, with $r_n \le r_{n, 2}$ so that after passing to a subsequence, we have 
\EQ{ \label{eq:w-vanishes} 
\| \bs{\ti w}_{n, 0} \|_{\E(r \le r_n)} \to 0 \mas n \to \infty. 
}
Using the pythagorean expansion of the energy, we conclude from~\eqref{eq:maximal-ti} that 
\EQ{
\| \dot{ \ti w}_{n, 0} \|_{L^2} \to 0 \mas n \to \infty.
}
We also have directly from Lemma~\ref{lem:pd} that, 
\EQ{
\| \ti w_{n, 0} \|_{L^\infty} \to 0  \mas n \to \infty.
}
After passing to a subsequence of the $\bs u_n$, we claim there is a sequence $r_n \to \infty$ with the following properties, 
\EQ{\label{eq:rn} 
1 \ll r_n \le \min(r_{2, n}, \ti R_{1, n}) , \quad \|\bs{\ti w}_n\|_{\E(  r_n^{-1} \le r \le 2r_n)} \to 0 \mas n \to \infty, 
}
where $R_{1, n}$ is as in Steps 1. and 2. Indeed, arguing as in Step. 5., we see that for any sequence $\lam_n \lesssim 1$ and any $A>1$ we have, 
\EQ{ \label{eq:every-scale} 
\|\bs{\ti w}_n\|_{\E(  \lam_n A^{-1} \le r \le  \lam_n A)} \to 0 \mas n \to \infty, 
}
see for example ~\cite[Step 2., p.1973-1975, Proof of Theorem 3.5]{Cote15} or~\cite[Proof of (5.29) in Theorem 5.1]{JK} for this conclusion in those analogous settings. Then,  considering the case $\lam_n =1$ above and passing to a subsequence of the $\bs{\ti u}_n$,   we obtain a sequence as in~\eqref{eq:rn}. 

Using the selection of $r_n$ in the previous line, we see from~\eqref{eq:jia-kenig} that, in addition to~\eqref{eq:maximal-ti},  $\bs{\ti u}_n$ satisfies 
\EQ{ \label{eq:jia-kenig-tiu} 
\limsup_{n\to \infty}\int_0^\infty\bigg( k^2 \frac{\sin^2(2\ti u_n(t_n))}{2r^2}- (\partial_r^2 \ti u_n(t_n) + \frac 1r \partial_r \ti u_n(t_n))\sin(2\ti u_n)\bigg)\chi(\cdot/ r_{ n})\,r\vd r \leq 0, 
}
Integration by parts of the second term in the integrand above yields, 
\EQ{
- \int_0^\infty(\partial_r^2 \ti u_n + \frac 1r \partial_r \ti u_n)\sin(2\ti u_n)\chi(\cdot/ r_{ n})\,r\vd r  &= \int_0^\infty (\p_r \ti u_n)^2 2 \cos (2 \ti u_n) \chi( \cdot/ r_n) \,r \,  \ud r  \\
&  \quad + \int_0^\infty \p_r \ti u_n \sin(2 \ti u_n) \frac{1}{r_n} \chi'(  \cdot/ r_n) \, r \, \ud r . 
}
The second term on the right above satisfies, 
\EQ{
\abs{\int_0^\infty \p_r \ti u_n \sin(2 \ti u_n) \frac{1}{r_n} \chi'(  \cdot/ r_n) \, r \, \ud r } &\lesssim E(( \ti u_n, 0); \frac{r_n}{2}, 2r_n)  \to 0 \mas n \to \infty \\
}
 by our selection of $r_n$. 
From the above and~\eqref{eq:jia-kenig-tiu} it follows that
\EQ{ \label{eq:jk-ineq} 
\limsup_{n\to \infty}\int_0^\infty \bigg( k^2 \frac{\sin^2(2\ti u_n(t_n))}{2r^2} + (\p_r \ti u_n(t_n))^2 2 \cos (2 \ti u_n(t_n))  \bigg)\chi( \cdot/ r_n) \,r \,  \ud r    \le 0 . 
}
We now use the second assumption in~\eqref{eq:rn}, in particular the fact that it implies $$\lim_{n \to \infty}E(\bs{\ti u}_n; r_n/4, 4r_n) =0,$$ because all the $\lam_{n, j}$ are bounded,  to truncate the sequence $\bs{\ti u}_n(t_n)$ yet again, obtaining a new sequence $\bs{\ti{\ti u}}_n$ and corresponding wave map evolutions $\bs{\ti{\ti u}}_n(t)$ on the interval $[0, 1]$, such that 
\EQ{\label{eq:titi} 
\bs{\ti{\ti u}}_n(t, r ) = \bs {\ti u}_n(t, r) = \bs u_n(t, r)  \mif r \le r_n, \mand  \lim_{n \to \infty} E( \bs{\ti {\ti u}}_n; r_n, \infty) = 0
}
Using the above along with~\eqref{eq:jk-ineq} we obtain the following global non-positivity of the Jia-Kenig virial functional for $\bs{\ti {\ti u}}_n$, 
\EQ{ \label{eq:jia-kenig-final} 
\limsup_{n\to \infty}\int_0^\infty \bigg( k^2 \frac{\sin^2(2\ti {\ti u}_n(t_n))}{2r^2} + (\p_r \ti {\ti u}_n(t_n))^2 2 \cos (2 \ti {\ti u}_n(t_n))  \bigg) \,r \,  \ud r    \le 0 
}
From~\eqref{eq:maximal-ti} and~\eqref{eq:titi} we obtain, 
\EQ{\label{eq:maximal-titi}
&\lim_{ n \to \infty} \sup_{I \ni t_n;  I \subset [0,  1]}  \frac{1}{\abs{I}} \int_I \int_0^{\infty} | \p_t \ti {\ti u}_n(t, r)|^2 \, r \, \ud r \, \ud t = 0,\\  
&\| \p_t {\ti {\ti u}}_n(t_n) \|_{L^2} \to 0 \mas n \to \infty, 
}
and finally from~\eqref{eq:profiles-2} we see that $\bs {\ti{\ti u}}_n(t_n)$ satisfies, 
\EQ{
\bs{\ti {\ti u}}_n(t_n) &= m_1 \bs \pi + \sum_{j =1}^{K_0} (  \iota_j Q \big(  \frac{ \cdot}{\lam_{n, j}} \big)  , 0) -  \bs \pi)  + \bs{\ti{\ti{w}}}_{n, 0}
}
with $\|\bs{\ti{\ti{w}}}_{n, 0} \|_{L^\infty \times L^2} \to 0$ and $\| \bs{\ti{\ti w}}_{n, 0} \|_{\E(r \ge  r_n^{-1})} \to 0$ as $n \to \infty$, and all the $\lam_{n, j}$ satisfy $\lam_{n, j} \lesssim 1$. Moreover, by~\eqref{eq:every-scale} we have the vanishing $\| \bs{\ti{\ti w}}_{n, 0} \|_{\E(A^{-1} \lam_n \le r \le A \lam_n)} \to 0$ as $n \to \infty$ for any sequence $\lam_n \lesssim 1$ and any $A>1$. 

We have now reduced to a setting that is completely analogous to~\cite[Proof of Theorem 3.2]{JK} and one may argue precisely as in that paper to conclude, via Lemma~\ref{lem:jk},  that 
\EQ{
\| \bs{\ti{\ti w}}_n \|_{\E} \to 0 \mas n \to \infty. 
}
 We have shown that $\bs \de_{\infty}( \bs{\ti{\ti u}}_n(t_n)) \to 0$ as $n \to \infty$.  By~\eqref{eq:titi} we in fact have proved that $\bs\de_{r_n}( \bs u_n(t_n)) \to 0$ as $n \to \infty$, (note that  $\lam_{n, K_0} \lesssim 1$ and $r_n \to \infty$ ensures that the final ratio $\lam_{n, K_0}/ \lam_{n, K_0 +1} \to 0$, where $\lam_{n, K_0 +1} := r_n$), completing the proof. 
\end{proof} 


\section{Decomposition of the solution and collision intervals} \label{sec:decomposition} 

In the final two sections we prove Theorem~\ref{thm:main} for equivariance classes $k \ge2$. We reserve the case $k=1$ for the appendix. 

\subsection{Proximity to a multi-bubble and collisions}
\label{ssec:proximity}
For the remainder of the paper we fix a solution
$\bs u(t) \in \E_{\ell, m}$ of \eqref{eq:wmk},
defined on the time interval $I_*=(0, T_0]$
in the blow-up case and on $I_*=[T_0, \infty)$ in the global case, for some $T_0 > 0$. 
We set $T_* := \infty$ in the global case and $T_* := 0$ in the blow-up case.
Let $\bs u^*(t)$ be the radiation as defined in Theorem~\ref{thm:stz}. More precisely, we let $m_\Delta := \lim_{t \to T_+}u(t, \frac 12 t) \in \bZ$
and shift the radiation so that $\bs u^*(t) \in \cE_{0, m_\infty}$ for some $m_\infty \in \bZ$,
and for $r \gtrsim t$, $\bs u(t, r) \sim m_\Delta\bs\pi + \bs u^*(t, r)$.
Note that $m_\infty = 0$ if $T_* = \infty$.

It is a crucial insight of \cite{CKLS1, CKLS2, Cote}
that $\bs u^*(t)$ is given for continuous time. Recall that Theorem~\ref{thm:stz} gives a function $\rho: I_* \to (0, \infty)$ such that 
\EQ{ \label{eq:rho-def} 
&\lim_{t \to T_*} \big((\rho(t) / t)^k + \|\bs u(t) -  \bs u^*(t) - m_\Delta \bs\pi\|_{\cE(\rho(t), \infty)}^2\big) = 0,
}
and that for any $\alpha \in (0, 1)$ we have
\EQ{
\label{eq:energy-of-ustar}
\lim_{t \to T_*} E(  \bs u^*(t); 0, \al t) = 0.
}

By Theorem~\ref{thm:seq} there exists a time sequence $t_n \to T_*$ and an integer $N \ge 0$, which we now fix,  such that $\bs u(t_n) - \bs u^*(t_n)$ approaches an $N$-bubble as $n \to \infty$.  Roughly, our goal is to show that on the region $r \in (0, \rho(t))$, the solution $\bs u(t)$ approaches a continuously modulated $N$-bubble, noting that the radiation $\bs u^*(t)$ is negligible in this region. By convention, we will set $\lambda_{N+1}(t) := t$ to be the ``scale'' of the radiation and $\lambda_0(t) := 0$. Our argument requires the following localized version of the distance function to a multi-bubble.

\begin{defn}[Proximity to a multi-bubble]
\label{def:proximity}
For all $t \in I$, $\rho \in (0, \infty)$, and $K \in \{0, 1, \ldots, N\}$, we define
the \emph{localized multi-bubble proximity function} as
\begin{equation}
\bfd_K(t; \rho) := \inf_{\vec \iota, \vec\lam}\bigg( \| \bs u(t) - \bs u^*(t) - \bs\calQ(m_\Delta, \vec\iota, \vec\lambda) \|_{\cE(\rho, \infty)}^2 + \sum_{j=K}^{N}\Big(\frac{ \lam_{j}}{\lam_{j+1}}\Big)^{k} \bigg)^{\frac{1}{2}},
\end{equation}
where $\vec\iota := (\iota_{K+1}, \ldots, \iota_N) \in \{-1, 1\}^{N-K}$, $\vec\lambda := (\lambda_{K+1}, \ldots, \lambda_N) \in (0, \infty)^{N-K}$, $\lambda_K := \rho$ and $\lambda_{N+1} := t$.

The \emph{multi-bubble proximity function} is defined by $\bfd(t) := \bfd_0(t; 0)$.
\end{defn}

\begin{rem} 
We emphasize that if $\bfd_K(t; \rho)$ is small, this means that $ \bs u(t) - \bs u^*(t)$ is close to $N-K$ bubbles in the exterior region  $r \in (\rho, \infty)$. 
\end{rem} 
We can now rephrase Theorem~\ref{thm:seq} in this notation: there exists a monotone sequence $t_n \to T_*$ such that
\begin{equation}
\label{eq:dtn-conv}
\lim_{n \to \infty} \bfd(t_n) = 0.
\end{equation}
Even though this fact is certainly a starting point of our analysis,
it will turn out that we cannot use it as a black box. Rather, we need to examine the proof
and use more precise information provided by the analysis in~\cite{Cote15, JK}; see Section~\ref{sec:compact}. 

We state and prove some simple consequences of the set-up above.
We always assume $N \geq 1$, since the pure radiation case $N = 0$ (in fact,
also the case $N = 1$) is already settled by C\^ote's and Jia's and Kenig's work~\cite{Cote15, JK}.

First, a direct consequence of~\eqref{eq:rho-def} is that $\bs u(t)- \bs u^*(t)$ always approaches a $0$-bubble in some exterior region. With $\rho_N(t) = \rho(t)$ given by the function in Theorem~\ref{thm:stz} the following lemma is  immediate from the conventions of Definition~\ref{def:proximity} 
\begin{lem}
\label{lem:conv-rhoN}
There exists a function $\rho_N: I \to (0, \infty)$ such that
\begin{equation}
\label{eq:conv-rhoN}
\lim_{t\to T_*}\bfd_N(t; \rho_N(t)) = 0.
\end{equation}
\end{lem}

Theorem~\ref{thm:main} will be a quick consequence of showing that, in fact, 
\begin{equation}
\label{eq:dt-conv}
\lim_{t \to T_*} \bfd(t) = 0.
\end{equation}
The approach which we adopt in order to prove~\eqref{eq:dt-conv} it is to study colliding bubbles.
A collision is defined as follows.
\begin{defn}[Collision interval]
\label{def:collision}
Let $K \in \{0, 1, \ldots, N\}$. A compact time interval $[a, b] \subset I_*$ is a \emph{collision interval}
with parameters $0 < \epsilon < \eta$ and $N - K$ exterior bubbles if
\begin{itemize}
\item $\bfd(a) \leq \epsilon$ and $\bfd(b) \leq \epsilon$,
\item there exists $c \in (a, b)$ such that $\bfd(c) \geq \eta$,
\item there exists a function $\rho_K: [a, b] \to (0, \infty)$ such that $\bfd_K(t; \rho_K(t)) \leq \epsilon$
for all $t \in [a, b]$.
\end{itemize}
In this case, we write $[a, b] \in \calC_K(\epsilon, \eta)$.
\end{defn}
\begin{defn}[Choice of $K$]
\label{def:K-choice}
We define $K$ as the \emph{smallest} nonnegative integer having the following property.
There exist $\eta > 0$, a decreasing sequence $\epsilon_n \to 0$
and sequences $(a_n), (b_n)$ such that $[a_n, b_n] \in \calC_K(\epsilon_n, \eta)$ for all $n \in \{1, 2, \ldots\}$.
\end{defn}
\begin{lem}[Existence of $K \ge 1$]
\label{lem:K-exist}
If \eqref{eq:dt-conv} is false, then $K$ is well defined and $K \in \{1, \ldots, N\}$.
\end{lem}

\begin{rem} 
The fact that $K \ge 1$ means that at least one bubble must lose its shape if~\eqref{eq:dt-conv} is false.
\end{rem} 
\begin{proof}[Proof of Lemma~\ref{lem:K-exist}]
Assume \eqref{eq:dt-conv} does not hold, so that there exist $\eta > 0$ and a monotone sequence $s_n \to T_*$ such that
\begin{equation}
\bfd(s_n) \geq \eta, \qquad\text{for all }n.
\end{equation}
We claim that there exist sequences $(\epsilon_n), (a_n), (b_n)$ such that $[a_n, b_n] \in \calC_N(\epsilon_n, \eta)$.
Indeed, \eqref{eq:dtn-conv} implies that there exist $\epsilon_n \to 0$, $a_n \leq s_n$ and $b_n \geq s_n$
such that $\bfd(a_n) \leq \epsilon_n$ and $\bfd(b_n) \leq \epsilon_n$. Note that $a_n \to T_*$ and $b_n \to T_*$.
Let $\rho_N: [a_n, b_n] \to (0, \infty)$ be the function given by Lemma~\ref{lem:conv-rhoN},
restricted to the time interval $[a_n, b_n]$.
Then \eqref{eq:conv-rhoN} yields
\begin{equation}
\lim_{n\to\infty}\sup_{t\in[a_n, b_n]}\bfd_N(t; \rho_N(t)) = 0.
\end{equation}
Upon adjusting the sequence $\epsilon_n$, we obtain that all the requirements of Definition~\ref{def:collision}
are satisfied for $K = N$.

We now prove that $K \geq 1$. Suppose $K = 0$. The definition of a collision interval yields
$\bfd_0(c_n;  \rho_n) \leq \epsilon_n$ for some sequence $\rho_n \ge 0$, and at the same time $\bfd(c_n) \geq \eta$ for some $\eta>0$. We show that this is impossible. 


Define $\bs v_n := \bs u(c_n) - \bs u^*(c_n)$. 
Since $\bfd_0(c_n;  \rho_n) \leq \epsilon_n$ we can find parameters, $\rho_n \ll  \lam_{n, 1} \ll \dots \ll \lam_{n, N}$ and signs $\vec{\iota}_n$ such that defining $\bs { g}_n = \bs v_n - \bs \calQ( m_\De, \vec{ \iota}_n, \vec{ \lam}_n)$ we have
\EQ{ \label{eq:tig-small} 
\bfd_0(c_n; \rho_n) \simeq \| \bs{g}_n \|_{\E( \rho_n, \infty)}^2 + \sum_{j =0}^N\Big( \frac{  \lam_{n, j}}{\lam_{n, j+1}}\Big)^k \lesssim \eps_n^2. 
}
Using~\eqref{eq:rho-def} we see that we must have $\lam_{n, N} \ll  \rho(c_n) \ll c_n$, and thus using~\eqref{eq:rho-def} along with~\eqref{eq:tig-small} and Lemma~\ref{lem:M-bub-energy} we have
\EQ{
E( \bs u(c_n); \rho_n, \infty) &= E( \bs {g}_n + \bs u^*(c_n) + \bs \calQ( m_\De, \vec{ \iota}_n, \vec{ \lam}_n); \rho_n, \rho(c_n))  \\
&\quad + E( \bs { g}_n + \bs u^*(c_n) + \bs \calQ( m_\De, \vec{\iota}_n, \vec{\lam}_n); \rho(c_n), \infty) \\
& = N E( \bs Q) + E( \bs u^*)  + o_n(1) 
}
Since by~\eqref{eq:en-ident} we know that $E(\bs u(c_n)) = N E( \bs Q) + E( \bs u^*(c_n))$, we conclude from the previous line that, 
\EQ{
E( \bs u(c_n); 0, \rho_n) = o_n(1) \mas n \to \infty. 
}
Using \eqref{eq:energy-of-ustar} and the fact that $\rho_n \ll \rho(c_n)$ it follows that $E( \bs v_n; 0, \rho_n) = o_n(1)$, and hence by~\eqref{eq:H-E-comp} we conclude that 
\EQ{
\|\bs v_n - \ell \bs \pi \|_{\E( 0, \rho_n)} \lesssim E( \bs v_n; 0, \rho_n) = o_n(1) \mas n \to \infty
}
Thus, combining the above with~\eqref{eq:tig-small} we have $\bfd(c_n) = o_n(1)$ as $n \to \infty$, a contradiction. 
\end{proof}

In the remaining part of the paper, we argue by contradiction, fixing $K$ to be the number provided by Lemma~\ref{lem:K-exist}.
We also let $\eta, \epsilon_n, a_n$ and $b_n$ be some choice of objects satisfying the requirements of Definition~\ref{def:K-choice}.
We fix choices of signs and scales for the $N-K$ ``exterior'' bubbles provided by Definition~\ref{def:proximity} in the following lemma.  

\begin{rem} \label{rem:collision} 
For each collision interval there exists a time $c_n \in [a_n, b_n]$ with $\bfd(c_n) \ge \eta$ and we may assume without loss of generality that $\bfd(a_n) = \bfd(b_n) = \eps_n$ and $\bfd(t) \ge \eps_n$ for each $t \in [a_n, b_n]$. Indeed, given some initial choice of $[a_n, b_n] \in \calC_K(\eta, \eps_n)$, we can find $a_n \le \ti a_n < c_n$ and $c_n < \ti b_n \le b_n$ so that $\bfd(a_n) = \bfd(b_n) = \eps_n$ and $\bfd(t) \ge \eps_n$ for each $t \in [\ti a_n, \ti b_n]$.  Just set $a_n \le \ti a_n := \inf\{t \le c_n \mid \bfd(t) \ge \eps_n \}$ and similarly for $\ti b_n$. 

Similarly, give some initial choice $\eps_n \to 0, \eta>0$ and intervals $[a_n, b_n] \in \calC_K( \eta, \eps_n)$ we are free to ``enlarge'' $\eps_n$ by choosing some other sequence $\eps_n \le \ti \eps_n  \to 0$, and new collision subintervals $[\ti a_n, \ti b_n]  \subset [a_n, b_n] \cap \calC_{K}(\eta, \ti \eps_n)$ as in the previous paragraph. We will enlarge our initial choice of $\eps_n$ in this fashion several times over the course of the proof. 
\end{rem} 




\begin{lem}\label{lem:ext-sign} 
Let $K \ge 1$ be the number given by Lemma~\ref{lem:K-exist}, and let $\eta, \epsilon_n, a_n$ and $b_n$ be some choice of objects satisfying the requirements of Definition~\ref{def:K-choice}. Then there exists a sequence $\vec\sigma_n \in \{-1, 1\}^{N-K}$,  a  function $\vec \mu = (\mu_{K+1}, \dots, \mu_N)  \in C^1(\cup_{n \in \N} [a_n, b_n] ;  (0, \infty)^{N-K})$, a sequence $\nu_n \to 0$, and a sequence $m_n \in \Z$, so that defining the function, 
\EQ{ \label{eq:nu-def} 
\nu:\cup_{n \in \N} [a_n, b_n] \to (0, \infty), \quad  \nu(t):= \nu_n \mu_{K+1}(t),  
}
we have, 
\EQ{ \label{eq:nu-prop} 
\lim_{n \to \infty}\sup_{t \in [a_n, b_n]} \Big(\bfd_K(t; \nu(t)) + E( \bs u(t), \nu(t), 2 \nu(t)) \Big) = 0, 
}
and defining  $\bs w(t), \bs h(t)$ for  $t \in \cup_n [a_n, b_n]$ by 
\EQ{ \label{eq:w_n} 
\bs w(t)= (1 - \chi_{\nu(t)})( \bs u(t) - \bs u^*(t)) + \chi_{\nu(t)} m_n \bs \pi  &= m_{\De} \bs \pi +  \sum_{j = K+1}^N  \s_{n, j} ( \bs Q_{\mu_{ j}(t)} - \bs \pi) + \bs h(t), 
}
we have, $\bs w(t) \in \E_{m_n, m_{\De}}$, $\bs h(t) \in \E$, and 
\EQ{ \label{eq:mu_n} 
\lim_{n \to \infty} \sup_{t \in [a_n, b_n]} \Big(\| \bs h(t) \|_{\E}^2 + \Big(\frac{\nu(t)}{ \mu_{K+1}(t)} \Big)^k +  \sum_{j = K+1}^{N} \Big( \frac{\mu_{ j}(t)}{\mu_{j+1}(t)} \Big)^k \Big) = 0,   
}
with the convention that $\mu_{N+1}(t) = t$. Finally, $\nu(t)$ satisfies the estimate, 
\EQ{ \label{eq:nu'} 
\lim_{n \to \infty} \sup_{t \in [a_n, b_n]}\abs{\nu'(t) } = 0.
}

\end{lem}

\begin{rem} 
One should think of $\nu(t)$ as the scale that separates the $N-K$ ``exterior'' bubbles, which are defined continuously on the union of the collision intervals $[a_n, b_n]$  from the $K$ ``interior'' bubbles that are coherent at the endpoints of $[a_n, b_n]$, but come into collision somewhere inside the interval and lose their shape. In the case $K =N$, there are no exterior bubbles,  $\mu_{K+1}(t) = t$ and $\nu_n \to 0$ is chosen using~\eqref{eq:rho-def}. 
\end{rem}

\begin{proof}
By Definition~\ref{def:proximity} for each $n$ we can find scales $\rho_K(t) \ll  \ti \mu_{K+1}(t)  \ll \dots \ll \ti  \mu_{N}(t) \ll t $ and signs $\vec \s(t) \in \{-1, 1\}^{N-k}$  for $t \in [a_n, b_n]$, such that defining $\bs {h}_{\rho_K}(t)$ for $r \in ( \rho_K(t), \infty)$ by 
\EQ{
\bs u(t) - \bs u^*(t) = \bs \calQ (m_{\De}, \vec \s(t), \vec{\ti  \mu}(t)) + \bs{h}_{\rho_K}(t) 
}
we have, 
\EQ{ \label{eq:timu-small} 
\bfd(t; \rho_K(t)) \simeq \| \bs{ h}_{\rho_K}(t) \|_{\E( \rho_K(t), \infty)}^2 + \sum_{j =K}^N \Big( \frac{\ti \mu_{j}(t)}{ \ti \mu_{j+1}(t)} \Big)^k  \lesssim \eps_n^2 , 
}
keeping the convention $\ti \mu_K(t) := \rho_K(t), \ti \mu_{N+1}(t) := t$.   
Using $\lim_{n \to \infty} \sup_{t \in [a, b]}\bfd_K( t; \rho_K(t)) = 0$ and the fact that 
\EQ{ \label{eq:ext-bubble-en} 
\lim_{n \to \infty} \sup_{t \in [a_n, b_n]}E( \bs \calQ (m_{\De}, \vec \s(t), \vec{ \ti \mu}(t)) ; \al_n \ti \mu_{K+1}(t), \be_n \ti \mu_{K+1}(t)) = 0, 
}
for any two sequence $\al_n \ll \be_n \ll 1$, 
we can choose a sequence $\nu_n \to 0$ with 
\EQ{
\rho_{K}(t) \le \nu_n \ti \mu_{K+1}(t), \mand  \lim_{n \to \infty}\sup_{t \in [a_n, b_n]} E( \bs u(t) - \bs u^*(t); \frac{1}{4}\nu_n \ti \mu_{K+1}(t), 4 \nu_n \ti   \mu_{K+1}(t)) =  0. 
}
Defining $\ti \nu(t)= \nu_n \ti \mu_{K+1}(t)$,  it follows from Lemma~\ref{lem:pi} that we can find  integers $m_n \in \Z$, which are independent of $t \in [a_n, b_n]$ due to continuity of the flow so that 
\EQ{ \label{eq:timu-vanish} 
&\lim_{n \to \infty} \sup_{t\in [a_n, b_n]} \sup_{r \in (\frac{1}{4}\nu_n \ti  \mu_{K+1}(t), 4 \nu_n \ti  \mu_{K+1}(t))}  | u(t, r) - u^*(t, r)- m_n \pi| = 0, \\
&\lim_{n \to \infty} \sup_{t\in [a_n, b_n]}\| \bs u(t) - \bs u^*(t) - m_n \pi \|_{\E((\frac{1}{4}\nu_n \ti  \mu_{K+1}(t), 4 \nu_n \ti  \mu_{K+1}(t))}  = 0
}
Thus, defining $\bs {\ti w}(t) \in\E_{m_n, m_{\De}}$ and,  $\bs{\ti h}(t) \in \E$  for $t \in \cup_n [a_n, b_n]$, by   
\EQ{ \label{eq:tiw-def} 
\bs{\ti w}(t)&:= (1 - \chi_{\ti \nu(t)})( \bs u(t) - \bs u^*(t)) + \chi_{\ti \nu(t)} m_n \bs \pi  \\
& = (1- \chi_{\ti \nu(t)}) m_{\De} \bs \pi  + \chi_{\ti \nu(t)} \sum_{j= K+1}^N \sigma_{ j}(t) \bs \pi  + \chi_{\ti \nu(t)} m_n \bs \pi   + \sum_{j = K+1}^N  \s_{ j}(t)( \bs Q_{\ti \mu_{ j}(t)} - \bs \pi) + \bs{\ti h}(t)  \\
& = m_{\De} \bs \pi +  \sum_{j = K+1}^N  \s_{ j}(t) ( \bs Q_{\ti {\mu}_{ j}(t)} - \bs \pi) + \bs{\ti h}(t)
}
we have using~\eqref{eq:timu-small}, 
\EQ{\label{eq:tih-eps} 
\sup_{t \in [a_n,  b_n]} \Big(\| \bs{\ti  h}(t)\|_{\E}^2 + \sum_{j =K}^N \Big( \frac{\ti \mu_{j}(t)}{ \ti \mu_{j+1}(t)} \Big)^k \Big)\le \theta_n^2. 
}
for some sequence $\te_n \to 0$.  We note that the last equality in~\eqref{eq:tiw-def} follows from the observation that we must have, 
\EQ{ \label{eq:m_De-sigma} 
m_{\De}-  \sum_{j=K+1}^N \sigma_{ j}(t)  = m_n
}
for any $t \in [a_n, b_n]$. 
We invoke Lemma~\ref{lem:bub-config} and continuity of the flow to conclude that for each $n$, the sign vector $\vec \s(t) = \vec \s_n$ is independent of $t \in [a_n, b_n]$,   and the functions $\ti \mu_{K+1}(t), \dots,\ti  \mu_{N}(t)$ can be adjusted to be continuous functions of $t$. However, in the next sections we require differentiability of the function $\ti \mu_{K+1}(t)$, so we must modify it slightly. 

Given a vector $\vec \mu(t) = (\mu_{K+1}(t), \dots \mu_N(t))$, set, 
\EQ{
\bs w( t, \vec \mu(t)) := (1- \chi_{\nu_n\mu_{K+1}(t)})( \bs u(t) - \bs u^*(t)) + \chi_{\nu_n\mu_{K+1}(t)} m_n \bs \pi
}
Fixing $t$ and  suppressing it in the notation, and setting up for an argument as in the proof of Lemma~\ref{lem:mod-static}, define 
\EQ{
F(h, \vec \mu) := h  - ( w( \cdot,  \vec{\ti \mu}) -  \calQ( m_{\De}, \vec \s_n, \vec{\ti \mu}) ) + w( \cdot, \vec \mu) - \calQ( m_\De, \vec \s_n, \vec \mu) 
}
and note that $F(0, \vec{\ti \mu}) = 0$. Moreover, 
\EQ{
\| F(h, \vec \mu)\|_H \lesssim \| h \|_{H} + \sum_{j= K+1}^N \abs{ \frac{\mu_j}{\ti \mu_j} - 1} 
}
Define, 
\EQ{
G( h, \vec \mu) := \Big( \frac{1}{\mu_{K+1}} \ang{ \calZ_{\U{\mu_{K+1}}} \mid F(h, \vec \mu)}, \dots , \frac{1}{\mu_N} \ang{ \calZ_{\U{\mu_N}} \mid F(h, \vec \mu)} \Big)
}
and thus $G(0, \vec{\ti \mu}) = (0, \dots, 0)$. Following the same scheme as the proof of Lemma~\ref{lem:mod-static} we obtain via Remark~\ref{rem:IFT} a mapping  $\varsigma: B_H(0; C_0 \te_n) \to (0, \infty)^{N-K}$ such that for each $h \in B_H(0; C_0 \te_n)$ we have 
\EQ{
\abs{ \varsigma_j(h)/ \ti \mu_j - 1} \lesssim \te_n
}
and such that
\EQ{
G( h, \vec \mu) = 0 \Longleftrightarrow \vec \mu = \varsigma(h)
}
Using~\eqref{eq:tih-eps} we define
\EQ{
h:= F(\ti h, \varsigma(\ti h)), \quad \vec \mu:= \varsigma(\ti h)
}
By construction we then have, 
\EQ{
\bs w(t, \vec \mu(t)) &= (1- \chi_{\nu(t)})( \bs u(t) - \bs u^*(t)) + \chi_{\nu(t)} m_n \bs \pi \\
& = \bs \calQ( m_{\De}, \vec \s_n, \vec \mu(t)) + \bs h(t) 
}
for $\nu(t):= \nu_n \mu_{K+1}(t)$, 
and for each $j = K+1, \dots, N$, 
\EQ{ \label{eq:h-small} 
\sup_{t \in [a_n, b_n]} \Big(\| \bs h(t) \|_{\E}^2 + \sum_{j=K}^N \Big( \frac{\mu_{j}(t)}{\mu_{j+1}(t)} \Big)^k\Big) \lesssim  \te_n^2 , \quad 
0 = \ang{ \calZ_{\U{\mu_j(t)}} \mid h(t)} 
}
Note that~\eqref{eq:nu-prop} follows from the above and from~\eqref{eq:rho-def}. 
The point is that we can now use orthogonality conditions above to deduce the differentiability of $\mu(t)$.  Indeed, noting the identity, 
\EQ{
\p_t h(t) &= \p_t w(t, \vec \mu(t)) - \p_t \calQ( m_\De, \vec \s_n, \vec \mu(t)) \\
& = \nu_n\mu_{K+1}'(t)  \Lam \chi_{\U{\nu(t)}} \big( u(t) - u^*(t) - m_n \pi \big) + \dot h(t) + \sum_{j=K+1}^N \s_{n, j} \mu_j'(t) \Lam Q_{\U{\mu_j(t)}} , 
}
differentiation of the $j$th orthogonality condition for $h(t)$ gives for each $j = K+1, \dots, N$
\EQ{ \label{eq:mu-sys} 
&\s_{n, j} \mu_j'(t) \ang{ \calZ \mid \Lam Q} + \sum_{i \neq j} \s_{n, i} \mu_i'(t) \ang{ \calZ_{\U \mu_j(t)} \mid \Lam Q_{\U {\mu_i(t)}}}  \\
&+  \nu_n\mu_{K+1}'(t)  \ang{ \calZ_{\U{\mu_j(t)}} \mid \Lam \chi_{\U{\nu(t)}} \big( u(t) - u^*(t) - m_n \pi \big)} - \mu_j'(t) \ang{ [r \Lam \calZ]_{\U{\mu_j(t)}} \mid r^{-1} h} \\
&\quad = - \ang{ \calZ_{\U \mu_j(t)} \mid \dot h(t)}, 
}
which, using~\eqref{eq:timu-vanish} and ~\eqref{eq:h-small}, is a diagonally dominant first order differential system for $\vec \mu(t)$. Fix any $t_0 \in \cup_n [a_n, b_n]$ so that~\eqref{eq:h-small} holds at  the initial data $\vec \mu(t_0)$. The existence and uniqueness theorem  gives a unique solution $\vec \mu_{\textrm{ode}} \in C^1(J)$ for $J \ni t_0$ a sufficiently small neighborhood. As the scales were uniquely defined using the implicit function theorem at each fixed $t$ and the solution of the ODE preserves the orthogonality conditions, we must have $\vec \mu(t) = \vec \mu_{\textrm{ode}}(t)$ must agree.  Hence $\vec \mu(t) \in C^1$. 
Finally, inverting~\eqref{eq:mu-sys} we obtain the estimates, 
\EQ{
\abs{ \mu_j'(t)} \lesssim \| \dot h \|_{L^2} \lesssim \te_n
}
Using the above with $j = K+1$ yields~\eqref{eq:nu'}. 
This completes the proof. 
\end{proof}

\subsection{Basic modulation}
\label{ssec:mod}
On some subintervals of the collision interval $[a_n, b_n]$, mutual interactions between the bubbles
dominate the evolution of the solution. We justify the \emph{modulation inequalities}
allowing to obtain explicit information on the solution on such time intervals.
We stress that in our current approach the modulation concerns only the bubbles from $1$ to $K$.

\begin{lem}[Basic modulation, $k \geq 2$] \label{lem:mod-1}
There exist $C_0, \eta_0 > 0$ and a sequence $\zeta_n \to 0$ such that the following is true.

Let $J \subset [a_n, b_n]$ be an open time interval
such that $\bfd(t) \leq \eta_0$ for all $t \in J$.
Then, there exist $\vec\iota \in \{-1, 1\}^K$ (independent of $t \in J$), modulation parameters $\vec\lam \in C^1(J; (0, \infty)^K)$, 
and $\bs g(t) \in \cE$ satisfying, for all $t \in J$,
\begin{align}  \label{eq:u-decomp} 
\chi_{\nu(t)}\bs u(t)  + (1- \chi_{\nu(t)}) m_n \bs  \pi &= \bs \calQ(m_n, \vec\iota, \vec\lambda(t)) + \bs g(t), \\
0 &= \big\la \calZ_{\U{\lam_j(t)}} \mid g(t) \big\ra,  \label{eq:g-ortho} 
\end{align} 
where $\nu(t)$ is as in~\eqref{eq:nu-def} and $m_n$ is as in  Lemma~\ref{lem:ext-sign}. The estimates, 
\begin{align}  
C_0^{-1}\bfd(t) - \zeta_n &\leq \|\bs g(t) \|_{\cE} +  \sum_{j=1}^{K-1} \Big( \frac{ \lam_{j}(t)}{\lam_{j+1}(t)} \Big)^{\frac{k}{2}} 
\leq C_0\bfd(t) + \zeta_n,  \label{eq:d-g-lam} 
\end{align} 
and
\begin{align} 
\|\bs g(t) \|_{\cE} +  \sum_{j \not \in \calA} \Big( \frac{ \lam_{j}(t)}{\lam_{j+1}(t)} \Big)^{\frac{k}{2}} &\leq C_0 \max_{j \in \calA} \left( \frac{ \lam_{j}(t)}{\lam_{j+1}(t)}\right)^{\frac{k}{2}} + \zeta_n, \label{eq:g-upper} 
\end{align}
hold, where 
\EQ{\label{eq:calA-def} 
\calA := \big\{j \in \{1, \ldots, K-1\}: \iota_j \neq \iota_{j+1}\big\}. 
} 
Moreover,  for all $j \in \{1, \ldots, K\}$ and $t \in J$, 
\EQ{ 
\abs{ \lam_j'(t)} &\leq  C_0\|\dot g(t) \|_{L^2} + \zeta_n.   \label{eq:lam'} 
}
If $j \in \{1, \dots,K\}$ we have
\begin{multline}
\Big| \iota_j \lam_j'(t) +  \frac{1}{ \ang{ \calZ \mid \calQ}} \La \calZ_{\U{\lam_j(t)}} \mid \dot g(t)\Ra\Big| \\ \leq C_0\| \bs g(t) \|_{\cE}^2 
+ C_0 \bigg(\Big( \frac{\lam_{j}(t)}{\lam_{j+1}(t)}\Big)^{k-1} + \Big(\frac{ \lam_{j-1}(t)}{\lam_{j}(t)}\Big)^{k-1}\bigg) \| \dot g(t) \|_{L^2} + \zeta_n,  \label{eq:lam'-refined}
\end{multline} 
where, by convention, $\lambda_0(t) = 0, \lam_{K+1}(t) = \infty$ for all $t \in J$. 
\end{lem}

We observe that Lemma~\ref{lem:mod-1} is sufficient to reduce to the case $K \ge 2$. More precisely, under the contradiction assumption that~\eqref{eq:dt-conv} fails, the set $\calA$ as defined in~\eqref{eq:calA-def} is non-empty. 
\begin{lem} 
If~\eqref{eq:dt-conv} is false, then both $N, K \ge 2$ and the set $\calA$ defined in~\eqref{eq:calA-def} is non-empty. 
\end{lem} 

\begin{proof} Assume~\eqref{eq:dt-conv} is false and $\calA$ is empty. For $n$ large, we have $\bfd(a_n) = \eps_n \le \eta_0$ as in Lemma~\ref{lem:mod-1}. 
Define $e_n:= \sup\{t \in [a_n, b_n]  \, : \, \bfd(\tau) \le \eta_0 \, \, \forall \, \tau \in [a_n, t)\}$.  
Since $\calA$ is empty, we see from~\eqref{eq:d-g-lam} and~\eqref{eq:g-upper} that $\bfd(t) \lesssim \zeta_n\ll 1$ for all $t \in [a_n, e_n)$ and thus $e_n  = b_n$, for large $n$. But this is a contradiction, as $[a_n, b_n] \in \calC_K( \eta, \eps_n)$ means there must be a $c_n \in [a_n, b_n]$ with $\bfd(c_n) \ge \eta >0$. Since $\calA$ being empty is impossible, this implies that $N, K \ge 2$ in the event that ~\eqref{eq:dt-conv} is false. 
\end{proof} 

\begin{proof}[Proof of Lemma~\ref{lem:mod-1}]

\textbf{Step 1:}(The decomposition~\eqref{eq:u-decomp} and the estimates~\eqref{eq:d-g-lam} and~\eqref{eq:g-upper}) First, observe that by Lemma~\ref{lem:ext-sign}, 
\EQ{ \label{eq:E>nu} 
 \sup_{t \in [a_n, b_n]} |E( \bs u(t) ; \nu(t),  \infty) - E( \bs u^*)   - (N-K) E( \bs Q)|  = o_n(1) \mas n\to \infty 
}
Since $E( \bs u) = E( \bs u^*) + N E( \bs Q)$ it follows from the above along with~\eqref{eq:nu-prop} that 
\EQ{ \label{eq:E<nu} 
 \sup_{t \in [a_n, b_n]} |E(\bs u(t); 0, 2\nu(t)) - K E( \bs Q)| = o_n(1)  \mas n \to \infty
}

Using continuity of the flow, the fact that $\bfd(t) \le \eta_0$ on $J$,  Lemma~\ref{lem:bub-config}, and by taking $\eta_0>0$ small enough,  we obtain continuous functions $\vec{\ti  \lam}(t) = ( \ti \lam_1(t), \dots, \ti \lam_N(t))$  and signs $\vec {\iota}$ independent of $t \in J$, so that 
\EQ{
\bs u(t) - \bs u^*(t) &=  \bs \calQ( m_{\De}, \vec {\iota},  \vec{\ti \lam}(t)) + \ti{\bs g}(t),
} 
and, 
\EQ{ \label{eq:tilam} 
\bfd(t)^2 \le \| \ti {\bs g}(t)  \|_{\E}^2 + \sum_{j =1}^{N} \Big(\frac{  \ti \lam_j(t)}{\ti \lam_{j+1}(t)}\Big)^{k}  \le 4 \bfd(t)^2.
}
with as usual the convention that $\ti \lam_{N+1}(t) = t$. 
 It follows from~\eqref{eq:E>nu} and~\eqref{eq:E<nu} that, 
 \EQ{ \label{eq:ti-lam-K} 
\sup_{t \in J} \Big[ \Big(\frac{ \ti \lam_K(t)}{ \nu(t)}\Big)^k + \Big( \frac{\nu(t)}{ \ti \lam_{K+1}(t)} \Big)^k  \Big]\lesssim \bfd(t)^2 + o_n(1) \mas n \to \infty, 
 }
 which means, roughly speaking, that there are $K$ bubbles to the left of the curve $\nu(t)$ and $N-K$ bubbles to the right of the curve $\nu(t)$. 
 
 For the purposes of this argument we denote by 
 \EQ{ \label{eq:v_n-w_n-def} 
 \bs v(t) &:= \bs u(t) \chi_{\nu(t)}  + (1- \chi_{\nu(t)}) m_n \bs \pi, \\
  \bs w(t)& := ( \bs u(t) - \bs u^*(t)) ( 1- \chi_{\nu(t)})  +\chi_{\nu(t)} m_n \bs \pi ,
 }
Noting that Lemma~\ref{lem:ext-sign} together with~\eqref{eq:ti-lam-K} imply the identity, 
\EQ{ 
( m_{\De} - \sum_{j=K+1}^N  \iota_j) \bs \pi  = m_n \bs \pi, 
}
we may express $\bs v(t)$  on $J \subset [a_n, b_n]$ as follows, 
 \EQ{
 \bs v (t)  &= m_n \bs \pi + \sum_{j=1}^K \iota_j ( \bs Q_{\ti \lam_j(t)} - \bs\pi)  \\
 & \quad + ( 1- \chi_{\nu(t)}) \sum_{j=1}^K  \iota_j ( \bs Q_{\ti \lam_j(t)} - \bs\pi)  + \chi_{\nu(t)} \sum_{j=K+1}^N  \iota_j  \bs Q_{\ti \lam_j(t)}  + \chi_{\nu(t)} \bs u^*(t) + \chi_{\nu(t)} \ti {\bs g} (t).
 }
 Using~\eqref{eq:rho-def} along with~\eqref{eq:tilam} and ~\eqref{eq:ti-lam-K}  we see that, 
 \EQ{ \label{eq:v_n-upper} 
 \|  \bs v (t)  -  m_n \bs \pi - \sum_{j=1}^K  \iota_j ( \bs Q_{\ti \lam_j(t)} - \bs\pi) \|_{\E}^2 + \sum_{j =1}^K  \Big(\frac{  \ti \lam_j(t)}{\ti \lam_{j+1}(t)}\Big)^{k}  \lesssim \bfd(t)^2 + o_n(1) \mas n \to \infty.
 }
 This means that 
 \EQ{
 \bfd_{ m_n, K}( \bs v(t)) \lesssim \bfd(t) + o_n(1) \mas n \to \infty
 }
 in the notation of Lemma~\ref{lem:mod-static}. By taking $\eta_0>0$ small enough, and $n$ large enough, we may apply Lemma~\ref{lem:mod-static}, (as well as Lemma~\ref{lem:bub-config}, which ensures the signs $\vec \iota$ stays fixed) at each $t \in J$, to obtain unique $\bs g(t) \in \E$, $\vec \lam (t) \in (0, \infty)^K$ so that 
 \EQ{ \label{eq:lam-def} 
 \bs v(t) &=  \bs \calQ( m_n, \vec \iota,  \vec \lam(t)) + \bs g(t),  \\
 0& = \ang{ \calZ_{\U{\lam_j(t)}} \mid g(t)} , \quad \forall j = 1, \dots, K, 
 }
 where in this formula $\vec \iota, \vec\lam$ are $K$-vectors, i.e., $\vec \iota = ( \iota_1, \dots, \iota_K)$, $\vec \lam(t) = ( \lam_1(t), \dots, \lam_K(t))$.  We note the estimate, 
 \EQ{ \label{eq:d_K} 
 \bfd_{m_n, K}( \bs v_n(t))^2 \le \| \bs g(t)\|_{\E}^2 + \sum_{j =1}^{K-1} \Big( \frac{ \lam_{j}(t)}{\lam_{j+1}(t)} \Big)^k  +\Big( \frac{\lam_K(t)}{ \nu(t)} \Big)^k &\le 4 \bfd_{ m_n, K}( \bs v(t))^2 + o_n(1)  \\
&  \lesssim \bfd(t)^2 + o_n(1), 
 }
 as $n \to \infty$. Next, using~\eqref{eq:E<nu} we see that 
 \EQ{
 E( \bs v)  \le K E( \bs Q) + o_n(1).
 }
 Therefore, the estimate~\eqref{eq:g-bound-A} from Lemma~\ref{lem:mod-static} applied here yields, 
 \EQ{
 \| \bs g(t)  \|_{\E}^2 \lesssim  \sup_{j \in \calA} \Big( \frac{ \lam_{j}(t)}{\lam_{j+1}(t)} \Big)^k  + o_n(1) 
 }
 where $\calA = \{ j \in \{1, \dots, K-1\} \, : \, \iota_j \neq \iota_{j+1} \}$,  proving~\eqref{eq:g-upper}. 
 
Next, we prove the lower bound in~\eqref{eq:d-g-lam}. Note the identity, 
 \EQ{ \label{eq:u-u^*-1} 
\bs u(t) - \bs u^*(t) &=  \bs v(t)  - m_n \bs \pi  +  \bs w(t)  - \chi_{\nu(t)} \bs u^*(t) \\
&=   m_{\De} \bs \pi +   \sum_{j=1}^K \iota_j ( \bs Q_{\lam_j(t)} - \bs\pi)  + \sum_{j= K+1}^N \s_{n , j}(\bs Q_{\mu_{ j}(t)} - \bs \pi) \\ 
& \quad + \bs g(t) + \bs h(t) -  \chi_{\nu_n(t)} \bs u^*(t) 
 }
 which follows from~\eqref{eq:v_n-w_n-def},~\eqref{eq:w_n} and~\eqref{eq:m_De-sigma}.
 
 First we prove that $(\iota_{K+1}, \dots, \iota_N) = ( \sigma_{K+1}, \dots, \sigma_N)$. From~\eqref{eq:w_n} and~\eqref{eq:mu_n} we see that 
\EQ{
\| \bs w(t) -  m_{\De} \bs \pi - \sum_{j= K+1}^N \s_{n , j}(\bs Q_{\mu_{j}(t)} - \bs \pi) \|_{\E}^2 + \Big(\frac{ \nu(t)}{ \mu_{ K+1}(t)}\Big)^{k} + \sum_{j = K+1}^N \Big( \frac{ \mu_{ j}(t)}{\mu_{j+1}(t)} \Big)^{k}  \lesssim \eps_n^2 .
}
On the other hand, we see from~\eqref{eq:ti-lam-K} that, 
\EQ{
 \| \bs w(t) -  m_{\De} \bs \pi - \sum_{j= K+1}^N \iota_j (\bs Q_{\ti \lam_{j}(t)} - \bs \pi) \|_{\E}^2 + \Big(\frac{ \nu(t)}{ \ti \lam_{ K+1}(t)}\Big)^{k} + \sum_{j = K+1}^N \Big( \frac{ \ti \lam_{j}(t)}{\ti \lam_{ j+1}(t)} \Big)^{k}  \lesssim \bfd(t)^2 + o_n(1) .
}
Hence, using Lemma~\ref{lem:bub-config} we see that for any $\te_0>0$ we may take $\eta_0>0$ small enough so that $(\iota_{K+1}, \dots, \iota_N) = ( \sigma_{K+1}, \dots, \sigma_N)$, and in addition we have 
\EQ{ \label{eq:tilam-mu} 
\Big|\frac{ \ti \lam_{j}(t)}{\mu_{n, j}(t)} - 1\Big| \le \te_0 \quad \forall j = K+1, \dots, N.
}
The above, together with~\eqref{eq:mu_n}  implies that 
\EQ{
\sum_{j=K+1}^{N} \Big(\frac{  \ti \lam_j(t)}{\ti \lam_{j+1}(t)}\Big)^{k}  \lesssim \eps_n^2 .
}
We may thus rewrite~\eqref{eq:u-u^*-1} as 
 \EQ{ \label{eq:u-u*-1} 
\bs u(t) - \bs u^*(t) 
&=   m_{\De} \bs \pi +   \sum_{j=1}^K \iota_j ( \bs Q_{\lam_j(t)} - \bs\pi)  + \sum_{j= K+1}^N \iota_ j(\bs Q_{\mu_{ j}(t)} - \bs \pi) \\ 
& \quad + \bs g(t) + \bs h(t) -  \chi_{\nu(t)} \bs u^*(t) 
 }
Noting that    
\EQ{
\sup_{t \in[a_n, b_n]}\|  \bs u^*(t) \chi_{\nu(t)} \|_{\E} = o_n(1) \mas n \to \infty, 
}
the previous line together with~\eqref{eq:d_K} and~\eqref{eq:mu_n} imply 
that, 
\EQ{
\bfd(t)^2  \lesssim \bfd_{ m_n, K}( \bs v(t))^2 + o_n(1) \lesssim \| \bs g(t)\|_{\E}^2 + \sum_{j =1}^{K-1} \Big( \frac{ \lam_{j}(t)}{\lam_{j+1}(t)} \Big)^k  \mas n \to \infty, 
}
which proves the lower bound in~\eqref{eq:d-g-lam}. 

\textbf{Step 2:}(The dynamical estimates~\eqref{eq:lam'} and~\eqref{eq:lam'-refined}) Momentarily assuming that $\vec \lam \in C^1(J)$ (we will justify this assumption below)  we record the computations, 
\EQ{
\p_t v(t) &=  \dot g(t) + ( m_n \pi - u(t))  \frac{\nu'(t)}{\nu(t)}  \Lam \chi_{\nu(t)},   \quad 
\p_t \calQ( m_n, \vec \iota, \vec \lam(t)) =   - \sum_{j=1}^K \iota_j \lam_{j}'(t) \Lam Q_{\U{\lam_j(t)}} , 
}
which lead to the expression, 
\EQ{ \label{eq:dt-g-1} 
\p_t g(t) = \dot g(t)  + \sum_{j=1}^K \iota_j \lam_{j}'(t) \Lam Q_{\U{\lam_j(t)}}  + ( m_n \pi - u(t))  \frac{\nu'(t)}{\nu(t)}  \Lam \chi_{\nu(t)} . 
}
We differentiate the orthogonality conditions~\eqref{eq:g-ortho} for each $j = 1, \dots, K$,  
\EQ{
0 &= -\frac{\lam_j'}{\lam_j} \ang{ \ULam \calZ_{\U{\lam_j}} \mid g} + \ang{ \calZ_{\U{\lam_j}} \mid \p_t g}  \\
& = -\frac{\lam_j'}{\lam_j} \ang{ \ULam \calZ_{\U{\lam_j}} \mid g} + \ang{ \calZ_{\U{\lam_j}} \mid \dot g} + \sum_{\ell=1}^K  \iota_{\ell} \lam_{\ell}' \ang{ \calZ_{\U{\lam_{j}}} \mid \Lam Q_{\U{\lam_{\ell}} }}  + \frac{\nu'}{\nu} \ang{ \calZ_{\U{\lam_j}} \mid ( m_n \pi - u)  \Lam \chi_{\nu} }, 
}
which we rearrange into the system, 
\begin{multline}  \label{eq:lam'-system} 
\iota_j \lam_j' \Big( \ang{\calZ \mid \Lam Q }  -  \lam_j^{-1}\La  \ULam \calZ_{\U{\lam_j}} \mid   g\Ra \Big) + \sum_{i \neq j}  \iota_{i} \lam_{i}' \ang{ \calZ_{\U{\lam_{j}}} \mid \Lam Q_{\U{\lam_{i}} }}  \\
=  - \ang{ \calZ_{\U{\lam_j}} \mid \dot  g} -  \frac{\nu'}{\nu} \ang{ \calZ_{\U{\lam_j}} \mid ( m_n \pi - u)  \Lam \chi_{\nu} }. 
\end{multline} 
This is a diagonally dominant system, hence invertible, and we arrive at the estimate, 
\EQ{ \label{eq:lam_j'} 
\abs{ \lam_j'} & \lesssim \| \dot g \|_{L^2} + o_n(1)  \quad j = 1, \dots, K, 
}
after noting the estimates, 
\EQ{
\abs{ \ang{ \calZ_{\U{\lam_j}} \mid \dot g}} & \lesssim \|\dot g \|_{L^2}  \\
\abs{ \frac{\nu'(t)}{\nu(t)} \ang{ \calZ_{\U{\lam_j}} \mid ( m_n \pi - u(t))  \Lam \chi_{\nu(t)} }} &\lesssim \abs{ \nu'}\frac{\lam_j}{\nu} \| r^{-1} ( m_n \pi - u(t)) \Lam \chi_{\nu(t)} \|_{L^2}  = o_n(1) , 
}
where the last line follows from ~\eqref{eq:nu'}. Lastly, we note that the system~\eqref{eq:lam'-system} implies that $\vec \lam(t)$ is a $C^1$ function on $J$. Indeed, arguing as in the end of the proof of Lemma~\ref{lem:ext-sign}, let $t_0 \in J$ be any time and let $\vec \lam(t_0)$ be defined as in~\eqref{eq:lam-def}. Using the smallness~\eqref{eq:d_K} at time $t_0$, the system~\eqref{eq:lam'-system} admits a unique $C^1$ solution $\vec\lam_{\textrm{ode}}(t)$ in a neighborhood of $t_0$. Due to the way the system~\eqref{eq:lam'-system} was derived, the orthogonality conditions in~\eqref{eq:lam-def} hold with $\vec \lam_{\textrm{ode}}(t)$. Since $\vec\lam(t)$ was obtained uniquely via the implicit function theorem, we must have $\vec \lam(t) = \vec \lam_{\textrm{ode}}(t)$, which means that $\vec \lam(t)$ is $C^1$. 

Lastly, the estimates~\eqref{eq:lam'-refined} are immediate from~\eqref{eq:lam'-system} using~\eqref{eq:lam_j'} along with the estimates, 
\EQ{
 \ang{ \calZ_{\U{\lam_{j}}} \mid \Lam Q_{\U{\lam_{i}} }}  & \lesssim \begin{cases} \Big( \frac{\lam_{j}}{\lam_i} \Big)^{k+1} \mif  j < i \\ \Big( \frac{\lam_{i}}{\lam_j} \Big)^{k-1} \mif j>i \end{cases}  \\
  \Big| \lam_j^{-1}\ang{ \U\Lam \calZ_{\U{\lam_j}} \mid  g} \Big|  &\lesssim  \| g \|_H, 
}
This completes the proof. 
%
%
%
\end{proof}

\subsection{Refined modulation}  \label{ssec:ref-mod} 

Next, our goal is to gain precise dynamical control of the modulation parameters in the spirit of \cite{JJ-APDE, JL1}.
The idea is to construct a virial correction to the modulation parameters; see~\eqref{eq:beta-def}. 
We start by  finding suitable truncation of the function $\frac{1}{2} r^2$, 
similar to~\cite[Lemma 4.6]{JJ-AJM}. Since here we may have arbitrary number of bubbles, we need to localize this function both away from $r =0$ and away from $r =\infty$. To make the exposition as uniform as possible, we restrict to equivariance classes $k \ge 2$ in this section, saving case $k =1$, which introduces additional technical complications, for the appendix.

\begin{lem} \label{lem:q} 
For any $c > 0$ and $R > 1$ there exists a function $q = q_{c, R} \in C^4((0, \infty))$
having the following properties:
\begin{enumerate}[(P1)]
\item $q(r) = \frac 12 r^2$ for all $r \in [R^{-1}, R]$,
\item there exists $\wt R > 0$ (depending on $c$ and $R$)
such that $q(r) = \tx{const}$ for $r \geq \wt R$ and $q(r) = \tx{const}$ for $r \leq \wt R^{-1}$,
\item $|q'(r)| \lesssim r$ and $|q''(r)| \lesssim 1$ for all $r > 0$,
with constants independent of $c$ and $R$,
\item $q''(r) \geq -c$ and $\frac 1r q'(r) \geq -c$ for all $r > 0$,
\item $\big|\big(\frac{\vd^2}{\vd r^2} + \frac 1r\frac{\vd}{\vd r}\big)^2q(r)\big| \leq cr^{-2}$ for all $r > 0$,
\item $\big|\big(\frac{q'(r)}{r}\big)'\big| \leq cr^{-1}$ for all $r > 0$.
\end{enumerate}
\end{lem}
\begin{proof}
\textbf{Step 1:} We construct a function $q(r)$ satisfying the desired properties for all $r \ge 1$. In this step, without loss of generality we can assume $R=1$. Let $c_1>0$ be small, to be chosen later and set 
\EQ{
q_{\textrm{o}}(r) := \begin{cases} \frac{1}{2} r^2  \mif r \le 1 \\ \frac{1}{2} r^2  - c_1 \psi_{\textrm{o}}(r)  \mif r \ge 1 \end{cases} 
}
where $\psi_{\oo}(r) = \frac{1}{2} r^2 \log r + \ti \psi_{\oo}(r)$ and $\ti \psi_{\oo}(r)$ is any smooth function satisfying, 
\EQ{
\ti \psi_{\oo} (1) = 0, \quad \ti \psi_{\oo}' (1) = -\frac{1}{2}, \quad \ti \psi_{\oo}'' (1) = -\frac{3}{2}, \quad \ti \psi_{\oo}''' (1) = -1, \quad \ti \psi_{\oo}^{(4)}(1) = 1,  \\
\abs{ \ti \psi_{\oo}^{(j)}(r) } \lesssim r^{2 - j} \quad \forall r \ge 1, \, \, j = 0, 1, \dots 4
}
which ensures $ 0 = \psi_{\oo} (1) = \psi_{\oo}' (1)  = \psi_{\oo}'' (1)  =\psi_{\oo}''' (1)  =\psi_{\oo}^{(4)} (1) $. To construct such a  function it suffices to take a suitable linear combination of negative powers of $r$, for example. Set $R_0 := \exp( 1/ c_1)$. We check all the properties for $1 \le r \le R_0$. We have, $q_{\oo}'(r) = r( 1- c_1 \log r) + O( c_1 r)$ and $q_{\oo}''(r) = (1- c_1 \log r) + O(c_1)$, so \textit{(P4)} holds. Also, 
\EQ{
 \Big| \Big(\frac{ \psi_{\oo}'(r)}{r}\Big)'\Big|  = \abs{ ( \log r + \frac{1}{2})'} + O(r^{-1}) = O( r^{-1})
}
for an absolute constant, implying \textit{(P6)}. \textit{(P5)} follows from $\De^2( r^2 \log r) = 0$, with all the remaining terms estimated brutally. We now truncate at $R_0$, setting $e_j(r):= \frac{1}{j !}r^j \chi(r)$ for $j  = 1, \dots, 4$ and defining, 
\EQ{
q(r) := \begin{cases} q_{\oo}(r) \mif r \le R_0 \\ q_{\oo}(R_0) + \sum_{j =1}^4 q_{\oo}^{(j)}(R_0) R_0^j e_j(-1 + r/ R_0), \mif r \ge R_0\end{cases} 
}
Noting that $\abs{ q_{\oo}^{(j)}(R_0)} \lesssim c_1 R_0^{2 - j}$ for $j = 1, \dots, 4$, we see that $q(r)$ inherits all the desired properties from $q_{\oo}(r)$ and is constant after $3 R_0$; see~\cite[Lemma 4.6]{JJ-AJM} for additional details. 

\textbf{Step 2:} We next find a function $q(r)$ with all the desired properties for $r \le 1$. As above, we may assume here that $R =1$. Let $c_1>0$ be small, to be chosen later, and set 
\EQ{
q_{\ii}(r):= \begin{cases} \frac{1}{2} r^2 \mif r \ge 1 \\ \frac{1}{2} r^2 + c_1 \psi_{\ii}(r)  \mif r \le 1 \end{cases}  
}
where  $\psi_{\ii}(r) = \frac{1}{2} r^2  \log r + \ti \psi_{\ii}(r)$ and $\ti \psi_{\ii}(r)$ is any smooth function satisfying, 
\EQ{
\ti \psi_{\ii} (1) = 0, \quad \ti \psi_{\ii}' (1) = -\frac{1}{2}, \quad \ti \psi_{\ii}'' (1) = -\frac{3}{2}, \quad \ti \psi_{\ii}''' (1) = -1, \quad \ti \psi_{\ii}^{(4)}(1) = 1,  \\
\abs{ \ti \psi_{\ii}^{(j)}(r) } \lesssim r^{2 - j} \quad \forall r \le  1, \, \, j = 0, 1, \dots 4
}
which ensures $ 0 = \psi_{\ii} (1) = \psi_{\ii}' (1)  = \psi_{\ii}'' (1)  =\psi_{\ii}''' (1)  =\psi_{\ii}^{(4)} (1) $. To obtain such a function it suffices to take a suitable linear combination of positive powers of $r$, for example. Set $R_0^{-1}:= \exp( - \frac{1}{c_1})$. One can check, as in Step 1, that all the properties hold for $R_0^{-1} \le r \le 1$, using that $1 + c_1 \log r \ge 0$ in this regime. Then truncate as in Step 1 to obtain the truncated function $q(r)$. 

\textbf{Step 3:} The final function $q(r)$ is obtained by gluing together the two functions called $q$ obtained in Steps 1, 2. 
\end{proof}
\begin{defn}[Localized virial operator]
For each $\lam>0$ we set
\begin{align}
A(\lambda)g(r) &:= q'\big(\frac{r}{\lambda}\big)\cdot \p_r g(r), \label{eq:opA}\\
\uln A(\lambda)g(r) &:=\big(\frac{1}{2\lambda}q''\big(\frac{r}{\lambda}\big) + \frac{1}{2r}q'\big(\frac{r}{\lambda}\big)\big)g(r) + q'\big(\frac{r}{\lambda}\big)\cdot\p_r g(r). \label{eq:opA0}
\end{align}
These operators depend on $c$ and $R$ as in Lemma~\ref{lem:q}. 
\end{defn}
Note the similarity between $A$ and $\frac{1}{\lambda} \Lam$ and between $\U A$ and $\frac{1}{\lam} \ULam$. For technical reasons we introduce the space 
\EQ{
X:= \{ g \in H \mid \frac{g}{r}, \p_r g \in H\}.
}

\begin{lem}[Localized virial estimates]  \emph{\cite[Lemma 5.5]{JJ-AJM}} \label{lem:A} 
For any $c_0>0$ there exist $c_1, R_1>0$, so that for all $c, R$ as Lemma~\ref{lem:q} with $c< c_1$, $R> R_1$ the operators $A(\lambda)$ and $\U A(\lambda)$ defined in~\eqref{eq:opA} and~\eqref{eq:opA0} have the following properties:

  \begin{itemize}[leftmargin=0.5cm]
    \item the families $\{A(\lambda): \lam > 0\}$, $\{\U A(\lambda): \lam> 0\}$, $\{\lambda\partial_\lambda A(\lambda): \lam > 0\}$
      and $\{\lambda\partial_\lambda \U A(\lambda): \lambda > 0\}$ are bounded in $\mathscr{L}(H; L^2)$, with the bound depending only on the choice of the function $q(r)$,
    \item 
   Let $g_1 = \calQ( m, \vec \iota, \vec \lam)$ be an $M$-bubble configuration and let $g \in X$. Then, for all $\lambda > 0$, 
      \begin{multline}  \label{eq:A-by-parts}
      \Big| \ang{ A(\lambda)g_1\mid  \frac{1}{r^2}\big(f(g_1 + g_2) - f(g_1) - f'(g_1)g_2\big)}  \\ +\ang{ A(\lambda)g_2\mid \frac{1}{r^2}\big(f(g_1+g_2) - f(g_1) -k^2 g_2\big)}\Big| 
        \leq \frac{c_0}{\lambda} \|g_2\|_H^2, 
      \end{multline}
    \item For all $g \in X$ we have  
\EQ{
        \label{eq:pohozaev}
        \ang{\U A(\lambda)g \mid \LL_0 g} \ge  - \frac{c_0}{\lambda}\|g\|_{H}^2 + \frac{1}{\lambda}\int_{R^{-1} \lam}^{R\lambda}  \Big((\partial_r g)^2 + \frac{k^2}{r^2}g^2\Big) \udr, 
        }
        \item For $\lam, \mu >0$ with either $\lam/\mu \ll 1$ or $\mu/\lam \ll 1$, 
\begin{align} 
      \label{eq:L0-A0}
      \|\ULam \Lambda Q_\uln{\lambda} - \U A(\lambda)\Lambda Q_{\lambda}\|_{L^2} &\leq c_0, \\
      \label{eq:L-A}
      \|\big(\frac{1}{\lam} \Lambda  - A(\lambda)\big) Q_\lambda\|_{L^\infty} &\leq \frac{c_0}{\lambda},  \\
    \| A(\lam) Q_\mu \|_{L^\infty} + \| \U A(\lam) Q_\mu \|_{L^\infty}  &\lesssim  \frac{1}{\lam}  \min \{ (\lam/ \mu)^k,  (\mu/ \lam)^k \}  \label{eq:A-mismatch} \\
 \| A(\lam) Q_\mu \|_{L^2} + \| \U A(\lam) Q_\mu \|_{L^2} &\lesssim   \min \{ (\lam/ \mu)^k,  (\mu/ \lam)^k \}  \label{eq:A-mismatch-2}
     \end{align} 
     
     \item Lastly, the following localized coercivity estimate holds. Fix any smooth function $\calZ \in L^2 \cap X$ such that $\ang{\calZ \mid \Lam Q} >0$. For any $g \in H, \lam>0$ with $\ang{g \mid \calZ_{\U \lam}} = 0$, 
     \EQ{ \label{eq:coercive}
    \frac{1}{\lambda}\int_{R^{-1} \lam}^{R\lambda}  (\partial_r g)^2 + \frac{k^2}{r^2}g^2 \udr &+   \frac{1}{\lam} \int_{0}^\infty  \big(\frac{1}{2}q''\big(\frac{r}{\lambda}\big) + \frac{\lam}{2r}q'\big(\frac{r}{\lambda}\big)\big) \frac{k^2}{r^2} ( f'(Q_{\lam}) - 1)g^2  \, r \, \ud r  \\
    \ge - \frac{c_0}{\lam} \| g \|_{H}^2 .
     }
  \end{itemize}
\end{lem}

\begin{proof}
See~\cite[Lemmas 4.7 and 5.5]{JJ-AJM} the proof in the cases $k \ge 2$ and~\cite[Lemma 3.7 and Remark 3.8]{R19} for modifications to handle the case $k=1$. 
\end{proof}

The modulation parameters $\vec \lam(t)$ defined in Lemma~\ref{lem:mod-1} are imprecise proxies for the dynamics in the case $k =2$ (and also $k=1$; see the appendix) due to the fact that the orthogonality conditions were imposed relative to $\calZ \neq \Lam Q$. Indeed, we use~\ref{eq:g-ortho} primarily to ensure coercivity, and thus the estimate~\eqref{eq:g-upper}, as well as the differentiability of $\vec\lam(t)$. To access the dynamics of~\eqref{eq:wmk} we  introduce a correction $\vec \xi(t)$ defined as follows. For each $t \in J \subset [a_n, b_n]$ as in Lemma~\ref{lem:mod-1} set, 
\EQ{\label{eq:xi-def} 
\xi_j(t) := \begin{cases} \lam_j(t) \mif k \ge 3 \\ \lam_j(t) - \frac{\iota_j}{\| \Lam Q \|_{L^2}^2} \La \chi_{L\lam_j(t)} \Lam Q_{\U{\lam_j(t)}} \mid g(t) \Ra  \mif  k=2\end{cases} 
}
for each $j = 1, \dots, K-1$, and where $L>0$ is a large constant to be determined below.  (Note that for $j =K$ we only require the brutal estimate~\eqref{eq:lam'}). We require yet another modification, since the dynamics  of~\eqref{eq:wmk} truly enter after taking two derivatives of the modulation parameters and it is not clear how to derive useful estimates from the expression for $\xi_j''(t)$. So we introduce a refined modulation parameter, which we view as a subtle correction to $\xi_j'(t)$.  For each $t \in J \subset [a_n, b_n]$ as in Lemma~\ref{lem:mod-1} and for each $j \in \{1, \dots, K\}$ define, 
\begin{equation} \label{eq:beta-def} 
\beta_j(t) :=- \frac{\iota_j}{ \| \Lam Q \|_{L^2}^2} \La \Lam Q_{\U{\lam_j(t)}} \mid \dot g(t)\Ra  -   \frac{1}{ \| \Lam Q \|_{L^2}^2} \ang{ \uln A( \lam_j(t)) g(t) \mid \dot g(t)}.
\end{equation}
Note that $\be_j(t)$ is similar to the function called $b(t)$ in~\cite{JL1}. 

%
%

\begin{lem}[Refined modulation] 
\label{cor:modul}
Let $k \ge 2$ and $c_0\in (0 ,1)$.  There exist $\eta_0>0, L>0, c>0, R>1,  C_0>0$ and a decreasing sequence $ \delta_n \to 0$ so that the following is true. 
Let  $J \subset [a_n, b_n]$ be an open time interval with 
\EQ{  \label{eq:de_n-lower} 
 \bfd(t) \leq \eta_0 \mand \max_{i \in \calA}\big(\lambda_i(t) / \lambda_{i+1}(t)\big)^{k/2} \geq \delta_n, 
}
for all $t \in J$, where $\calA :=  \{ j \in \{1, \dots, K-1\} \mid \iota_{j}  \neq \iota_{j+1} \}$.  Then, for all $t \in J$, 
\EQ{\label{eq:g-bound} 
\| \bs g(t) \|_{\E} + \sum_{i \not \in \calA}  \big(\lambda_i(t) / \lambda_{i+1}(t)\big)^{k/2} \le C_0 \max_{ i  \in \calA}  \big(\lambda_i(t) / \lambda_{i+1}(t)\big)^{k/2}, \\
}
and, 
\EQ{ \label{eq:d-bound} 
\frac{1}{C_0} \bfd(t) \le \max_{i \in \calA} \big(\lambda_i(t) / \lambda_{i+1}(t)\big)^{k/2}  \le C_0 \bfd(t) .
}
Moreover, for all $j \in \{1, \dots, K-1 \}$ and $t \in J$, 
\begin{equation}\label{eq:xi_j-lambda_j} 
|\xi_j(t) / \lambda_j(t) - 1| \leq c_0, 
\end{equation}
\begin{equation}\label{eq:lam_j'-beta_j} 
|\xi_j'(t) - \beta_j(t)| \leq c_0\max_{i  \in \calA}\bigg( \frac{\lam_{i}(t)}{\lam_{i+1}(t)} \bigg)^{k/2} , 
\end{equation}
and,  
\EQ{ \label{eq:beta_j'} 
 \beta_{j}'(t) &\ge  \Big({-} \iota_j \iota_{j+1}\om^2 -   c_0\Big) \frac{1}{\lam_{j}(t)} \left( \frac{\lam_{j}(t)}{\lam_{j+1}(t)} \right)^{k} +    \Big(\iota_j \iota_{j-1}\om^2 -  c_0\Big) \frac{1}{\lam_{j}(t)} \left( \frac{\lam_{j-1}(t)}{\lam_{j}(t)} \right)^{k}   \\
&\quad  -  \frac{c_0}{\lam_j(t)} \max_{i \in \calA}\bigg( \frac{\lam_{i}(t)}{\lam_{i+1}(t)} \bigg)^{k}.
}
where, by convention, $\lambda_0(t) = 0, \lam_{K+1}(t) = \infty$ for all $t \in J$, and $\om^2>0$ is defined by 
\EQ{ \label{eq:omega-def} 
\om^2 = \om^2(k) := 8k^2 \| \Lam Q \|_{L^2}^{-2} = 4k^2 \pi^{-1}  \sin(\pi/k)>0. 
} 
\end{lem}

\begin{rem}
\label{rem:deltan}
By \eqref{eq:g-upper}, without loss of generality (upon enlarging $\epsilon_n$), we can assume that
\begin{equation}
\eta_0 \geq \bfd(t) \geq \epsilon_n\quad\text{implies} \quad \max_{i \in \calA}\big(\lambda_i(t) / \lambda_{i+1}(t)\big)^{k/2} \geq \delta_n,
\end{equation}
so that Lemma~\ref{cor:modul} can always be applied on the time intervals
$J \subset [a_n, b_n]$ as long as $\bfd(t) \leq \eta_0$ on $J$. 
\end{rem}

Before beginning the proof of Lemma~\ref{cor:modul} we record the equation satisfied by $\bs g(t)$. Observe the identity, 
\EQ{
\De \calQ(m_n, \vec \iota, \vec \lam) = \frac{k^2}{r^2} \sum_{ j =1}^K \iota_j f(Q_{\lam_j}), 
}
and hence, 
\EQ{
(\p_t^2 u) \chi_{\nu} &= \chi_{\nu} \De u - \frac{k^2}{r^2} f( u) \chi_{\nu}  \\
& = \De \big( u \chi_{\nu} + (1- \chi_{\nu}) m_n\pi \big) - \frac{k^2}{r^2} f\big( u \chi_{\nu} + (1- \chi_{\nu}) m_n \pi)\big)  \\
&\quad - \frac{2}{r} \p_r u \Lam \chi_{\nu} +  \frac{1}{r^2}( m_n\pi - u)( r^2 \De \chi)_{\nu}  + \frac{1}{r^2} \Big(  f\big( u \chi_{\nu} + (1- \chi_{\nu}) m_n \pi)\big)  - f( u) \chi_{\nu} \Big) \\
& = \De g -\frac{1}{r^2}\Big( f( \calQ(m_n, \vec \iota, \vec \lam) + g) - \sum_{ j =1}^K \iota_j f(Q_{\lam_j}) \Big) \\
&\quad - \frac{2}{r} \p_r u \Lam \chi_{\nu} +  \frac{1}{r^2}( m_n\pi - u)( r^2 \De \chi)_{\nu}  + \frac{1}{r^2} \Big(  f\big( u \chi_{\nu} + (1- \chi_{\nu})m_n \pi)\big)  - f( u) \chi_{\nu} \Big) . 
}
Recalling~\eqref{eq:dt-g-1},  we are led to the system of equations, 
\EQ{ \label{eq:g-eq} 
\p_t g(t) &= \dot g(t)  + \sum_{j=1}^K \iota_j \lam_{j}'(t) \Lam Q_{\U{\lam_j(t)}}  + \phi( u(t), \nu(t))   \\
\p_t \dot g(t) &=  - \LL_{\calQ} g +f_{\bfi}( m_n, \iota, \vec \lam) + f_{\bfq}(m_n, \vec \iota,  \vec \lam, g) + \dot \phi( u(t), \nu(t)), 
}
where,  
\EQ{ \label{eq:h-def} 
\phi( u, \nu) &:= ( m_n \pi - u)  \frac{\nu'}{\nu}  \Lam \chi_{\nu} \\
\dot \phi( u, \nu) &:= - \frac{2}{r} \p_r u \Lam \chi_{\nu} +  \frac{1}{r^2}(m_n \pi - u)( r^2 \De \chi)_{\nu}  \\
&\quad + \frac{1}{r^2} \Big(  f\big( u \chi_{\nu} + (1- \chi_{\nu})m_n \pi)\big)  - f( u) \chi_{\nu} \Big)  - \frac{\nu'}{\nu} \Lam \chi_{\nu} \p_t u , 
}
which we note are supported in $r  \in ( \nu, \infty)$, 
and 
\EQ{
f_{\bfi}( m_n, \vec \iota, \vec \lam) &:= - \frac{k^2}{r^2} \Big( f\big( \calQ(m_n, \vec \iota, \vec \lam)\big) - \sum_{j =1}^{K} \iota_j  f(Q_{\lam_{j}}) \Big) \\
f_{\bfq}(m_n, \vec \iota, \vec \lam, g) &:= - \frac{k^2}{r^2} \Big( f\big( \calQ(m_n, \vec \iota, \vec \lam) + g \big)  - f\big(\calQ(m_n, \vec \iota, \vec \lam) \big) -  f'\big( \calQ(m_n, \vec \iota, \vec \lam)\big) g \Big). 
}
The subscript  $\bfi$ above stands for ``interaction'' and $\bfq$ stands for ``quadratic.'' 
In particular, $f_{\bfq}(m_n, \vec \iota, \vec \lam, g) $ satisfies, 
\EQ{ \label{eq:fq-est} 
\| f_{\bfq}(m_n, \vec \iota, \vec \lam, g)  \|_{L^1} \lesssim \| g \|_{H}^2. 
}
In one instance it will be convenient to rewrite the right-hand side of the equation for $\dot g$ as follows, 
\EQ{ \label{eq:g-eq-alt} 
\p_t \dot g = -\LL_0 g + f_{\bfi}( m_n, \iota, \vec \lam)  + \ti  f_{\bfq}(m_n, \vec \iota,  \vec \lam, g)  + \dot  \phi( u, \chi_{\nu}),
}
where $\ti  f_{\bfq}(m_n, \vec \iota, \vec \lam, g)$ is defined by the formula, 
 \EQ{\label{eq:ti-fq-def} 
\ti f_{\bfq}(m_n, \vec \iota, \vec \lam, g) &:= - \frac{k^2}{r^2} \Big( f\big( \calQ(m_n, \vec \iota, \vec \lam) + g \big)  - f\big(\calQ(m_n, \vec \iota, \vec \lam) \big)-  g \Big). 
}

\begin{proof}[Proof of Lemma~\ref{cor:modul} ]
First, we prove the estimates~\eqref{eq:g-bound} and~\eqref{eq:d-bound}.  Let $\zeta_n$ be the sequence given by Lemma~\ref{lem:mod-1} and let $\de_n$ be any sequence such that $\zeta_n/ \de_n \to 0$ as $n \to \infty$. Using Lemma~\ref{lem:mod-1}, estimate~\eqref{eq:g-bound} follows from~\eqref{eq:g-upper}  and the estimate~\eqref{eq:d-bound} follows from~\eqref{eq:d-g-lam}. 

Note also that with this choice of $\de_n$ and ~\eqref{eq:g-bound}, the estimate~\eqref{eq:lam'} leads to, 
\EQ{ \label{eq:lam'-new} 
\abs{\lam_j'(t)} \lesssim \max_{ i  \in \calA}  \big(\lambda_i(t) / \lambda_{i+1}(t)\big)^{k/2}. 
}

Next, we treat~\eqref{eq:xi_j-lambda_j}, which is only relevant in the case $k =2$. From~\eqref{eq:xi-def} we see that, 
\EQ{
|\xi_j / \lam_j -1| &= | \| \Lam Q \|_{L^2}^{-2} \lam_j^{-1}\La \chi_{L\lam_j} \Lam Q_{\U{\lam_j}} \mid g \Ra | \\
& \lesssim \| g \|_{L^\infty} ( 1 + \int_1^{L} \Lam Q(r) \, r \, \ud r) \lesssim (1+ \log(L)) \| g \|_H  \le c_0 
}
using~\eqref{eq:g-bound} and~\eqref{eq:d-bound} in the last line. Next we compute $\xi_j'(t)$. 

For $k=2$, from~\eqref{eq:xi-def} we have 
\EQ{ \label{eq:xi'-exp} 
\xi_j' &=  \lam_j'  - \frac{\iota_j}{\| \Lam Q \|_{L^2}^2}  \La \chi_{L\lam_j } \Lam Q_{\U{\lam_j}} \mid \p_t g  \Ra
  \\
  &\quad + \frac{\iota_j}{\| \Lam Q \|_{L^2}^2}\frac{ \lam_j' }{\lam_j} \La  \Lam \chi_{L\lam_j}  \Lam Q_{\U{\lam_j}} \mid g\Ra + \frac{\iota_j}{\| \Lam Q \|_{L^2}^2} \frac{\lam_{j}'}{\lam_j} \La \chi_{L\lam_j} \ULam \Lam Q_{\U{\lam_j}} \mid g\Ra . 
}
We examine each of the terms on the right above. The last two terms are negligible. Indeed, using $\| g \|_{L^\infty} \lesssim \| g \|_H$, 
\EQ{
\Big| \frac{ \lam_j'}{\lam_j} \La  \Lam \chi_{L\lam_j}  \Lam Q_{\U{\lam_j}} \mid g\Ra \Big| &\lesssim \abs{ \lam_j'}   \| g \|_{L^{\infty}}  \int_{2^{-1}L}^{2L} \Lam Q(r) \,r \, \ud r \\
& \lesssim  \| \bs g \|_{\E}^2 , 
}
and, 
\EQ{
\Big|\frac{\lam_{j}'}{\lam_j} \La \chi_{L\lam_j} \ULam \Lam Q_{\U{\lam_j}} \mid g\Ra \Big| & \lesssim \abs{ \lam_j'} \|g \|_{L^\infty} \int_0^{2L} \Lam Q(r)\, r \, \ud r  \lesssim (1+ \log( L)) \| \bs g \|_{\E}^2  , 
}
which is small relative to $\| \bs g \|_{\E}$ because of~\eqref{eq:g-bound}. Using~\eqref{eq:g-eq} in the second term in~\eqref{eq:xi'-exp} gives 
\EQ{
- \frac{\iota_j}{\| \Lam Q \|_{L^2}^2}  \La \chi_{L\lam_j } \Lam Q_{\U{\lam_j}} \mid \p_t g  \Ra &= - \frac{\iota_j}{\| \Lam Q \|_{L^2}^2}  \La \chi_{L\lam_j } \Lam Q_{\U{\lam_j}} \mid \dot g  \Ra  -  \frac{\iota_j}{\| \Lam Q \|_{L^2}^2}  \La \chi_{L\lam_j } \Lam Q_{\U{\lam_j}} \mid \sum_{i=1}^K \iota_i \lam_{i}' \Lam Q_{\U{\lam_i}} \Ra\\
&\quad - \frac{\iota_j}{\| \Lam Q \|_{L^2}^2}  \La \chi_{L\lam_j } \Lam Q_{\U{\lam_j}} \mid \phi( u, \nu)  \Ra. 
}
The first term on the right satisfies, 
\EQ{
-  \frac{\iota_j}{\| \Lam Q \|_{L^2}^2}  \La \chi_{L\lam_j } \Lam Q_{\U{\lam_j}} \mid \dot g  \Ra  &=  -\frac{\iota_j}{\| \Lam Q \|_{L^2}^2}  \La  \Lam Q_{\U{\lam_j}} \mid \dot g  \Ra  +  \frac{\iota_j}{\| \Lam Q \|_{L^2}^2}  \La (1-\chi_{L\lam_j }) \Lam Q_{\U{\lam_j}} \mid \dot g  \Ra  \\
& =  -\frac{\iota_j}{\| \Lam Q \|_{L^2}^2}  \La  \Lam Q_{\U{\lam_j}} \mid \dot g  \Ra  + o_L(1)\|\bs g \|_{\E}. 
}
where the $o_L(1)$ term can be made as small as we like by taking $L>0$ large. 
Using~\eqref{eq:lam'-new}, the second term yields, 
\EQ{
& -  \frac{\iota_j}{\| \Lam Q \|_{L^2}^2}  \La \chi_{L \lam_j } \Lam Q_{\U{\lam_j}} \mid \sum_{i=1}^K \iota_i \lam_{i}' \Lam Q_{\U{\lam_i}} \Ra 
 = - \lam_j'  \\  
 - \sum_{i \neq j} \frac{\iota_j\iota_i \lam_{i}'}{\| \Lam Q \|_{L^2}^2} &  \La \chi_{L\lam_j} \Lam Q_{\U{\lam_j}} \mid   \Lam Q_{\U{\lam_i}} \Ra
  +   \frac{  \lam_{j}' }{\| \Lam Q \|_{L^2}^2}  \La (1-  \chi_{L\lam_j }) \Lam Q_{\U{\lam_j}} \mid \Lam Q_{\U{\lam_j}} \Ra \\
 & = - \lam_j' + O(( \lam_{j-1}/ \lam_j) + (\lam_j/ \lam_{j+1}) + o_L(1))\max_{ i  \in \calA}  \big(\lambda_i / \lambda_{i+1}\big)^{k/2}.  
} 
Finally, the third term vanishes due to the fact that for each $j< K$,  $L\lam_j  \ll \lam_K \ll \nu$, and hence
\EQ{
  \La \chi_{L\lam_j } \Lam Q_{\U{\lam_j}} \mid \phi( u, \nu)  \Ra  = 0. 
}
Plugging all of this back into~\eqref{eq:xi'-exp} we obtain, 
\EQ{ \label{eq:xi'-est} 
\Big|\xi_j'(t) + \frac{\iota_j}{\| \Lam Q \|_{L^2}^2}  \La  \Lam Q_{\U{\lam_j}} \mid \dot g  \Ra  \Big| \le c_0 \max_{ i  \in \calA}  \big(\lambda_i / \lambda_{i+1}\big)^{k/2}.  
}
for $k =2$, after fixing $L>0$ sufficiently large. The same estimate for $k \ge 3$, i.e., when $\xi_j'(t) = \lam_j'(t)$, is immediate from~\eqref{eq:lam'-refined} since in this case we take $\calZ = \Lam Q$. Thus~\eqref{eq:xi'-est} holds for all $k \ge 2$. 
The estimate~\eqref{eq:lam_j'-beta_j} is then immediate from~\eqref{eq:xi'-est}, the definition of $\be_j$,  and the estimate, 
\EQ{
\Big| \frac{1}{ \| \Lam Q \|_{L^2}^2} \ang{ \uln A( \lam_j) g \mid \dot g} \Big| \lesssim \| \bs g \|_{\E}^2 , 
}
which follows from the first bullet point in Lemma~\ref{lem:A}. 

We prove~\eqref{eq:beta_j'}. We compute, 
\EQ{ \label{eq:beta'-1} 
\beta_j' &=  \frac{\iota_j}{ \| \Lam Q \|_{L^2}^2}  \frac{\lam_j'}{\lam_j} \La \ULam \Lam Q_{\U{\lam_j}} \mid \dot g \Ra  - \frac{\iota_j}{ \| \Lam Q \|_{L^2}^2} \La \Lam Q_{\U{\lam_j}} \mid \p_t  \dot g \Ra\\
&\quad -   \frac{1}{ \| \Lam Q \|_{L^2}^2} \frac{\lam_j'}{\lam_j} \ang{ \lam_j \p_{\lam_j} \uln A( \lam_j) g \mid \dot g} -   \frac{1}{ \| \Lam Q \|_{L^2}^2} \ang{ \uln A( \lam_j) \p_t g \mid \dot g} -    \frac{1}{ \| \Lam Q \|_{L^2}^2} \ang{ \uln A( \lam_j) g \mid \p_t \dot g}.
}
Using~\eqref{eq:g-eq} we arrive at the expression, 
\EQ{
-\La \Lam Q_{\U{\lam_j}} \mid \p_t  \dot g \Ra &=  \ang{ \Lam Q_{\U{\lam_j}} \mid  (\LL_{\calQ}- \LL_{\lam_j}) g} - \ang{ \Lam Q_{\U{\lam_j}} \mid f_{\bfi}( m_n, \iota, \vec \lam) }    \\
&\quad   - \ang{ \Lam Q_{\U{\lam_j}} \mid f_{\bfq}(m_n, \vec \iota,  \vec\lam, g)}  - \ang{ \Lam Q_{\U{\lam_j}} \mid  \dot \phi( u, \chi_{\nu})}, 
}
where in the first term on the right we used that $\LL_{\lam_j} \Lam Q_{\U {\lam_j} } = 0$. Using~\eqref{eq:g-eq}  we obtain, 
\EQ{
- &\ang{ \uln A( \lam_j) \p_t g \mid \dot g}  = - \ang{ \uln A( \lam_j) \dot g \mid \dot g} -   \sum_{i=1}^K \iota_i \lam_{i}'  \ang{ \uln A( \lam_j) \Lam Q_{\U{\lam_i}} \mid \dot g} -  \ang{ \uln A( \lam_j) \phi( u, \nu_n) \mid \dot g}\\
& = - \iota_j \lam_j'  \ang{ \uln A( \lam_j) \Lam Q_{\U{\lam_j}} \mid \dot g} - \sum_{i \neq j} \iota_i \lam_{i}'  \ang{ \uln A( \lam_j) \Lam Q_{\U{\lam_i}} \mid \dot g} -  \ang{ \uln A( \lam_j) \phi( u, \nu_n) \mid \dot g}
}
where we used that $ \ang{ \uln A( \lam_j) \dot g \mid \dot g} =0$. Finally, using~\eqref{eq:g-eq-alt} we have, 
\EQ{
-\ang{ \uln A( \lam_j) g \mid \p_t \dot g} & =  \ang{ \uln A( \lam_j) g \mid \LL_0 g } - \ang{ \uln A( \lam_j) g \mid f_{\bfi}( m_n, \iota, \vec \lam) } \\
& \quad  - \ang{ \uln A( \lam_j) g \mid \ti  f_{\bfq}(m_n, \vec \iota,  \vec \lam, g)}-  \ang{ \uln A( \lam_j) g \mid \dot  \phi( u, \chi_{\nu})}.
}
Plugging these back into~\eqref{eq:beta'-1} and rearranging we have,
\EQ{ \label{eq:beta'-exp} 
\| \Lam Q \|_{L^2}^2 \beta_j' &=  -  \frac{\iota_j }{\lam_j}  \ang{ \Lam Q_{\lam_j} \mid f_{\bfi}( m_n, \iota, \vec \lam) } +  \ang{ \uln A( \lam_j) g \mid \LL_0 g } \\
&\quad  + \ang{ (A(\lam_j) - \uln A( \lam_j)) g \mid \ti  f_{\bfq}(m_n, \vec \iota,  \vec \lam, g)} \\
&\quad + \ang{ \Lam Q_{\U{\lam_j}} \mid ( \LL_{\calQ} - \LL_{\lam_j}) g} + \iota_j \frac{\lam_j' }{\lam_j}\ang{ \big( \frac{1}{\lam_j} \ULam - \U{A}( \lam_j) \big) \Lam Q_{\lam_j} \mid \dot g} \\
&\quad    - \ang{ {A}(\lam_j) \sum_{i =1}^K \iota_i Q_{\lam_i} \mid f_{\bfq}(m_n, \vec \iota,  \vec\lam, g)}   - \ang{ A( \lam_j) g \mid \ti  f_{\bfq}(m_n, \vec \iota,  \vec \lam, g)} \\
&\quad    + \iota_j \ang{ ({A}( \lam_j) -  \frac{1}{\lam_j}\Lam) Q_{\lam_j} \mid f_{\bfq}(m_n, \vec \iota,  \vec\lam, g)}  -    \frac{\lam_j'}{\lam_j} \ang{ \lam_j \p_{\lam_j} \uln A( \lam_j) g \mid \dot g} \\
&\quad + \sum_{ i \neq j} \iota_i \ang{ {A}(\lam_j) Q_{\lam_i} \mid  f_{\bfq}(m_n, \vec \iota,  \vec\lam, g)} \\ 
&\quad - \sum_{i \neq j} \iota_i \lam_{i}'  \ang{ \uln A( \lam_j) \Lam Q_{\U{\lam_i}} \mid \dot g}   - \ang{ \uln A( \lam_j) g \mid f_{\bfi}( m_n, \iota, \vec \lam) } \\
&\quad  - \iota_j \ang{ \Lam Q_{\U{\lam_j}} \mid  \dot \phi( u, \nu)} -  \ang{ \uln A( \lam_j) \phi( u, \nu) \mid \dot g} -  \ang{ \uln A( \lam_j) g \mid \dot  \phi( u, \nu)}
}
We examine each of the terms on the right-hand side above. The leading order contribution comes from the first term, i.e., by Lemma~\ref{lem:interaction} 
\EQ{
-   \frac{\iota_j }{\lam_j\| \Lam Q \|_{L^2}^{2}}  \ang{ \Lam Q_{\lam_j} \mid f_{\bfi}( m_n, \iota, \vec \lam) } = - (\om^2 + O( \eta_0^2))  \frac{ \iota_j \iota_{j+1}}{\lam_j}  \Big( \frac{\lam_j}{ \lam_{j+1}}  \Big)^k + (\om^2 + O( \eta_0^2)) \frac{ \iota_j \iota_{j-1}}{\lam_j}\Big( \frac{\lam_{j-1}}{ \lam_{j}}  \Big)^k
}
The second  and third terms together will have a sign, up to an acceptable error. First, using~\eqref{eq:pohozaev} we have, 
\EQ{
 \ang{ \uln A( \lam_j) g \mid \LL_0 g } \ge -\frac{c_0}{\lam_j} \| g \|_{H}^2 + \frac{1}{\lam_j} \int_{R^{-1} \lam_j}^{R \lam_j} \Big( ( \p_r g)^2 + \frac{k^2}{r^2} g^2  \Big) \, r \ud r
}
To treat the third term, we start by using the definition~\eqref{eq:ti-fq-def} to observe the identity, 
\EQ{\label{eq:tif-exp} 
\ti  f_{\bfq}(m_n, \vec \iota,  \vec \lam, g) 
& = - \frac{k^2}{r^2}( f'(Q_{\lam_j}) - 1)g  - \frac{k^2}{r^2} ( f'( \calQ(m_n, \vec\iota_j, \vec \lam_j)) - f'(Q_{\lam_j}))g + f_{\bfq}(m_n, \vec \iota, \vec\lam, g) 
}
Next, by definition,  
\EQ{
(A( \lam_j) - \U A( \lam_j)) g = - \frac{1}{\lam_j}  \big(\frac{1}{2}q''\big(\frac{r}{\lambda_j}\big) + \frac{\lam_j}{2r}q'\big(\frac{r}{\lambda_j}\big)\big)g
}
The contributions of the second two terms in~\eqref{eq:tif-exp} yield acceptable errors. Indeed, 
\EQ{
\big| \La (A( \lam_j) - \U A( \lam_j)) g & \mid  \frac{k^2}{r^2} ( f'( \calQ(m_n, \vec\iota_j, \vec \lam_j)) - f'(Q_{\lam_j}))g \Ra \Big| \\
& \lesssim \frac{1}{\lam_j} \int_{\ti  R^{-1} \lam_j}^{\ti R \lam_j} g^2  \abs{ f'( \calQ(m_n, \vec\iota_j, \vec \lam_j)) - f'(Q_{\lam_j}))} \, \frac{\ud r}{r}  \le c_0 \frac{ \| g\|_H^2}{\lam_j}
}
with $\ti R$ as in Lemma~\ref{lem:q}, and by~\eqref{eq:fq-est} and the definition of $q$ from Lemma~\ref{lem:q}, 
\EQ{
\abs{\ang{(A( \lam_j) - \U A( \lam_j)) g  \mid  f_{\bfq}(m_n, \vec \iota, \vec\lam, g) }} \lesssim \frac{1}{\lam_j} \|g \|_{L^\infty} \| g \|_H^2  \le c_0 \frac{ \| g\|_H^2}{\lam_j}
}
Putting this together we obtain, 
\EQ{
\Big|  \La (A(\lam_j) - \uln A( \lam_j)) g \mid \ti  f_{\bfq}(m_n, \vec \iota,  \vec \lam, g)\Ra &-  \frac{1}{\lam_j} \int_{0}^\infty  \big(\frac{1}{2}q''\big(\frac{r}{\lambda_j}\big) + \frac{\lam_j}{2r}q'\big(\frac{r}{\lambda_j}\big)\big) \frac{k^2}{r^2} ( f'(Q_{\lam_j}) - 1)g^2  \, r \, \ud r \Big| \\
&\lesssim c_0 \frac{ \| g\|_H^2}{\lam_j}. 
}
%
We show that the remaining terms contribute acceptable errors. For the fourth term a direct calculation gives,  
\EQ{
\abs{ \ang{ \Lam Q_{\U{\lam_j}} \mid ( \LL_{\vec \lam} - \LL_{\lam_j}) g}} & \lesssim \frac{1}{\lam_j} \| g \|_H \sum_{ i \neq j}( \|r^{-1} \Lam Q_{\lam_j}  \Lam Q_{\lam_i}^2 \|_{L^2} + \|r^{-1} \Lam Q_{\lam_j}^2  \Lam Q_{\lam_i} \|_{L^2} ) \\
& \lesssim \frac{1}{\lam_j} \| g \|_H  \Big( \Big( \frac{ \lam_j}{\lam_{j+1}} \Big)^k + \Big( \frac{ \lam_{j-1}}{\lam_{j}} \Big)^k\Big).
}
By~\eqref{eq:L0-A0} along with~\eqref{eq:lam'} we have, 
\EQ{
\abs{\iota_j \frac{\lam_j' }{\lam_j}\ang{ \big( \frac{1}{\lam_j} \ULam - \U{A}( \lam_j) \big) \Lam Q_{\lam_j} \mid \dot g}} \lesssim \frac{c_0}{\lam_j} \| \bs g \|_{\E}^2 . 
}
For the sixth term on the right-hand side of~\eqref{eq:beta'-exp} we note that 
\EQ{
A( \lam_j) \sum_{i =1}^K \iota_i Q_{\lam_i} = A( \lam_j) \calQ( m_n, \vec \iota, \vec \lam), 
}
and hence we may apply~\eqref{eq:A-by-parts} with $g_1 = \calQ( m_n, \vec \iota, \vec \lam)$ and $g_2 = g$ to conclude that 
\EQ{
\Big|   \La {A}(\lam_j) \sum_{i =1}^K \iota_i Q_{\lam_i} \mid f_{\bfq}(m_n, \vec \iota,  \vec\lam, g)\Ra  + \ang{ A( \lam_j) g \mid \ti  f_{\bfq}(m_n, \vec \iota,  \vec \lam, g)} \Big| \le \frac{c_0}{\lam_j} \|g \|_{H}^2,  
}
which takes care of the sixth and seventh terms. 
By~\eqref{eq:L-A} and ~\eqref{eq:fq-est} we see that, 
\EQ{
\abs{ \ang{ ({A}( \lam_j) -  \frac{1}{\lam_j}\Lam) Q_{\lam_j} \mid f_{\bfq}(m_n, \vec \iota,  \vec\lam, g)} } \lesssim \frac{c_0}{ \lam_j}  \| g \|_{H}^2 . 
}
Using the first bullet point in Lemma~\ref{lem:A} and~\eqref{eq:lam'} we estimate the eighth term as follows, 
\EQ{
 \abs{ \frac{\lam_j'}{\lam_j} \ang{ \lam_j \p_{\lam_j} \uln A( \lam_j) g \mid \dot g}} & \lesssim \frac{1}{\lam_j}\| \bs g \|_{\E}^3  \le \frac{c_0}{ \lam_j}  \| g \|_{H}^2 . 
 } 
Next, using~\eqref{eq:A-mismatch} and~\eqref{eq:fq-est} we have, 
\EQ{
 \Big| \sum_{ i \neq j} \iota_i \ang{ {A}(\lam_j) Q_{\lam_i} \mid  f_{\bfq}(m_n, \vec \iota,  \vec\lam, g)}  \Big| \lesssim \frac{c_0}{ \lam_j}  \| g \|_{H}^2 . 
}
 An application of~\eqref{eq:A-mismatch-2} and~\eqref{eq:lam'} gives 
 \EQ{
 \sum_{i \neq j} \abs{\lam_{i}'  \ang{ \uln A( \lam_j) \Lam Q_{\U{\lam_i}} \mid \dot g} }  \le   \frac{c_0}{ \lam_j}  \| g \|_{H}^2 . 
}
Next, consider the twelfth term. Using the first bullet point in Lemma~\ref{lem:A}, and in particular the spatial localization of $\U A( \lam_j)$ we obtain 
\EQ{
\abs{ \ang{ \uln A( \lam_j) g \mid f_{\bfi}( m_n, \iota, \vec \lam) } } \lesssim \| g \|_H \| f_{\bfi}( m_n, \iota, \vec \lam) \|_{ L^2(\ti R^{-1} \lam_j \le r \le \ti R \lam_j)} . 
}
Using the expansion~\eqref{eq:fi-exp} from Lemma~\ref{lem:interaction} we have the pointwise estimate, 
\EQ{\label{eq:fi-point} 
\abs{ f_{\bfi}( m_n, \iota, \vec \lam) } \lesssim r^{-2} \sum_{\ell \neq i }\Lam Q_{\lam_\ell} \Lam Q_{\lam_{i}}^2 . 
}
It follows that 
\EQ{
\| f_{\bfi}( m_n, \iota, \vec \lam)  \|_{L^2(\ti R^{-1} \lam_j \le r \le \ti R \lam_j)} \lesssim \frac{1}{\lam_j}  \Big( \frac{\lam_j}{\lam_{j+1} }\Big)^k +   \frac{1}{\lam_j} \Big( \frac{\lam_{j-1}}{\lam_{j}} \Big)^k. 
}
We obtain
\EQ{
\abs{ \ang{ \uln A( \lam_j) g \mid f_{\bfi}( m_n, \iota, \vec \lam) } } \lesssim \frac{1}{\lam_j}\| g \|_H\Big( \Big( \frac{\lam_j}{\lam_{j+1} }\Big)^k +   \frac{1}{\lam_j} \Big( \frac{\lam_{j-1}}{\lam_{j}} \Big)^k \Big). 
}
Finally, we treat the last line of~\eqref{eq:beta'-exp}. First, using Lemma~\ref{lem:ext-sign} and the definition of $\dot\phi$ in~\eqref{eq:h-def} we have
\EQ{ \label{eq:outer-error} 
\abs{ \ang{ \Lam Q_{\U{\lam_j}} \mid  \dot \phi( u, \chi_{\nu})}} & \lesssim  \frac{1}{\lam_j} \Big( \frac{ \lam_j}{\nu} \Big)^k  E( \bs u(t); \nu(t), 2 \nu(t))  \lesssim \frac{\theta_n}{\lam_j}. 
}
for some sequence $\theta_n \to 0$ as $n \to \infty$. The last two terms in~\eqref{eq:beta'-exp} vanish due to the support properties of $\U A( \lam_j), \phi(u, \nu), \dot \phi(u, \nu)$ and the fact that $\lam_j \le \lam_K \ll \nu$. 

Combining these estimates in~\eqref{eq:beta'-exp} we obtain the inequality, 
\EQ{
\beta' &\ge  \Big({-} \iota_j \iota_{j+1}\om^2 -   c_0\Big) \frac{1}{\lam_{j}} \left( \frac{\lam_{j}}{\lam_{j+1}} \right)^{k} +    \Big(\iota_j \iota_{j-1}\om^2 -  c_0\Big) \frac{1}{\lam_{j}} \left( \frac{\lam_{j-1}}{\lam_{j}} \right)^{k}  \\
&   + \frac{1}{\lam_j} \int_{R^{-1} \lam_j}^{R \lam_j} \Big( ( \p_r g)^2 + \frac{k^2}{r^2} g^2  \Big) \, r \ud r +  \frac{1}{\lam_j} \int_{0}^\infty  \big(\frac{1}{2}q''\big(\frac{r}{\lambda_j}\big) + \frac{\lam_j}{2r}q'\big(\frac{r}{\lambda_j}\big)\big) \frac{k^2}{r^2} ( f'(Q_{\lam_j}) - 1)g^2  \, r \, \ud r \\
&  -  c_0 \frac{ \| \bs g \|_{\cE}^2}{ \lam_{j}}  - c_0\frac{\de_n}{\lam_j}. 
}
where to obtain $c_0 \de_n$ in the last term we enlarged $\de_n$ so as to ensure $\de_n \gg \theta_n$ in the estimate~\eqref{eq:outer-error}. Finally, we use~\eqref{eq:coercive} on the second line above followed by~\eqref{eq:g-bound} and~\eqref{eq:de_n-lower} to conclude the proof. 
\end{proof}


Finally, we prove that, again by enlarging $\epsilon_n$, we can control
the error in the virial identity, see Lemma~\ref{lem:vir}, by $\bfd$. 
\begin{lem}
\label{lem:virial-error}
There exist $C_0, \eta_0 > 0$ depending only on $k$ and $N$
and a decreasing sequence $\eps_n \to 0$ such that
\begin{equation}
|\Omega_{\rho(t)}(\bs u(t))| \leq C_0\bfd(t) 
\end{equation}
for all $t \in [a_n, b_n]$ such that $\eps_n \le \bfd(t) \leq \eta_0$,
$\rho(t) \leq \nu(t)$ and $|\rho'(t)| \leq 1$.
\end{lem}
\begin{proof}
Since $\lim_{n\to\infty}\sup_{t\in[a_n, b_n]}\|\bs u(t)\|_{\cE(\nu(t), 2\nu(t))} = 0$, Lemma~\ref{lem:mod-1} yields
\begin{equation}
\|\bs u(t) - \bs\calQ(m_n, \vec\iota, \vec\lambda(t)) - \bs g(t)\|_{\cE(0, 2\nu(t))} \to 0, \qquad\text{as }n \to \infty.
\end{equation}
Using Remark~\ref{rem:deltan},~\eqref{eq:g-bound} and~\eqref{eq:d-bound} we have $\|\bs g(t)\|_\cE \lesssim \bfd(t)$, hence, after choosing $\eps_n \to 0$ sufficiently large, it suffices to check that
\begin{equation}
|\Omega_{\rho(t)}(\bs\calQ(m_n, \vec\iota, \vec\lambda(t)))|
\leq C_0\bfd(t),
\end{equation}
which in turn will follow from
\begin{equation}
\int_0^\infty \Big|(\partial_r \calQ(m_n, \vec\iota, \vec\lambda(t)))^2
- k^2 \frac{\sin^2 \calQ(m_n, \vec\iota, \vec\lambda(t))}{r^2}\Big|\,r\vd r \leq C_0\bfd(t).
\end{equation}
Recall that $\Lambda Q_\lambda = r \partial_r Q_\lambda = k\sin Q_\lambda$,
so it suffices to estimate the cross terms.
It is easy to check that
\begin{equation}
\big|\sin^2\big(\sum a_i\big) - \sum \sin^2 a_i\big| \leq 4\sum_{i\neq j}|\sin a_i||\sin a_j|.
\end{equation}
Invoking the bound
\begin{equation}
\int_0^\infty |\Lambda Q_\lambda(r)\Lambda Q_\mu(r)|\frac{\vd r}{r} \lesssim (\lambda / \mu)^{k/2}
\end{equation}
from \cite[p. 1277]{JL1}, we obtain the claim.
\end{proof}


\section{Conclusion of the proof} \label{sec:conclusion} 
  
\subsection{The scale of the $K$-th bubble}

As mentioned in the Introduction, the $K$-th bubble is of particular importance.
We introduce below a function $\mu$ which is well-defined on every $[a_n, b_n]$,
and close to $\lambda_K$ on time intervals where the solution approaches a multi-bubble configuration.
\begin{defn}[The scale of the $K$-th bubble] \label{def:mu} 
For all $t \in I$, we set
\begin{equation}
\label{eq:mu-def}
\mu(t) := \sup\big\{r: E(\bs u(t); r) = (N - K + 1/2)E(\bs Q) + E(\bs u^*)\big\}.
\end{equation}
\end{defn}
Note that $K > 0$ implies $0 < (N - K + 1/2)E(\bs Q) + E(\bs u^*) < E(\bs u)$,
hence $\mu(t)$ is a well-defined finite positive number for all $t \in I$.
By Lemma~\ref{lem:ext-sign},
\begin{equation}
\lim_{n\to\infty}\sup_{t \in [a_n, b_n]}\big|E(\bs u(t); \nu(t)) - (N-K)E(\bs Q) - E(\bs u^*)\big| = 0,
\end{equation}
which implies $\mu(t) \leq \nu(t)$ for all $n$ large enough and $t \in [a_n, b_n]$,
thus $\mu(t) \ll \mu_{K+1}(t)$ as $n \to \infty$.
\begin{lem}
\label{lem:mu-prop}
The function $\mu$ defined above has the following properties:
\begin{enumerate}[(i)]
\item its Lipschitz constant is $\leq 1$,
\item for any $\epsilon > 0$ there exist $0 < \delta \leq \eta_0$ and $n_0 \in \bN$
such that $t \in [a_n, b_n]$ with $n \geq n_0$ and $\bfd(t) \leq \delta$
imply $|\mu(t)/ \lambda_K(t) - 1| \leq \epsilon$,
where $\lambda_K(t)$ is the modulation parameter defined in Lemma~\ref{lem:mod-1},
\item if $t_n \in [a_n, b_n]$, $1 \ll r_n \ll \mu_{K+1}(t_n)/\mu(t_n)$ and $\lim_{n \to \infty}\bs\de_{r_n\mu(t_n)}(t_n) = 0$, then $\lim_{n\to \infty}\bfd(t_n) = 0$.
\end{enumerate}
\end{lem}
\begin{proof}
Let $s, t \in I$.
We prove that $|\mu(s) - \mu(t)| \leq |s-t|$.
Assume, without loss of generality, $\mu(t) \geq \mu(s)$. Of course, we can also assume $\mu(t) > |s-t|$.
By \eqref{eq:ext-energy-monoton},
\begin{equation}
E(\bs u(s); \mu(t) - |s-t|) \geq E(\bs u(t); \mu(t)) = (N-K+1/2)E(\bs Q) + E(\bs u^*),
\end{equation}
which implies $\mu(s) \geq \mu(t) - |s-t|$.

In order to prove (ii), it suffices to check that
\begin{align}
E(\bs u(t); (1+\epsilon)\lambda_K(t)) &< (N - K +1/2)E(\bs Q) + E(\bs u^*), \\
E(\bs u(t); (1-\epsilon)\lambda_K(t)) &> (N - K +1/2)E(\bs Q) + E(\bs u^*).
\end{align}
By \eqref{eq:E>nu}, this will follow from
\begin{align}
E(\bs u(t); (1+\epsilon)\lambda_K(t), \nu(t)) &< E(\bs Q)/2, \\
E(\bs u(t); (1-\epsilon)\lambda_K(t), \nu(t)) &> E(\bs Q)/2.
\end{align}
We use \eqref{eq:u-decomp}. By \eqref{eq:d-g-lam}, $\|\bs g\|_\cE \ll 1$
when $\delta \ll 1$ and $n_0 \gg 1$. Thus, it suffices to see that
\begin{align}
E(\bs\calQ(m_n, \vec\iota, \vec\lambda); (1+\epsilon)\lambda_K(t), \nu(t)) &< E(\bs Q)/2, \\
E(\bs\calQ(m_n, \vec\iota, \vec\lambda); (1-\epsilon)\lambda_K(t), \nu(t)) &> E(\bs Q)/2
\end{align}
whenever $\sum_{j=1}^K \lambda_j(t)/\lambda_{j+1}(t) \ll 1$,
which is obtained directly from the definition of $\bs \calQ$.

We now prove (iii). Let $R_n$ be a sequence such that $r_n\mu(t_n) \ll R_n \ll \mu_{K+1}(t_n)$.
Without loss of generality, we can assume $R_n \geq \nu(t_n)$, since it suffices to replace $R_n$ by $\nu(t_n)$
for all $n$ such that $R_n < \nu(t_n)$.
Let $M_n, m_n, \vec\iota_n, \vec\lambda_n$ be parameters such that
\begin{equation}
\label{eq:conv-delta-iii}
\| u(t_n) - \calQ( m_n, \vec \iota_n, \vec \lam_n) \|_{H( r \le r_n\mu(t_n))}^2 + \| \dot u(t_n) \|_{L^2(r \le r_n\mu(t_n))}^2 + \sum_{j = 1}^{M-1} \Big(\frac{ \lam_{n, j}}{ \lam_{n, j+1}}\Big)^k \to 0,
\end{equation}
which exist by the definition of the localized distance function \eqref{eq:delta-def}. Since
\begin{equation}
K - \frac 12 \leq \liminf_{n\to\infty}E(\bs u(t_n); 0, r_n\mu(t_n)) \leq \limsup_{n\to\infty}E(\bs u(t_n); 0, r_n\mu(t_n)) \leq K,
\end{equation}
we have $M_n = K$ for $n$ large enough.
We set $\lambda_{n, j} := \mu_j(t_n)$ and $\iota_{n, j} := \sigma_j$ for $j > K$.
We claim that
\begin{equation}
\label{eq:d-conv-0}
\lim_{n\to \infty}\bigg(\| \bs u(t) - \bs u^*(t) - \bs\calQ(m_\Delta, \vec\iota_n, \vec\lambda_n) \|_{\cE}^2 + \sum_{j=1}^{N}\Big(\frac{ \lam_{n, j}}{\lam_{n, j+1}}\Big)^{k}\bigg) = 0.
\end{equation}
By the definition of $\bfd$, the proof will be finished.
First, we observe that $\lambda_{n, K} \ll r_n\mu(t_n)$, so $\lambda_{n, K} / \lambda_{n, K+1} \to 0$.
In the region $r \leq r_n\mu(t_n)$, convergence follows from \eqref{eq:conv-delta-iii},
since the energy of the exterior bubbles asymptotically vanishes there.
In the region $r \ge R_n$, the energy of the interior bubbles vanishes, hence it suffices to apply Lemma~\ref{lem:ext-sign}
and recall that $R_n \geq \nu(t_n)$.
In particular
\begin{equation}
\lim_{n\to\infty}E(\bs u(t_n); 0, r_n\mu(t_n)) = KE(\bs Q), \qquad \lim_{n\to\infty} E(\bs u(t_n); R_n) = (N-K)E(\bs Q) + E(\bs u^*),
\end{equation}
which implies
\begin{equation}
\lim_{n\to\infty} E(\bs u(t_n); r_n\mu(t_n), R_n) = 0,
\end{equation}
and \eqref{eq:H-E-comp} yields convergence of the error also in the region $r_n\mu(t_n) \leq r \leq R_n$.
\end{proof}

Our next goal is to prove that the minimality of $K$ (see Definition~\ref{def:K-choice}) implies a lower bound
on the length of the collision intervals. First, we have the following fact.
\begin{lem}
\label{lem:evol-of-Q}
If $m_n \in \bZ$, $\iota_n \in \{-1, 1\}$, $0 < r_n \ll \mu_n \ll R_n$, $0 < t_n \ll \mu_n$ and $\bs u_{n}$ a sequence of
solutions of \eqref{eq:wmk} such that $\bs u_n(t)$ is defined for $t \in [0, t_n]$ and
\begin{equation}
\lim_{n\to\infty}\|\bs u_n(0) - (m_n\bs\pi + \iota_n \bs Q_{\mu_n})\|_{\cE(r_n, R_n)} = 0,
\end{equation}
then
\begin{equation}
\lim_{n\to\infty}\sup_{t\in [0, t_n]}\|\bs u_n(t) - (m_n\bs\pi + \iota_n \bs Q_{\mu_n})\|_{\cE(r_n + t, R_n-t)} = 0.
\end{equation}
\end{lem}
\begin{proof}
Without loss of generality, we can assume $m_n = 0$, $\iota_n = 1$ and $\mu_n = 1$.
After these reductions, the conclusion directly follows from \cite[Lemma 3.4]{Cote}.
\end{proof}

\begin{lem}
\label{lem:cd-length}
If $\eta_1 > 0$ is small enough, then for any $\eta \in (0, \eta_1]$ there exist $\epsilon \in (0, \eta)$ and $C_{\bs u} > 0$
having the following property.
If $[c, d] \subset [a_n, b_n]$, $\bfd(c) \leq \epsilon$, $\bfd(d) \leq \epsilon$ and there exists $t_0 \in [c, d]$
such that $\bfd(t_0) \geq \eta$, then
\begin{equation}
\label{eq:cd-length}
d - c \geq C_{\bs u}^{-1}\max(\mu(c), \mu(d)).
\end{equation}
\end{lem}
\begin{proof}
We argue by contradiction. If the statement is false, then there exist $\eta > 0$, a decreasing sequence $(\epsilon_n)$
tending to $0$, an increasing sequence $(C_n)$ tending to $\infty$
and intervals $[c_n, d_n] \subset [a_n, b_n]$ (up to passing to a subsequence in the sequence of the collision intervals $[a_n, b_n]$) such that
$\bfd(c_n) \leq \epsilon_n$, $\bfd(d_n) \leq \epsilon_n$, there exists $t_n \in [c_n, d_n]$ such that $\bfd(t_n) \geq \eta$
and $d_n - c_n \leq C_n^{-1}\max(\mu(c_n), \mu(d_n))$.
We will check that, up to adjusting the sequence $\epsilon_n$, $[c_n, d_n] \in \calC_{K-1}(\epsilon_n, \eta)$ for all $n$,
contradicting Definition~\ref{def:K-choice}.

The first and second requirement in Definition~\ref{def:collision} are clearly satisfied.
It remains to construct a function $\rho_{K-1}:[c_n, d_n] \to [0, \infty)$ such that
\begin{equation}
\label{eq:dK-1conv0}
\lim_{n\to\infty}\sup_{t\in [c_n, d_n]}\bfd_{K-1}(t; \rho_{K-1}(t)) = 0.
\end{equation}
Assume $\mu(c_n) \geq \mu(d_n)$ (the proof in the opposite case is very similar).
Let $r_n$ be a sequence such that $\lambda_{K-1}(c_n) \ll r_n \ll \lambda_K(c_n)$ (recall that $\lambda_K(c_n)$
is at main order equal to $\mu(c_n)$ and that $\lambda_0(t) = 0$ by convention).
Set $\rho_{K-1}(t) := r_n + (t - c_n)$ for $t \in [c_n, d_n]$.
Recall that $\vec \sigma_n \in \{-1, 1\}^{N-K}$ and $\vec \mu(t) \in (0, \infty)^{N-K}$ are defined in Lemma~\ref{lem:ext-sign}.
Let $\iota_{n}$ be the sign of the $K$-th bubble at time $c_n$, and set
$\wt \sigma := (\iota_n, \vec\sigma_n) \in \{-1, 1\}^{N-(K-1)}$
and $\wt\mu(t) := (\mu(c_n), \vec\mu(t)) \in (0, \infty)^{N-(K-1)}$.
Let $R_n$ be a sequence such that $\nu_n(c_n) \ll R_n \ll \mu_{K+1}(c_n)$.
Applying Lemma~\ref{lem:evol-of-Q} with these sequences $r_n, R_n$ and $\bs u_n(t) := \bs u(c_n + t)$, we obtain
\begin{equation}
\lim_{n\to\infty}\sup_{t\in[c_n, d_n]}\|\bs u(t) - \bs\calQ(m_\Delta, \wt\sigma_n, \wt\mu(t))\|_{\cE(\rho_{K-1}(t), \infty)} = 0,
\end{equation}
implying \eqref{eq:dK-1conv0}
\end{proof}
\begin{rem}
We denote the constant $C_{\bs u}$ to stress that it depends on the solution $\bs u$
and is obtained in a non-constructive way as a consequence of the assumption that $\bs u$
does not satisfy the continuous time soliton resolution.
\end{rem}
\subsection{Demolition of the multi-bubble} \label{ssec:bub-dem} 
Recall the following notion from Real Analysis.
If $X \subset \bR$, $U: X \to \bR\cup\{+\infty\}$ and $t_0 \in X$,
we say that $t_0$ is a \emph{local minimum from the right} if there exists $t_1 > t_0$
such that $U(t_0) \leq U(t)$ for all $t \in X \cap (t_0, t_1)$.
Similarly, we say that $t_0$ is a \emph{local minimum from the left} if there exists $t_1 < t_0$
such that $U(t_0) \leq U(t)$ for all $t \in X \cap (t_1, t_0)$.


\begin{defn}\label{def:w-int} [Weighted interaction energy]
On each collision interval $[a_n, b_n]$, we define the function $U: [a_n, b_n]\to \conj\bR_+$ as follows:
\begin{itemize}
\item if $\bfd(t) \geq \eta_0$, then $U(t) := +\infty$.
\item if $\bfd(t) < \eta_0$, then $U(t) := \max_{i \in \cA}\big(2^{-i}\xi_i(t) / \lambda_{i+1}(t)\big)^k$, where $\lambda_{i+1}$ and $\xi_i$ are the modulation parameter
and its refinement defined above, see Lemma~\ref{cor:modul}.
\end{itemize}
\end{defn}
\begin{rem}
Continuity of $\bfd$, $\xi_i$ and $\lambda_i$
implies that $U$ is finite and continuous in a neighborhood of any point where it is finite.
\end{rem}
\begin{lem}
\label{lem:ejection}
Let $k \geq 2$. If $\eta_0$ is small enough,
then there exists $C_0 \geq 0$ depending only on $k$ and $N$ such that the following is true.
If $t_0$ is a local minimum from the right of $U$ such that $U(t_0) < +\infty$
and $t_* \geq t_0$ is such that $U(t) < \infty$ for all $t \in [t_0, t_*]$, then
\begin{gather}
\label{eq:lambdaK-bd}
\frac 34 \lambda_K(t_0) \leq \lambda_K(t_*) \leq\frac 43\lambda_K(t_0), \\
\label{eq:int-d-bd}
\int_{t_0}^{t_*} \bfd(t)\ud t\leq C_0\bfd(t_*)^\frac 2k \lambda_K(t_0).
\end{gather}
An analogous statement is true if $t_*$ is a local minimum from the left.
\end{lem}
\begin{rem}
Since $\bfd(t_0) \leq \eta_0$ and $\bfd(t_*) \leq \eta_0$ are small,
$\lambda_K$ differs from $\mu$ by a small relative error,
so in the formulation of the lemma we could just as well write $\mu$
instead of $\lambda_K$.
\end{rem}
\begin{proof}[Proof of Lemma~\ref{lem:ejection}]
\textbf{Step 1.}
We can assume $t_* > t_0$. For $j \in \cA$, denote $\wt\xi_j(t) := 2^{-j}\xi_j(t)/\lambda_{j+1}(t)$ and let
\begin{equation}
\begin{aligned}
\cA_0 &:= \{j \in \cA: U(t_0) = \wt \xi_j(t_0)^k\} = \{j \in \cA: \wt \xi_j(t_0) = \max_{i\in \cA}\wt\xi_i(t_0)\}, \\
\wt\cA_0 &:= \{j \in \cA_0: \wt\xi_j'(t_0) \geq 0\}.
\end{aligned}
\end{equation}
Since $t_0$ is a local minimum from the right of $U$, $\wt\cA_0 \neq \emptyset$. Let $j_0 := \min\wt\cA_0 \in \cA$.

We now define by induction a sequence of times $t_0 \leq t_1 \leq  \ldots \leq  t_{l_*} = t_*$
and a sequence of elements of $\cA$, $j_0 > j_1 > \ldots > j_{l_* -1}$, in the following way.
Assume $t_0 \leq  t_1 \leq \ldots \leq t_{l-1}$ and $j_0 > j_1 > \ldots > j_{l-1}$ are already defined.
We set
\begin{equation}
t_l := \sup\big\{t \in [t_{l-1}, t_*]: \wt\xi_j(\tau) \leq \wt\xi_{j_{l-1}}(\tau)\ \text{for all }\tau \in [t_{l-1}, t)
\text{ and }j \in \cA\text{ such that }j < j_{l-1}\big\}.
\end{equation}
If $t_l = t_*$, then we set $l_* := l$ and terminate the procedure.
If not, let
\begin{equation}
\cA_l := \{j \in \cA: j < j_{l-1} \text{ and }\wt \xi_j(t_l) = \wt\xi_{j_{l-1}}(t_l)\}.
\end{equation}
By the definition of $t_l$ and continuity, $\cA_l \neq \emptyset$.
We set $j_l := \min\cA_l$.

\noindent
\textbf{Step 2.}
We check that $t_{l} > t_{l-1}$ for $l = 1, \ldots, l_*$.

In order to prove that $t_1 > t_0$, we need to show that there exists $t > t_0$ such that
$\wt \xi_j(\tau) \leq \wt\xi_{j_0}(\tau)$ for all $\tau \in [t_0, t)$ and $j \in \cA$ such that $j < j_0$.
Since $\cA$ is a finite set, it suffices to check this separately for each $j\in \cA$.
If $j \notin \cA_0$, the claim is clear, by continuity.
If $j \in \cA_0 \setminus \wt\cA_0$, then $\wt \xi_j(t_0) = \wt\xi_{j_0}(t_0)$, $\wt\xi_j'(t_0) < 0$ and $\wt\xi_{j_0}'(t_0) \geq 0$, again implying the claim.

For $l \geq 1$, the definition of $j_{l}$ implies that $\wt\xi_j(t_{l}) < \wt\xi_{j_{l}}(t_{l})$ for all $j < j_{l}$.
Writing $l-1$ instead of $l$, we get $\wt\xi_j(t_{l-1}) < \wt\xi_{j_{l-1}}(t_{l-1})$ for all $j < j_{l-1}$, whenever $l \geq 2$.
Thus, by continuity, $t_l > t_{l-1}$.

\noindent
\textbf{Step 3.}
By induction with respect to $l$, we show that there exists a constant $C_0$ depending only on $k$ and $N$ such that
for all $l \in \{1, \ldots, l_*\}$ we have
\begin{gather}
\label{eq:collaps-induction}
\int_{t_{l-1}}^{t_l} \bfd(t)\ud t \leq C_0\bfd(t_l)^\frac 2k\lambda_{j_{l-1}+1}(t_{l-1}), \\
\label{eq:maxquot-induction}
\wt\xi_{j_{l-1}}(t) \geq \frac 12 \wt \xi_j(t), \qquad\text{for all }t \in (t_{l-1}, t_l)\text{ and }j > j_{l-1}.
\end{gather}

Suppose \eqref{eq:collaps-induction} is proved for $l \in \{1, \ldots, l_0\}$ and let
$T_0 \in (t_{l_0}, t_{l_0+1}]$ be the largest number such that
\begin{gather}
\label{eq:collaps-boot-ass}
\int_{t_{l_0}}^{T_0} \bfd(t)\ud t \leq 2C_0\bfd(T_0)^\frac 2k\lambda_{j_{l_0}+1}(t_{l_0}), \\
\label{eq:maxquot-boot-ass}
\wt\xi_{j_{l_0}}(t) \geq \frac 14 \wt \xi_j(t), \qquad\text{for all }t \in (t_{l_0}, T_0)\text{ and }j > j_{l_0}, \\
\label{eq:increase-boot-ass}
\xi_{j_{l_0}}(t) \geq \frac 34 \xi_{j_{l_0}}(t_{l_0}), \qquad\text{for all }t \in (t_{l_0}, T_0).
\end{gather}
It suffices to prove that
\begin{gather}
\label{eq:collaps-boot}
\int_{t_{l_0}}^{T_0} \bfd(t)\ud t \leq C_0\bfd(T_0)^\frac 2k\lambda_{j_{l_0}+1}(t_{l_0}), \\
\label{eq:maxquot-boot}
\wt\xi_{j_{l_0}}(t) \geq \frac 12 \wt \xi_j(t), \qquad\text{for all }t \in (t_{l_0}, T_0)\text{ and }j > j_{l_0}, \\
\label{eq:increase-boot}
\xi_{j_{l_0}}(t) \geq \frac 78 \xi_{j_{l_0}}(t_{l_0}), \qquad\text{for all }t \in (t_{l_0}, T_0).
\end{gather}
It will be convenient to assume $T_0 = t_* = t_{l_0 + 1}$, which is allowed.
Also, in order to simplify the notation, we write $l$ instead of $l_0$ in the induction step which follows.

The first observation is that if $j_{l} < j \leq j_{l-1}$, then $\lambda_j(t)$
is ``almost constant'' on the time interval $(t_{l}, t_{l+1})$.
More precisely, we claim that
\begin{equation}
\label{eq:lambdaj-const}
|\lambda_j(t)/ \lambda_j(t_{l}) - 1| \leq c_0,\qquad\text{if }j > j_{l}\text{ and }t > t_l,
\end{equation}
where $c_0$ can be made arbitrarily small by taking $\eta_0$ small enough.
Indeed, $|\lambda_j'(t)| \lesssim \bfd(t)$, so \eqref{eq:collaps-boot-ass}
implies the claim (we stress again that $C_0$ will not depend on $\eta_0$).

The definition of $t_{l+1}$ implies
\begin{equation}
\label{eq:xi-bigger}
\wt\xi_{j_{l}}(t) \geq \wt \xi_j(t), \qquad\text{for all }t \in (t_{l}, t_{l+1})\text{ and }j \leq j_{l},
\end{equation}
so \eqref{eq:maxquot-boot-ass} yields
\begin{equation}
\label{eq:xi-really-bigger}
\max_{i \in \calA}\wt\xi_{i}(t) \lesssim \wt\xi_{j_{l}}(t), \qquad\text{for all }t \in (t_{l}, t_{l+1}).
\end{equation}

The bound \eqref{eq:beta_j'} yields for all $t \in (t_l, t_{l+1})$
\begin{equation}
\begin{aligned}
\label{eq:beta_j'-all-2}
 \lambda_{j_l}(t)\beta_{j_l}'(t) &\ge  ({-}\iota_{j_l} \iota_{j_l+1}\om^2 -   c_0)\big(2^{j_l}\wt\xi_{j_l}(t)\big)^k  + (\iota_{j_l} \iota_{j_l-1}\om^2 -  c_0) \big(2^{j_l-1}\wt\xi_{j_l-1}(t)\big)^k    \\
&\quad  -  c_0\max_{i \in \calA}\big(2^{i}\wt\xi_{i}(t)\big)^k,
\end{aligned}
\end{equation}
with the convention $\xi_0(t) = 0$. By \eqref{eq:xi-bigger}, $2^{j_l-1}\wt \xi_{j_l-1}(t) \leq \frac 12 2^{j_l}\wt\xi_{j_l}(t)$. Taking $c_0$ small enough and applying \eqref{eq:xi-really-bigger}, we obtain
\begin{equation}
\begin{aligned}
\label{eq:beta_j'-all-3}
 \lambda_{j_l}(t)\beta_{j_l}'(t) &\ge \frac{\omega^2}{4}\big(2^{j_l}\wt\xi_{j_l}(t)\big)^k
 \quad\Rightarrow\quad \beta_{j_l}'(t) \geq c_1 \frac{\xi_{j_l}(t)^{k-1}}{\lambda_{j_l+1}(t_l)^k},
   \end{aligned}
\end{equation}
where $c_1 > 0$ depends only on $k$ and $N$, and in the last step we used \eqref{eq:lambdaj-const}.

With $c_2 > 0$ to be determined, consider the auxiliary function
\begin{equation}
\phi(t) := \beta_{j_l}(t) + c_2\big(\xi_{j_l}(t)/\lambda_{j_l+1}(t_l)\big)^\frac k2.
\end{equation}
The Chain Rule gives
\begin{equation}
\phi'(t) = \beta_{j_l}'(t) + c_2\frac k2\lambda_{j_l+1}(t_l)^{-\frac k2}\xi_{j_l}(t)^{\frac k2 -1}\xi_{j_l}'(t).
\end{equation}
We have $|\xi_{j_l}'(t)| \leq c_3 (\xi_{j_l}(t) / \lambda_{j_l +1}(t_l))^\frac k2$, with $c_3$ depending only on $k$ and $N$,
hence \eqref{eq:beta_j'-all-3} implies
\begin{equation}
\phi'(t) \geq \frac{c_3}{\lambda_{j_l+1}(t_l)}\bigg(\frac{\xi_{j_l}(t)}{\lambda_{j_l+1}(t_l)}\bigg)^{k-1} \geq \frac{c_4}{\lambda_{j_l+1}(t_l)}\phi(t)^\frac{2k-2}{k},
\end{equation}
with $c_2, c_3, c_4$ depending only on $k$ and $N$.
The last inequality yields
\begin{equation}
\label{eq:fund-thm-for-f}
\big(\lambda_{j_l+1}(t_l)\phi(t)^{2/k}\big)' \gtrsim \phi(t)\quad\Rightarrow\quad \int_{t_l}^{t_{l+1}}\phi(t)\ud t \lesssim
\lambda_{j_l+1}(t_l)\phi(t_{l+1})^{2/k} \lesssim \bfd(t_{l+1})^{2/k}\lambda_{j_l+1}(t_l).
\end{equation}
If we consider $\wt \phi(t) := \beta_{j_l}(t) + \frac{c_2}{2}\big(\xi_{j_l}(t)/\lambda_{j_l+1}(t_l)\big)^\frac k2$
instead of $\phi$, then the computation above shows that $\wt \phi$ is increasing.
From \eqref{eq:xijl-init}, we have $\wt \phi(t_l) \geq 0$, so $\wt \phi(t) \geq 0$ for all $t \in (t_l, t_{l+1})$,
implying $\bfd(t) \lesssim \phi(t)$.
Thus, \eqref{eq:fund-thm-for-f} yields \eqref{eq:collaps-boot} if $C_0$ is sufficiently large
(but depending on $k$ and $N$ only).

We now prove \eqref{eq:increase-boot}.
By the definition of $t_l$ and the fact that $j_l < j_{l-1}$, we have
$\wt\xi_{j_l}(\tau) \leq \wt\xi_{j_{l-1}}(\tau)$ for all $\tau \in [t_{l-1}, t_l)$.
By the definition of $j_l$, $\wt\xi_{j_l}(t_l) = \wt\xi_{j_{l-1}}(t_l)$,
in particular we have
\begin{equation}
\label{eq:xitl'-at-tl}
\wt\xi_{j_l}'(t_l) \geq \wt\xi_{j_{l-1}}'(t_l).
\end{equation}
Recalling that $\wt\xi_j(t) = 2^{-j}\xi_j(t) / \lambda_{j+1}(t)$, we find
\begin{equation}
\label{eq:xijl-init-0}
\begin{aligned}
2^{-j_l}\frac{\xi_{j_l}'(t_l)}{\lambda_{j_l+1}(t_l)}
- 2^{-j_l}\frac{\xi_{j_l}(t_l)\lambda_{j_l+1}'(t_l)}{\lambda_{j_l+1}(t_l)^2}
\geq 2^{-j_{l-1}}\frac{\xi_{j_{l-1}}'(t_l)}{\lambda_{j_{l-1}+1}(t_l)}
- 2^{-j_{l-1}}\frac{\xi_{j_{l-1}}(t_l)\lambda_{j_{l-1}+1}'(t_l)}{\lambda_{j_{l-1}+1}(t_l)^2}.
\end{aligned}
\end{equation}
Since $\lambda_{j_l+1}(t_l) /\lambda_{j_{l-1}+1}(t_l) + \xi_{j_l}(t_l)/\lambda_{j_l+1}(t_l)+\xi_{j_{l-1}}(t_l)/\lambda_{j_{l-1}+1}(t_l)$ is small when $\eta_0$ is small and, see Lemma~\ref{cor:modul},
\begin{equation}
|\lambda_{j_l+1}'(t_l)| + |\xi_{j_l-1}'(t_l)| + |\lambda_{j_{l-1}+1}'(t_l)| \lesssim \max_{i \in \cA}(\xi_i / \lambda_{i+1})^{k/2},
\end{equation}
we obtain
\begin{equation}
\label{eq:xijl-init}
\xi_{j_l}'(t_l) \geq -c_0 \max_{i \in \cA}(\xi_i / \lambda_{i+1})^{k/2},
\end{equation}
where $c_0$ can be made arbitrarily small upon taking $\eta_0$ small.

By \eqref{eq:xijl-init} and \eqref{eq:lam_j'-beta_j},
we have $\beta_{j_l}(t_l) \geq -c_0 (\xi_{j_l}(t_l)/\lambda_{j_l+1}(t_l))^\frac k2$,
where $c_0$ can be made as small as needed, and
\begin{equation}
\beta_{j_l}'(t) \geq (3/4)^{k-1}c_1\frac{\xi_{j_l}(t_l)^{k-1}}{\lambda_{j_l+1}(t_l)^k}.
\end{equation}
We deduce that $\xi_{j_l}'(t) \geq 0$ provided
\begin{equation}
t - t_l \geq \frac{2c_0}{c_1}(4/3)^{k-1}\xi_{j_l}(t_l)(\xi_{j_l}(t_l)/\lambda_{j_l+1}(t_l))^{-\frac k2}.
\end{equation}
But, if the opposite inequality is satisfied, the bound $|\xi_{j_l}'(t)| \lesssim (\xi_{j_l}(t_l)/\lambda_{j_l+1}(t_l))^\frac k2$
yields \eqref{eq:increase-boot}, if $c_0$ is small enough.
In fact, the argument gives the bound with $\frac 78$ replaced by $1-c_0$, where $c_0 > 0$ is as small as we want.
Combining this with \eqref{eq:lambdaj-const}, we obtain in particular
\begin{equation}
\label{eq:wtxi-nondecrease}
\wt \xi_{j_{l}}(t_{l+1}) \geq (1-c_0) \wt \xi_{j_{l}}(t_{l}),
\end{equation}
with $c_0 > 0$ arbitrarily small.

Finally, we prove \eqref{eq:maxquot-boot}.
By \eqref{eq:increase-boot} and \eqref{eq:lambdaj-const}, it suffices to show that
\begin{equation}
\wt\xi_{j_l}(t_l) \geq \frac 34 \wt\xi_j(t_l), \qquad\text{for all }j > j_l.
\end{equation}
Let $l' \leq l$ be such that $j_{l'} < j \leq j_{l'-1}$. The definition of $j_{l'}$ yields
$\wt\xi_{j_l'}(t_{l'}) \geq \wt\xi_j(t_{l'})$, so it suffices to check that
\begin{equation}
\wt\xi_{j_{\wt l}}(t_{\wt l}) \geq (3/4)^{\frac{1}{K}} \wt\xi_{j_{\wt l-1}}(t_{\wt l-1}), \qquad\text{for all }\wt l,
\end{equation}
with $c_0 > 0$ small, and use this inequality $l - l'$ times. The last inequality
follows from \eqref{eq:xitl'-at-tl} and \eqref{eq:wtxi-nondecrease}.

\noindent
\textbf{Step 4.}
Taking the sum over $l$ of \eqref{eq:collaps-boot}, we get \eqref{eq:int-d-bd}.
The bound \eqref{eq:lambdaK-bd} follows from \eqref{eq:lambdaj-const}.
\end{proof}

Starting from now, $\eta_0 > 0$ is fixed so that Lemma~\ref{lem:ejection} holds and Lemma~\ref{lem:cd-length}
can be applied with $\eta = \eta_0$.
We also fix $\epsilon > 0$ to be the value given by Lemma~\ref{lem:cd-length} for $\eta = \eta_0$.

Directly from the definitions, we see that there exists $C_1 > 0$ such that
\begin{equation}
\label{eq:U-bound-by-d}
\bfd(t) < \eta_0\quad\text{implies}\quad C_1^{-1} \bfd(t)^2 \leq U(t) \leq C_1\bfd(t)^2.
\end{equation}
Recall that $\bfd(a_n) = \bfd(b_n) = \epsilon_n$ and
$\bfd(t) \geq \epsilon_n$ for all $t \in [a_n, b_n]$.
\begin{lem}
\label{lem:cedf}
There exists $\theta_0 > 0$ such that for any sequence satisfying $\epsilon_n \ll \theta_n \leq \theta_0$
and for all $n$ large enough there exists a partition of the interval $[a_n, b_n]$
\begin{equation}
\begin{aligned}
a_n = e^L_{n, 0} \leq e^R_{n, 0} \leq c^R_{n, 0} \leq d^R_{n, 0} \leq f^R_{n, 0}
\leq f^L_{n, 1} \leq d^L_{n, 1} \leq c^L_{n, 1} \leq e^L_{n, 1} \leq \ldots \leq e^R_{n, N_n} = b_n,
\end{aligned}
\end{equation}
having the following properties.
\begin{enumerate}[(1)]
\item
\label{it:1}
For all $m \in \{0, 1, \ldots, N_n\}$ and $t \in [e_{n, m}^L, e_{n, m}^R]$, $\bfd(t) \leq \eta_0$, and
\begin{equation}
\label{eq:error-on-eLeR}
\int_{e_{n, m}^L}^{e_{n, m}^R}\bfd(t)\ud t \leq C_2 \theta_n^{2/k} \min(\mu(e_{n, m}^L), \mu(e_{n, m}^R)),
\end{equation}
where $C_2\geq 0$ depends only on $k$ and $N$.
\item \label{it:2}
For all $m \in \{0, 1, \ldots, N_n-1\}$ and $t \in [e_{n, m}^R, c_{n, m}^R] \cup [f_{n, m}^R, f_{n, m+1}^L] \cup [c_{n, m+1}^L, e_{n, m+1}^L]$, $\bfd(t) \geq \theta_n$.
\item \label{it:3}
For all $m \in \{0, 1, \ldots, N_n-1\}$ and $t \in [c_{n,m}^R, f_{n,m}^R]\cup [f_{n,m+1}^L, c_{n,m+1}^L]$, $\bfd(t) \geq \epsilon$.
\item \label{it:4}
For all $m \in \{0, 1, \ldots, N_n-1\}$, $\bfd(d_{n,m}^R) \geq \eta_0$
and $\bfd(d_{n,m+1}^L) \geq \eta_0$.
\item \label{it:5}
For all $m \in \{0, 1, \ldots, N_n-1\}$, $\bfd(c_{n, m}^R) = \bfd(c_{n, m+1}^L) = \epsilon$.
\item \label{it:6}
For all $m \in \{0, 1, \ldots, N_n-1\}$,
either $\bfd(t) \geq \epsilon$ for all $t \in [c_{n, m}^R, c_{n, m+1}^L]$,
or $\bfd(f_{n, m}^R) = \bfd(f_{n,m+1}^L) = \epsilon$.
\item \label{it:7}
For all $m \in \{0, 1, \ldots, N_n-1\}$,
\begin{equation}
\label{eq:mu-no-change}
\begin{aligned}
\sup_{t \in [e_{n, m}^L, c_{n, m}^R]}\mu(t) / \inf_{t \in [e_{n, m}^L, c_{n, m}^R]}\mu(t) &\leq 2, \\
\sup_{t \in [c_{n, m+1}^L, e_{n, m+1}^R]}\mu(t) / \inf_{t \in [c_{n, m+1}^L, e_{n, m+1}^R]}\mu(t) &\leq 2.
\end{aligned}
\end{equation}
\end{enumerate}
\end{lem}
\begin{proof}
For all $t_0 \in [a_n, b_n]$ such that $U(t_0) < \infty$, let $J(t_0) \subset [a_n, b_n]$ be the union of all the open (relatively in $[a_n, b_n]$) intervals containing $t_0$ on which $U$ is finite. Equivalently, we have one of the following three cases:
\begin{itemize}
\item $J(t_0) = (\wt a_n, \wt b_n)$, $t_0 \in (\wt a_n, \wt b_n)$, $\bfd(\wt a_n) = \bfd(\wt b_n) = \eta_0$ and $\bfd(t) < \eta_0$ for all $t \in (\wt a_n, \wt b_n)$,
\item $J(t_0) = [a_n, \wt b_n)$, $t_0 \in [a_n, \wt b_n)$,
$\bfd(\wt b_n) = \eta_0$ and $\bfd(t) < \eta_0$ for all $t \in [a_n, \wt b_n)$,
\item $J(t_0) = (\wt a_n, b_n]$, $t_0 \in (\wt a_n, b_n]$,
$\bfd(\wt a_n) = \eta_0$ and $\bfd(t) < \eta_0$ for all $t \in (\wt a_n, b_n]$.
\end{itemize}
Note that $\theta_n \gg \epsilon_n$ implies $\wt a_n > a_n$ and $\wt b_n < b_n$.
Clearly, any two such intervals are either equal or disjoint.

Consider the set
\begin{equation}
A := \{t \in [a_n, b_n]: \bfd(t) \leq \theta_n\}.
\end{equation}
Since $A$ is a compact set, there exists a finite sequence
\begin{equation}
a_n \leq s_{n, 0} < s_{n, 1} < \ldots < s_{n, N_n} \leq b_n
\end{equation}
such that
\begin{equation}
\label{eq:cover-of-A}
s_{n, m} \in A, \qquad A \subset \bigcup_{m=0}^{N_n} J(s_{n, m}).
\end{equation}
Without loss of generality, we can assume $J(s_{n, m}) \cap J(s_{n, m'}) = \emptyset$ whenever $m \neq m'$ (it suffices to remove certain elements from the sequence).

Observe, using \eqref{eq:U-bound-by-d},
that $U(s_{n, m}) \leq C\bfd(s_{n, m})^2 = o_n(1)$,
whereas $U(\wt a_n) \geq C_1^{-1}\bfd(\wt a_n)^2 \geq C_1^{-1}\eta_0^2$
and similarly $U(\wt b_n) \geq C_1^{-1}\eta_0^2$,
which for $n$ large enough implies that $U$, restricted to $J(s_{n, m})$, attains its global minimum.
Let $t_{n, m} \in J(s_{n, m})$ be one of these global minima, in particular we have $J(t_{n, m}) = J(s_{n, m})$ and one of the following three cases:
\begin{itemize}
\item $t_{n, m} \in (a_n, b_n)$ is a local minimum of $U$,
\item $t_{n, m} = a_n$ is a local minimum from the right of $U$,
\item $t_{n, m} = b_n$ is a local minimum from the left of $U$.
\end{itemize}
Note also that, again by \eqref{eq:U-bound-by-d},
\begin{equation}
\label{eq:d-at-tmn}
\bfd(t_{n, m}) \leq \sqrt{C_1U(t_{n, m})} \leq \sqrt{C_1U(s_{n, m})}
\leq C_1\bfd(s_{n, m}) \leq C_1\theta_n,
\end{equation}
where the last inequlity follows since $s_{n, m} \in A$.

Let $m \in \{0, 1, \ldots, N_n-1\}$. Since $J(t_{n, m}) \cap J(t_{n, m+1}) = \emptyset$, there exists $t \in (t_{n, m}, t_{n, m+1})$ such that
$U(t) = \infty$. Let $d_{n, m}^R$ be the smallest such $t$, and $d_{n, m+1}^L$
the largest one. Let $c_{n, m}^R$ be the smallest number such that
$\bfd(t) \geq \epsilon$ for all $t \in (c_{n, m}^R, d_{n, m}^R)$.
Similarly, let $c_{n, m+1}^L$ be the biggest number such that
$\bfd(t)\geq \epsilon$ for all $t \in (d_{n, m+1}^L, c_{n, m+1}^L)$.
Next, let $e_{n, m}^R$ be the smallest number such that
$\bfd(t) \geq 2C_1\theta_n$ for all $t \in (e_{n, m}^R, c_{n, m}^R)$.
If we take $\theta_n < \frac{\epsilon}{2C_1}$, then we have $e_{n, m}^R < c_{n, m}^R$.
It follows from \eqref{eq:d-at-tmn} that $e_{n, m}^R > t_{n, m}$.
Similarly, let $e_{n, m+1}^L$ be the biggest number such that
$\bfd(t)\geq 2C_1\theta_n$ for all $t \in (c_{n, m+1}^L, e_{n, m+1}^L)$
(again, it follows that $e_{n, m+1}^L < t_{n, m+1}$).
Finally, if $\bfd(t) \geq \epsilon$ for all $t \in (d_{n, m}^R, d_{n, m+1}^L)$,
we set $f_{n, m}^R$ and $f_{n, m+1}^L$ arbitrarily,
for example $f_{n, m}^R := d_{n, m}^R$ and $f_{n, m+1}^L := d_{n, m+1}^L$.
If, on the contrary, there exists $t \in (d_{n, m}^R, d_{n, m+1}^L)$
such that $\bfd(t) < \epsilon$, we let $f_{n, m}^R$
be the biggest number such that $\bfd(t)\geq \epsilon$ for all $t \in (d_{n, m}^R, f_{n, m}^R)$, and $f_{n, m+1}^L$ be the smallest number such that
$\bfd(t) \geq \epsilon$ for all $t \in (f_{n, m+1}^L, d_{n, m+1}^L)$.

We check all the desired properties. For all $n \in \{0, 1, \ldots, N_n\}$, we have $e_{n, m}^L \leq t_{n, m} \leq e_{n, m}^R$.
Moreover, if $t_{n, m} = e_{n, m}^L$ (which can only happen for $m = 0$), then $t_{n, m}$ is a local minimum from the right of $U$,
and if $t_{n, m} = e_{n, m}^R$ (which can only happen for $m = N_n$), then $t_{n, m}$ is a local minimum from the left of $U$.
Since $\bfd(e_{n, m}^L) \leq 2C_1\theta_n$ and $\bfd(e_{n, m}^R) \leq 2C_1\theta_n$, the property \ref{it:1} follows from \eqref{eq:int-d-bd}.
The properties \ref{it:3}, \ref{it:4}, \ref{it:5} and \ref{it:6} follow directly from the construction.
The property \ref{it:2} is now equivalent to the following statement: if $\bfd(t_0) < \theta_n$, then
there exists $m \in \{0, 1, \ldots, N_n\}$ such that $t_0 \in [e_{n, m}^L, e_{n, m}^R]$.
But \eqref{eq:cover-of-A} implies that $t_0 \in J(s_{n, m}) = J(t_{n, m})$ for some $m$
and, by construction, $\bfd(t) > \theta_n$ for all $t \in J(t_{n, m}) \setminus [e_{n, m}^L, e_{n, m}^R]$,
so we obtain $t \in [e_{n, m}^L, e_{n, m}^R]$. Finally, using again Lemma~\ref{lem:ejection},
but on the time intervals $[t_{n, m}, c_{n, m}^R]$ and $[c_{n, m+1}^L, t_{n, m+1}]$, we deduce the property \ref{it:7}
from \eqref{eq:lambdaK-bd}.

\end{proof}

\subsection{End of the proof: virial inequality with a cut-off}
In this section, we conclude the proof, by integrating the virial identity on the time interval $[a_n, b_n]$.
The radius where the cut-off is imposed has to be carefully chosen, which is the object of the next lemma.
\begin{lem}
\label{lem:rho}
There exist $\theta_0 > 0$ and a locally Lipschitz function $\rho : \cup_{n=1}^\infty [a_n, b_n] \to (0, \infty)$
having the following properties:
\begin{enumerate}
\item $\max(\rho(a_n)\|\partial_t u(a_n)\|_{L^2}, \rho(b_n)\|\partial_t u(b_n)\|_{L^2}) \ll \max(\mu(a_n), \mu(b_n))$ as $n \to \infty$,
\item $\lim_{n \to \infty}\inf_{t\in[a_n, b_n]}\big(\rho(t) / \mu(t)\big) = \infty$ and $\lim_{n\to\infty}\sup_{t\in[a_n, b_n]}\big(\rho(t) / \mu_{K+1}(t)\big) = 0$,
\item if $\bfd(t_0) \leq\frac 12\theta_0$, then $|\rho'(t)| \leq 1$ for almost all $t$ in a neighborhood of $t_0$,
\item $\lim_{n\to\infty}\sup_{t\in[a_n, b_n]}|\Omega_{\rho(t)}(\bs u(t))| = 0$.
\end{enumerate}
\end{lem}
\begin{proof}
We will define two functions $\rho^{(a)}, \rho^{(b)}$, and then set $\rho := \min(\rho^{(a)}, \rho^{(b)}, \nu)$.
First, we let
\begin{equation}
\rho^{(a)}(a_n) := \min(R_n\mu(a_n), \nu(a_n)),
\end{equation}
where $1 \ll R_n \ll \|\partial_t u(a_n)\|_{L^2}^{-1}$.
Consider an auxiliary sequence
\begin{equation}
\label{eq:deltan-virial-def}
\delta_n := \sup_{t \in [a_n, b_n]}
\|\bs u(t)\|_{\cE(\min(\rho^{(a)}(a_n) + t - a_n, \nu(t)); 2\nu(t))}.
\end{equation}
We have $\lim_{n \to \infty} \delta_n = 0$. Indeed, we see from the finite speed of propagation that
\begin{equation}
\limsup_{n\to\infty}E(\bs u(t); \rho^{(a)}(a_n) + t - a_n, \infty) \leq E(\bs u^*) + (N-K)E(\bs Q).
\end{equation}
This and Lemma~\ref{lem:ext-sign} yield
\begin{equation}
\lim_{n\to\infty}E(\bs u(t); \min(\rho^{(a)}(a_n) + t - a_n, \nu(t)); 2\nu(t)) = 0,
\end{equation}
thus Lemma~\ref{lem:pi} implies $\delta_n \to 0$.

Let $\theta_0 > 0$ be given by Lemma~\ref{lem:cedf}, and divide $[a_n, b_n]$ into subintervals applying this lemma
for the constant sequence $\theta_n = \theta_0$.
We let $\rho^{(a)}$ be the piecewise affine function such that
\begin{equation}
\dd t\rho^{(a)}(t) := 1\ \text{if }t \in [e_{n, m}^L, e_{n, m}^R],\qquad \dd t\rho^{(a)}(t) := \delta_n^{-\frac 12}\text{ otherwise.}
\end{equation}
We check that $\lim_{n\to\infty}\inf_{t\in[a_n, b_n]}\big(\rho^{(a)}(t) / \mu(t)\big) = \infty$.
First, suppose that $t \in [e_{n, m}^R, e_{n, m+1}^L]$ and $t - e_{n, m}^R \gtrsim \mu(e_{n, m}^R)$.
Then $\mu(t) \leq \mu(e_{n, m}^R) + (t - e_{n, m}^R) \lesssim t - e_{n, m}^R$
and $\rho^{(a)}(t) \geq \delta_n^{-\frac 12}(t - e_{n, m}^R)$, so $\rho^{(a)}(t) \gg \mu(t)$.

By Lemma~\ref{lem:cd-length}, $e_{n, m+1}^L - e_{n, m}^R \geq C_{\bs u} \mu(e_{n, m}^R)$,
so in particular we obtain $\rho^{(a)}(e_{n, m+1}^L)\gg \mu(e_{n, m+1}^L)$ for all $m \in \{0, 1, \ldots, N_n - 1\}$.
Note that we also have $\rho^{(a)}(e_{n, 0}^L) = \rho^{(a)}(a_n) \gg \mu(a_n) = \mu(e_{n, 0}^L)$,
by the choice of $\rho^{(a)}(a_n)$.
Since, by the property (7), $\mu$ changes at most by a factor $2$ on $[e_{n, m}^L, e_{n, m}^R]$ and $\rho^{(a)}$ is increasing,
we have $\rho^{(a)}(e_{n, m}^R) \gg \mu(e_{n, m}^R)$.

Finally, if $t - e_{n, m}^R \leq \mu(e_{n, m}^R)$, then $\mu(t) \leq 2\mu(e_{n, m}^R)$,
which again implies $\rho^{(a)}(t) \gg \mu(t)$.

The function $\rho^{(b)}$ is defined similarly, but integrating from $b_n$ backwards.
Properties (1), (2), (3) are clear.
By the expression for $\Omega_{\rho(t)}(\bs u(t))$, see Lemma~\ref{lem:vir},
we have
\begin{equation}
|\Omega_{\rho(t)}(\bs u(t))| \lesssim (1+|\rho'(t)|)\|\bs u(t)\|_{\cE(\rho(t), 2\rho(t))}^2 \lesssim \sqrt{\delta_n} \to 0,
\end{equation}
which proves the property (4).

\end{proof}
We need one more elementary result.
\begin{lem}
\label{lem:subdivision}
If $\mu: [a, b] \to (0, \infty)$ is a $1$-Lipschitz function
and $b - a \geq \frac 14 \mu(a)$, then there exists a sequence
$a = a_0 < a_1 < \ldots < a_l < a_{l+1} = b$ such that
\begin{equation}
\label{eq:lip-partition}
\frac 14 \mu(a_i) \leq a_{i+1} - a_i \leq \frac 34 \mu(a_i),\qquad\text{for all }i \in \{1, \ldots, l\}.
\end{equation}
\end{lem}
\begin{proof}
We define inductively $a_{i+1} := a_i + \frac 14\mu(a_i)$, as long as $b - a_i > \frac 34 \mu(a_i)$. We need to prove that $b - a_i > \frac 34 \mu(a_i)$
implies $b - a_{i+1} > \frac 14\mu(a_{i+1})$.

Since $\mu$ is $1$-Lipschitz, $\mu(a_{i+1}) = \mu(a_i + \mu(a_i)/4)
\leq \mu(a_i) + \mu(a_i)/4 = \frac 54 \mu(a_i)$, thus
\begin{equation}
b - a_{i+1} = b - a_i - \frac 14 \mu(a_i) > \frac 34 \mu(a_i) - \frac 14\mu(a_i)
> \frac{5}{16}\mu(a_i) \geq \frac 14 \mu(a_{i+1}).
\end{equation}
\end{proof}
\begin{rem}
\label{rem:subdivision}
Note that \eqref{eq:lip-partition} and the fact that $\mu$ is $1$-Lipschitz
imply $\inf_{t \in [a_i, a_{i+1}]}\mu(t) \geq \frac 14\mu(a_i)$
and $\sup_{t \in [a_i, a_{i+1}]}\mu(t) \leq \frac 74\mu(a_i)$, thus
\begin{equation}
\frac 17 \sup_{t \in [a_i, a_{i+1}]}\mu(t) \leq a_{i+1} - a_i \leq 3\inf_{t \in [a_i, a_{i+1}]}\mu(t),
\end{equation}
in other words the length of each subinterval is comparable with
both the smallest and the largest value of $\mu$ on this subinterval.
\end{rem}
\begin{lem}
\label{lem:virial-decrease}
Let $\rho$ be the function given by Lemma~\ref{lem:rho} and set
\begin{equation}
\label{eq:fv-def}
\fv(t) := \int_0^\infty\partial_t u(t) r\partial_r u(t) \chi_{\rho(t)}\,r\vd r.
\end{equation}
\begin{enumerate}[1.]
\item There exists a sequence $\theta_n \to 0$ such that the following is true. If $[\wt a_n, \wt b_n] \subset [a_n, b_n]$ is such that
\begin{equation}
\wt b_n - \wt a_n \geq \frac 14 \mu(\wt a_n)\quad\text{and}\quad
\bfd(t) \geq \theta_n\text{ for all }t \in [\wt a_n, \wt b_n],
\end{equation}
then
\begin{equation}
\label{eq:intermediate}
\fv(\wt b_n) < \fv(\wt a_n).
\end{equation}
\item 
For any $c, \theta > 0$ there exists $\delta > 0$ such that if $n$ is large enough, $[\wt a_n, \wt b_n] \subset [a_n, b_n]$,
\begin{equation}
c \mu(\wt a_n) \leq \wt b_n - \wt a_n \quad\text{and}\quad
\bfd(t) \geq \theta\text{ for all }t \in [\wt a_n, \wt b_n],
\end{equation}
then
\begin{equation}
\label{eq:intermediate-2}
\fv(\wt b_n) - \fv(\wt a_n) \leq  -\delta \sup_{t\in[\wt a_n, \wt b_n]}\mu(t).
\end{equation}
\end{enumerate}
\end{lem}
\begin{proof}
By the virial identity, we obtain
\begin{equation}
\label{eq:virial-fin}
\fv'(t)  =
- \int_0^\infty (\partial_t u(t))^2 \chi_{\rho(t)}\,r\vd r + o_n(1).
\end{equation}
We argue by contradiction. If the claim is false, then there exists $\theta >0$
and an infinite sequence $[\wt a_n, \wt b_n] \subset [a_n, b_n]$
(as usual, we pass to a subsequence in $n$ without changing the notation)
such that
\begin{equation}
\wt b_n - \wt a_n \geq \frac 14 \mu(\wt a_n)\quad\text{and}\quad
\bfd(t) \geq \theta\text{ for all }t \in [\wt a_n, \wt b_n],
\end{equation}
and
\begin{equation}
\fv(\wt b_n) - \fv(\wt a_n) \geq 0.
\end{equation}
By Lemma~\ref{lem:subdivision}, there exists a subinterval of $[\wt a_n, \wt b_n]$, which we still denote $[\wt a_n, \wt b_n]$, such that
\begin{equation}
\frac 14 \mu(\wt a_n) \leq \wt b_n - \wt a_n \leq \frac 34 \mu(\wt a_n)\quad\text{and}\quad
\fv(\wt b_n) - \fv(\wt a_n) \geq 0.
\end{equation}
Let $\wt\rho_n := \inf_{t \in [\wt a_n, \wt b_n]}\rho(t)$. From \eqref{eq:virial-fin}, we have
\begin{equation}
\lim_{n\to\infty}\frac{1}{\wt b_n - \wt a_n}\int_{\wt a_n}^{\wt b_n}\int_0^{\frac 12 \wt \rho_n}(\partial_t u(t))^2\,r\vd r = 0.
\end{equation}
By Lemma~\ref{lem:rho}, $\inf_{t\in[\wt a_n, \wt b_n]}\mu_{K+1}(t) \gg \wt\rho_n \gg \inf_{t\in[\wt a_n, \wt b_n]}\mu(t)
\simeq \sup_{t\in[\wt a_n, \wt b_n]}\mu(t)$, so Lemma~\ref{lem:compact} yields sequences $t_n \in [\wt a_n, \wt b_n]$ and $1 \ll r_n \ll \mu_{K+1}(t_n)/\mu(t_n)$ such that
\begin{equation}
\lim_{n\to\infty}\bs\de_{r_n\mu(t_n)}(\bs u(t_n)) = 0,
\end{equation}
which is impossible by Lemma~\ref{lem:mu-prop} (iii). The first part of the lemma is proved.

In the second part, we can assume without loss of generality $\wt b_n - \wt a_n \leq \frac 34 \mu(\wt a_n)$.
Indeed, in the opposite case, we apply Lemma~\ref{lem:subdivision}
and keep only one of the subintervals where $\mu$ attains its supremum, and on the remaining subintervals we use
\eqref{eq:intermediate}.

After this preliminary reduction, we argue again by contradiction.
If the claim is false, then there exist $c, \theta >0$, a sequence $\delta_n \to 0$
and a sequence $[\wt a_n, \wt b_n] \subset [a_n, b_n]$ (after extraction of a subsequence) such that
\begin{equation}
c\mu(\wt a_n) \leq \wt b_n - \wt a_n \leq \frac 34 \mu(\wt a_n)\quad\text{and}\quad
\bfd(t) \geq \theta\text{ for all }t \in [\wt a_n, \wt b_n],
\end{equation}
and
\begin{equation}
\fv(\wt b_n) - \fv(\wt a_n) \geq -\delta_n\mu(\wt a_n)
\end{equation}
(we use the fact that $\mu(\wt a_n)$ is comparable to $\sup_{t\in[\wt a_n, \wt b_n]}\mu(t)$, see Remark~\ref{rem:subdivision}).

Let $\wt\rho_n := \inf_{t \in [\wt a_n, \wt b_n]}\rho(t)$. From \eqref{eq:virial-fin}, we have
\begin{equation}
\lim_{n\to\infty}\frac{1}{\wt b_n - \wt a_n}\int_{\wt a_n}^{\wt b_n}\int_0^{\frac 12 \wt \rho_n}(\partial_t u(t))^2\,r\vd r = 0.
\end{equation}
We now conclude as in the first part.
\end{proof}

\begin{proof}[Proof of Theorem~\ref{thm:main}]
Let $\theta_n$ be the sequence given by Lemma~\ref{lem:virial-decrease}, part 1.
We partition $[a_n, b_n]$ applying Lemma~\ref{lem:cedf} for this sequence $\theta_n$.
Note that this partition is different than the one used in the proof of Lemma~\ref{lem:rho}.
We claim that
for all $m \in \{0, 1, \ldots, N_n-1\}$
\begin{align}
\label{eq:intermediate-fin-1}
\fv(c_{n, m}^R) - \fv(e_{n, m}^R) &\leq o_n(1)\mu(c_{n, m}^R), \\
\label{eq:intermediate-fin-2}
\fv(f_{n, m+1}^L) - \fv(f_{n, m}^R) &\leq o_n(1)\mu(f_{n, m}^R), \\
\label{eq:intermediate-fin-3}
\fv(e_{n, m+1}^L) - \fv(c_{n, m+1}^L) &\leq o_n(1)\mu(c_{n, m+1}^L).
\end{align}
Here, $o_n(1)$ denotes a sequence of positive numbers converging to $0$ when $n\to\infty$.
In order to prove the first inequality, we observe that if $c_{n, m}^R - e_{n, m}^R \geq \frac 14 \mu(e_{n, m}^R)$,
then \eqref{eq:intermediate} applies and yields $\fv(c_{n, m}^R) - \fv(e_{n, m}^R)  < 0$.
We can thus assume $c_{n, m}^R - e_{n, m}^R \leq \frac 14 \mu(e_{n, m}^R) \leq \frac 12 \mu(c_{n, m}^R)$,
where the last inequality follows from Lemma~\ref{lem:cedf}, property \ref{it:7}.
But then \eqref{eq:virial-fin} again implies the required bound.
The proofs of the second and third bound are analogous.

We now analyse the compactness intervals $[c_{n, j}^R, f_{n, j}^R]$ and $[f_{n, j+1}^L, c_{n, j+1}^L]$.
We claim that there exists $\delta > 0$ such that for all $n$ large enough and $m \in \{0, 1, \ldots, N_n\}$
\begin{equation}
\label{eq:compactness-fin}
\fv(c_{n, m+1}^L) - \fv(c_{n, m}^R) \leq -\delta \max(\mu(c_{n, m}^R), \mu(c_{n, m+1}^L)).
\end{equation}
We consider separately the two cases mentioned in Lemma~\ref{lem:cedf}, property \ref{it:6}.
If $\bfd(t) \geq \epsilon$ for all $t \in [c_{n, m}^R, c_{n, m+1}^L]$, then Lemma~\ref{lem:cd-length}
yields $c_{n, m+1}^L - c_{n,m}^R \geq C_{\bs u}^{-1}\mu(c_{n,m}^R)$, so we can apply \eqref{eq:intermediate-2}, which proves \eqref{eq:compactness-fin}.
If $\bfd(f_{n,m}^R) = \epsilon$, then we apply the same argument on the time interval $[c_{n,m}^R, f_{n,m}^R]$ and obtain
\begin{equation}
\label{eq:compactness-fin-4}
\fv(f_{n, m}^R) - \fv(c_{n, m}^R) \leq -\delta \max(\mu(c_{n, m}^R), \mu(f_{n, m}^R)),
\end{equation}
and similarly
\begin{equation}
\label{eq:compactness-fin-5}
\fv(c_{n, m+1}^L) - \fv(f_{n, m+1}^L) \leq -\delta \max(\mu(c_{n, m+1}^L), \mu(f_{n,m+1}^L)).
\end{equation}
The bound \eqref{eq:intermediate-fin-2} yields \eqref{eq:compactness-fin}.

Finally, on the intervals $[e_{n, m}^L, e_{n,m}^R]$, for $n$ large enough Lemma~\ref{lem:rho} yields $|\rho'(t)| \leq 1$
for almost all $t$, and Lemma~\ref{lem:virial-error} implies $|\fv'(t)| \lesssim \bfd(t)$.
By Lemma~\ref{lem:cedf}, properties \ref{it:1} and \ref{it:7}, we obtain
\begin{equation}
\label{eq:modulation-fin}
\begin{aligned}
\fv(e_{n,m}^R) - \fv(e_{n,m}^L) &\leq o_n(1)\mu(c_{n,m}^R),\qquad\text{for all }m \in \{0, 1, \ldots, N_n-1\}, \\
\fv(e_{n,m}^R) - \fv(e_{n,m}^L) &\leq o_n(1)\mu(c_{n,m}^L),\qquad\text{for all }m \in \{1, \ldots, N_n-1, N_n\}.
\end{aligned}
\end{equation}

Taking the sum in $m$ of \eqref{eq:intermediate-fin-1}, \eqref{eq:intermediate-fin-3}, \eqref{eq:compactness-fin} and
\eqref{eq:modulation-fin}, we deduce that there exists $\delta > 0$
and $n$ arbitrarily large such that
\begin{equation}
\fv(b_n) - \fv(a_n) \leq -\delta\max(\mu(c_{n,0}^R), \mu(c_{n,N_n}^L)).
\end{equation}
But $\mu(a_n) \simeq \mu(c_{n,0}^R)$ and $\mu(b_n) \simeq \mu(c_{n,N_n}^L)$, hence
\begin{equation}
\fv(b_n) - \fv(a_n) \leq -\wt\delta\max(\mu(a_n), \mu(b_n)).
\end{equation}
Lemma~\ref{lem:rho} (1) and \eqref{eq:fv-def} yield
\begin{equation}
|\fv(a_n)| \ll \mu(a_n), \qquad |\fv(b_n)| \ll \mu(b_n),
\end{equation}
a contradiction which finishes the proof.
\end{proof}

\subsection{Absence of elastic collisions}
\label{ssec:inelastic}
This section is devoted to proving Proposition~\ref{prop:inelastic}
Our proof closely follows Step 3 in our proof of \cite[Theorem 1.6]{JL1}.
\begin{proof}[Proof of Proposition~\ref{prop:inelastic}]
Suppose that a solution of \eqref{eq:wmk}, $\bs u$,  defined on its maximal time of existence
$t \in (T_-, T_+)$, is a pure multi-bubble in both time directions in the sense of Definition~\ref{def:pure}, in other words
\begin{equation}
\lim_{t \to T_+}\bfd(t) = 0, \qquad\text{and}\qquad \lim_{t \to T_-}\bfd(t) = 0, 
\end{equation}
and the radiation $\bs u^* = \bs u^*_L$ or $\bs u^* = \bs u^*_0$ in both time directions satisfies $\bs u^* \equiv 0$. 
In this proof, all the $N$ bubbles can be thought of as ``interior'' bubbles
thus, whenever we invoke the results from the preceding sections, it should always
be understood that $K = N$.
Applying Lemma~\ref{lem:mod-static} with $\theta = 0$ and $M = N$,
we obtain from \eqref{eq:g-bound-A} and \eqref{eq:g-bound-0} that
\begin{equation}
\label{eq:pure-g-bound}
\bfd(t) \leq C \max_{j \in \cA}\Big(\frac{\lambda_j}{\lambda_{j+1}}\Big)^k.
\end{equation}
Let $\eta > 0$ be a small number to be chosen later and $t_+$ be such that
$\bfd(t) \leq \eta$ for all $t \geq t_+$.
If $\eta$ is sufficiently small, then the modulation parameters are well-defined for $t \geq t_+$, so we can set
\begin{equation}
U(t) := \max_{i \in \cA}\big(2^{-i}\xi_i(t) / \lambda_{i+1}(t)\big)^k,\qquad \text{for all }t \geq t_+,
\end{equation}
cf. Definition~\ref{def:w-int}. Since $U$ is a positive continuous function and
$\lim_{t \to T_+}U(t) = 0$, there exists an increasing sequence $t_n \to T_+$
such that $t_n$ is a local minimum from the left of $U$. Thus, Lemma~\ref{lem:ejection} yields
\begin{equation}
\int_{t_+}^{t_n}\bfd(t)\ud t \leq C_0 \bfd(t_+)^\frac 2k \lambda_N(t_+),
\end{equation}
and passing to the limit $n \to +\infty$ we get
\begin{equation}
\label{eq:integral-of-d}
\int_{t_+}^{T_+}\bfd(t)\ud t \leq C_0 \bfd(t_+)^\frac 2k \lambda_N(t_+).
\end{equation}
By inspecting the proof of Lemma~\ref{lem:mod-1}, one finds that in the present case it holds
with $\zeta_n = 0$, in particular we have $|\lambda_N'(t)| \lesssim \bfd(t)$.
This bound, together with \eqref{eq:integral-of-d}, implies that $\lim_{t \to T_+}\lambda_N(t)$
is a finite positive number, thus $T_+ = +\infty$.

Analogously, $T_- = -\infty$ and $\lim_{t \to -\infty}\lambda_N(t) \in (0, +\infty)$ exists.

The remaining part of the argument is exactly the same as in \cite{JL1}, but we reproduce it
here for the reader's convenience.

Let $\delta > 0$ be arbitrary. Inspecting the proof of Lemma~\ref{lem:virial-error},
we see that in the present case it holds with $\delta_n = 0$, thus for any $R > 0$ we have
$|\Omega_R(\bs u(t))| \leq C_0\bfd(t)$.
From this bound and the estimates above, we obtain existence of $T_1, T_2 \in \bR$ such that
\begin{align}
\int_{-\infty}^{T_1}|\Omega_R(\bs u(t))|\ud t &\leq \frac 13 \delta, \\
\int_{T_2}^{+\infty}|\Omega_R(\bs u(t))|\ud t &\leq \frac 13 \delta
\end{align}
for any $R > 0$. On the other hand, because of the bound $|\Om_R(\bs u(t))| \le C_0 E(\bs u(t); R, 2R)$ and since $[T_1, T_2]$ is a finite time interval, for all $R$ sufficiently large we have
\begin{equation}
\int_{T_1}^{T_2}|\Omega_R(\bs u(t))|\ud t \leq \frac 13 \delta,
\end{equation}
in other words
\begin{equation}
\int_{\bR}|\Omega_R(\bs u(t))|\ud t \leq \delta.
\end{equation}
Integrating the virial identity from Lemma~\ref{lem:vir} with $\rho(t) = R$ over the real line,
we obtain
\begin{equation}
\int_{-\infty}^{+\infty}\int_0^\infty (\partial_t u(t, r)\chi_R(r))^2\,r\vd r\ud t \leq \delta.
\end{equation}
By letting $R \to +\infty$, we get
\begin{equation}
\int_{-\infty}^{+\infty}\int_0^\infty (\partial_t u(t, r))^2\,r\vd r\ud t \leq \delta,
\end{equation}
which implies the $\bs u$ is stationary since $\delta$ is arbitrary.
\end{proof}
\appendix 

\section{Modifications to the argument in the case $k=1$} 
In this section we outline the changes to the arguments in Section~\ref{sec:decomposition} and  Section~\ref{sec:conclusion} needed to prove Theorem~\ref{thm:main} for the equivariance class $k =1$. 

\subsection{Modulation and refined modulation}
The set-up in Sections~\ref{ssec:proximity} holds without modification for $k=1$. To be precise the number $K \ge1$ is defined as in Lemma~\ref{lem:K-exist}, the collision intervals $[a_n, b_n] \in \calC_K(\eta, \eps_n)$ are as in Definition~\ref{def:K-choice}, and  the sequences of signs $\vec \s_n \in \{-1, 1\}^{N-K}$,  scales $\vec \mu(t) \in (0, \infty)^{N-K}$, and integers $m_n \in \Z$ associated to  the exterior bubbles, and the sequence $\nu_n \to 0$ and  the function $\nu(t) = \nu_n \mu_{K+1}(t)$ are as in Lemma~\ref{lem:ext-sign}. 

 Lemma~\ref{lem:mod-1} also holds without modification. Let $J \subset [a_n, b_n]$ be any time interval on which $\bfd(t) \le \eta_0$, where $\eta_0$ is as in  Lemma~\ref{lem:mod-1}. Let $\vec \iota \in \{ -1, 1\}^K, \vec \lam(t) \in (0, \infty)^K$, and $\bs g(t) \in \E$ be as in the statement of Lemma~\ref{lem:mod-1}. Let $L>0$ be a parameter to be fixed below and define for each $j \in \{1, \dots, K-1\}$, 
\EQ{\label{eq:xi-def-k1} 
\xi_j(t) &:=  \lam_j(t)  - \frac{ \iota_j}{2\log(\frac{\lam_{j+1}(t)}{\lam_j(t)}) }\La \chi_{L\sqrt{\lam_j(t) \lam_{j+1}(t)}} \Lam Q_{\U{\lam_j(t)}} \mid g(t) + \sum_{i<j} \iota_i( Q_{\lam_i(t)} - \pi) \Ra,  
}
and, 
\begin{equation} \label{eq:beta-def-k1} 
\beta_j(t) :=- \iota_j\La \chi_{L\sqrt{\xi_j(t) \lam_{j+1}(t)}} \Lam Q_{\U{\lam_j(t)}} \mid \dot g(t)\Ra  -   \ang{ \uln A( \lam_j(t)) g(t) \mid \dot g(t)}.
\end{equation}

\begin{prop}[Refined modulation, $k=1$]
\label{prop:modul-k1} 
 Let $c_0  \in (0, 1)$ and $c_1>0$. There exists  constants $L_0 = L_0(c_0, c_1)>0$, $\eta_0 = \eta_0(c_0, c_1)$, as well as  $c= c(c_0, c_1)$ and $R = R(c_0, c_1)>1$ as in Lemma~\ref{lem:q},  a constant $C_0>0$, and a 
decreasing sequence $\epsilon_n \to 0$ so that the following is true. 

Suppose $L > L_0$ and  $J \subset [a_n, b_n]$ is an open time interval with $\epsilon_n \leq \bfd(t) \le \eta_0$
for all $t \in J$, where $\calA :=  \{ j \in \{1, \dots, K-1\} \mid \iota_{j}  \neq \iota_{j+1} \}$.  
Then, for all $t \in J$, 
\EQ{\label{eq:g-bound-k1} 
\| \bs g(t) \|_{\E} + \sum_{i \not \in \calA}  (\lambda_i(t) / \lambda_{i+1}(t))^\frac 12 \le  \max_{ i  \in \calA}  (\lambda_i(t) / \lambda_{i+1}(t))^\frac 12, \\
}
and, 
\EQ{ \label{eq:d-bound-k1} 
\frac{1}{C_0} \bfd(t) \le \max_{i \in \calA} (\lambda_i(t) / \lambda_{i+1}(t))^\frac 12  \le C_0 \bfd(t) ,
}
\begin{equation}\label{eq:xi_j-lambda_j-k1}
\Big|\frac{\xi_j(t)}{\lambda_{j+1}(t)} - \frac{\lambda_j(t)}{\lambda_{j+1}(t)}\Big| \leq c_0\bfd(t)^2.
\end{equation}
Moreover, let $j \in \calA$ be such that for all $t \in J$
\EQ{ \label{eq:max-ratio-ass}
c_1 \bfd(t) \le \Big(\frac{\lam_j(t)}{\lam_{j+1}(t)}\Big)^{\frac{1}{2}}.
}
Then for all $t \in J$, 
\EQ{ \label{eq:xi_j'-k1} 
\abs{ \xi'_j(t) }\Big(\log\Big(\frac{\lam_{j+1}(t)}{\lam_j(t)}\Big)\Big)^{\frac{1}{2}} \le C_0 \max_{i  \in \calA}\sqrt{ \frac{\lam_{i}(t)}{\lam_{i+1}(t)} } , 
}

\begin{equation}\label{eq:xi_j'-beta_j-k1} 
\Big|\xi_j'(t) 2\log(\frac{\lam_{j+1}(t)}{\lam_j(t)})- \beta_j(t)\Big| \leq C_0\max_{i  \in \calA}\sqrt{ \frac{\lam_{i}(t)}{\lam_{i+1}(t)} } 
\end{equation}
and,  
\EQ{ \label{eq:beta_j'-k1} 
 \beta_{j}'(t) &\ge  \Big({-}\iota_j \iota_{j+1}8 -   c_0\Big) \frac{1}{\lam_{j+1}(t)} + \Big(\iota_j \iota_{j-1}8 -  c_0\Big) \frac{\lam_{j-1}(t)}{\lam_{j}(t)^2}  \\
&\quad  -  \frac{c_0}{\lam_j(t)} \max_{i \in \calA}\frac{\lam_{i}(t)}{\lam_{i+1}(t)} .
}
where, by convention, $\lambda_0(t) = 0, \lam_{K+1}(t) = \infty$ for all $t \in J$.  
\end{prop}


\begin{proof} 
The estimates~\eqref{eq:g-bound-k1} and~\eqref{eq:d-bound-k1} follow as in the proofs of the corresponding estimates in Lemma~\ref{cor:modul}. We next prove~\eqref{eq:xi_j-lambda_j-k1}. From the definition of $\xi_j(t)$, 
\EQ{
\bigg|\frac{\xi_j}{\lambda_{j+1}} - \frac{\lam_j}{\lam_{j+1}}\bigg| &\lesssim \Big|\frac{ 1}{\log(\frac{\lam_{j+1}}{\lam_j}) } \lam_{j+1}^{-1} \La \chi_{L\sqrt{\lam_j \lam_{j+1}}} \Lam Q_{\U{\lam_j}} \mid g \Ra \Big| \\
&\quad + \Big| \frac{ 1}{\log(\frac{\lam_{j+1}}{\lam_j}) }\lam_{j+1}^{-1}\La \chi_{L\sqrt{\lam_j \lam_{j+1}}} \Lam Q_{\U{\lam_j}} \mid \sum_{i<j} ( Q_{\lam_i} - \pi) \Ra \Big|
}
For the first term on the right we have, 
\EQ{
\Big|\frac{ 1}{\log(\frac{\lam_{j+1}}{\lam_j}) } \lam_{j+1}^{-1} \La \chi_{L\sqrt{\lam_j \lam_{j+1}}} \Lam Q_{\U{\lam_j}} \mid g \Ra \Big| 
& \lesssim_L \frac{ 1}{\log(\frac{\lam_{j+1}}{\lam_j}) } \|g \|_{H} (\lam_{j}/ \lam_{j+1})^{\frac{1}{2}}  
}
 Next, for any $i <j$ we have, 
\EQ{
\lam_{j+1}^{-1}|\La \chi_{L\sqrt{\lam_j \lam_{j+1}}} \Lam Q_{\U{\lam_j}} \mid ( Q_{\lam_i} - \pi) \Ra | &\lesssim_L   \frac{\lam_j}{\lam_{j+1}}  \int_0^{L (\lam_{j+1}/ \lam_j)^{\frac{1}{2}}} \Lam Q(r) \abs{ Q_{\lam_i/ \lam_j}(r) - \pi } \, r \, \ud r\\
&\lesssim_L  \frac{\lam_j}{\lam_{j+1}}\frac{\lam_i}{\lam_j} \log( \lam_{j+1}/ \lam_j)
}
and hence, 
\EQ{
 \Big| \frac{ 1}{\log(\frac{\lam_{j+1}}{\lam_j}) }\lam_{j+1}^{-1}\La \chi_{L\sqrt{\lam_j \lam_{j+1}}} \Lam Q_{\U{\lam_j}} \mid \sum_{i<j} ( Q_{\lam_i} - \pi) \Ra \Big|  \lesssim_L \frac{\lam_j}{\lam_{j+1}} \sum_{i<j} \frac{\lam_{i}}{\lam_j}
}
and ~\eqref{eq:xi_j-lambda_j-k1} follows. 

Next using~\eqref{eq:g-bound-k1} and~\eqref{eq:lam'} for each $j$, we have
\EQ{ \label{eq:lam'-k1} 
\abs{ \lam_j'}  \lesssim \max_{ i  \in \calA}  (\lambda_i(t) / \lambda_{i+1}(t))^\frac 12. 
}
We show that in fact $\xi_j'$ satisfies the improved estimate~\eqref{eq:xi_j'-k1}. We compute, 
\EQ{ \label{eq:xi'-exp-k1} 
\xi_j' &=  \lam_j'  - \frac{ \iota_j}{2(\log(\frac{\lam_{j+1}}{\lam_j}))^2 }( \frac{\lam_{j}'}{\lam_j} - \frac{\lam_{j+1}'}{\lam_{j+1}})\La \chi_{L\sqrt{\lam_j \lam_{j+1}}} \Lam Q_{\U{\lam_j}} \mid g + \sum_{i<j} \iota_i( Q_{\lam_i} - \pi)\Ra\\
& \quad + \frac{ \iota_j}{4\log(\frac{\lam_{j+1}}{\lam_j}) }\big( \frac{\lam_j'}{\lam_j} + \frac{\lam_{j+1}'}{\lam_{j+1}}\big)\La \Lam \chi_{L\sqrt{\lam_j \lam_{j+1}}} \Lam Q_{\U{\lam_j}} \mid g + \sum_{i<j} \iota_i( Q_{\lam_i} - \pi)\Ra \\
&\quad +  \frac{ \iota_j}{2\log(\frac{\lam_{j+1}}{\lam_j}) } \frac{\lam_j'}{\lam_j} \La \chi_{L\sqrt{\lam_j \lam_{j+1}}} \ULam \Lam Q_{\U{\lam_j}} \mid g + \sum_{i<j} \iota_i( Q_{\lam_i} - \pi) \Ra \\
&\quad  -  \frac{ \iota_j}{2\log(\frac{\lam_{j+1}}{\lam_j}) }\La \chi_{L\sqrt{\lam_j \lam_{j+1}}} \Lam Q_{\U{\lam_j}} \mid \p_t g \Ra\\
&\quad +  \sum_{i<j} \frac{ \iota_i \iota_j}{2\log(\frac{\lam_{j+1}}{\lam_j}) } \lam_i' \La \chi_{L\sqrt{\lam_j \lam_{j+1}}} \Lam Q_{\U{\lam_j}} \mid \Lam Q_{\U{\lam_i}} \Ra
}
The second, third, and fourth terms on the right above contribute acceptable errors. Indeed, 
\EQ{
\Big|  \frac{ \iota_j}{2(\log(\frac{\lam_{j+1}}{\lam_j}))^2 }( \frac{\lam_{j}'}{\lam_j} - \frac{\lam_{j+1}'}{\lam_{j+1}})\La \chi_{L\sqrt{\lam_j \lam_{j+1}}} \Lam Q_{\U{\lam_j}} \mid g + \sum_{i<j} \iota_i( Q_{\lam_i} - \pi)\Ra \Big| &\lesssim \frac{ \max_{ i  \in \calA}  (\lambda_i(t) / \lambda_{i+1}(t))^\frac 12}{c_1(\log(\frac{\lam_{j+1}}{\lam_j}))^2} \\
\Big|  \frac{ \iota_j}{4\log(\frac{\lam_{j+1}}{\lam_j}) }\big( \frac{\lam_j'}{\lam_j} + \frac{\lam_{j+1}'}{\lam_{j+1}}\big)\La \Lam \chi_{L\sqrt{\lam_j \lam_{j+1}}} \Lam Q_{\U{\lam_j}} \mid g + \sum_{i<j} \iota_i( Q_{\lam_i} - \pi)\Ra  \Big| & \lesssim
\frac{ \max_{ i  \in \calA}  (\lambda_i(t) / \lambda_{i+1}(t))^\frac 12}{c_1 \log(\frac{\lam_{j+1}}{\lam_j})} \\
\Big|\frac{ \iota_j}{2\log(\frac{\lam_{j+1}}{\lam_j}) } \frac{\lam_j'}{\lam_j} \La \chi_{L\sqrt{\lam_j \lam_{j+1}}} \ULam \Lam Q_{\U{\lam_j}} \mid g + \sum_{i<j} \iota_i( Q_{\lam_i} - \pi)\Ra \Big| &\lesssim \frac{ \max_{ i  \in \calA}  (\lambda_i(t) / \lambda_{i+1}(t))}{\log(\frac{\lam_{j+1}}{\lam_j})}
}
with the gain in the last line arising from the fact that $\ULam \Lam Q \in L^1$; see~\eqref{eq:Lam0LamQ}. 
The leading order comes from the second to last term in~\eqref{eq:xi'-exp-k1}.  Using~\eqref{eq:g-eq} gives 
\EQ{
 -  \frac{ \iota_j}{2\log(\frac{\lam_{j+1}}{\lam_j}) }\La \chi_{L\sqrt{\lam_j \lam_{j+1}}} \Lam Q_{\U{\lam_j}} \mid \p_t g \Ra &= -    \frac{ \iota_j}{2\log(\frac{\lam_{j+1}}{\lam_j}) }\La \chi_{L\sqrt{\lam_j \lam_{j+1}}} \Lam Q_{\U{\lam_j}} \mid \dot g \Ra\\
&\quad -  \lam_j' \frac{ 1}{2\log(\frac{\lam_{j+1}}{\lam_j}) }\La \chi_{L\sqrt{\lam_j \lam_{j+1}}} \Lam Q_{\U{\lam_j}} \mid   \Lam Q_{\U{\lam_j}} \Ra\\
&\quad  -  \frac{ \iota_j}{2\log(\frac{\lam_{j+1}}{\lam_j}) }\La \chi_{L\sqrt{\lam_j \lam_{j+1}}} \Lam Q_{\U{\lam_j}} \mid \sum_{i \neq j} \iota_i \lam_{i}' \Lam Q_{\U{\lam_i}} \Ra\\
&\quad - \frac{ \iota_j}{2\log(\frac{\lam_{j+1}}{\lam_j}) }\La \chi_{L\sqrt{\lam_j \lam_{j+1}}} \Lam Q_{\U{\lam_j}}  \mid \phi( u, \nu)  \Ra. 
}
We estimate the contribution of each of the terms on the right above to~\eqref{eq:xi'-exp-k1}. The last term above vanishes due to the support properties of $\phi(u, \nu)$. Using~\eqref{eq:LamQL2}, ~\eqref{eq:lam'-k1} on the second term above, gives 
\EQ{
\Big|-\lam_j' \frac{ 1}{2\log(\frac{\lam_{j+1}}{\lam_j}) }\La \chi_{L\sqrt{\lam_j \lam_{j+1}}} \Lam Q_{\U{\lam_j}} \mid   \Lam Q_{\U{\lam_j}} \Ra + \lam_j'| \lesssim\frac{ \max_{ i  \in \calA}  (\lambda_i(t) / \lambda_{i+1}(t))^\frac 12}{\log(\frac{\lam_{j+1}}{\lam_j})}
}
which means this terms cancels the term $\lam'$ on the right-hand side of~\eqref{eq:xi'-exp-k1}  up to an acceptable error. Next we write, 
\EQ{
-  \frac{ \iota_j}{2\log(\frac{\lam_{j+1}}{\lam_j}) }\La \chi_{L\sqrt{\lam_j \lam_{j+1}}} \Lam Q_{\U{\lam_j}} \mid \sum_{i \neq j} \iota_i \lam_{i}' \Lam Q_{\U{\lam_i}} \Ra &=  -\sum_{i<j} \frac{ \iota_i \iota_j}{2\log(\frac{\lam_{j+1}}{\lam_j}) } \lam_i' \La \chi_{L\sqrt{\lam_j \lam_{j+1}}} \Lam Q_{\U{\lam_j}} \mid \Lam Q_{\U{\lam_i}} \Ra\\
&\quad  -  \sum_{i>j} \frac{ \iota_i \iota_j}{2\log(\frac{\lam_{j+1}}{\lam_j}) } \lam_i' \La \chi_{L\sqrt{\lam_j \lam_{j+1}}} \Lam Q_{\U{\lam_j}} \mid \Lam Q_{\U{\lam_i}} \Ra
}
The first term cancels the last term in~\eqref{eq:xi'-exp-k1}. For the second term we estimate, if $i>j$, 
\EQ{
|\La \chi_{L\sqrt{\lam_j \lam_{j+1}}} \Lam Q_{\U{\lam_j}} \mid \Lam Q_{\U{\lam_i}} \Ra|  \lesssim \lam_j/ \lam_{j+1}
}
and thus, using~\eqref{eq:lam'-k1} the second term in the previous equation contributes an acceptable error. Plugging all of these estimates back into~\eqref{eq:xi'-exp-k1} gives the estimate, 
\EQ{ \label{eq:xi'-lead} 
\Big| \xi_j' +  \frac{ \iota_j}{2\log(\frac{\lam_{j+1}}{\lam_j}) }\La \chi_{L\sqrt{\lam_j \lam_{j+1}}} \Lam Q_{\U{\lam_j}} \mid \dot g \Ra \Big| \lesssim \frac{ \max_{ i  \in \calA}  (\lambda_i(t) / \lambda_{i+1}(t))^\frac 12}{c_1\log(\frac{\lam_{j+1}}{\lam_j})} 
}
Using~\eqref{eq:g-bound-k1} and $\|\chi_{L\sqrt{\lam_j \lam_{j+1}}} \Lam Q_{\U{\lam_j}}  \|_{L^2} \lesssim (\log(\frac{\lam_{j+1}}{\lam_j}))^{\frac{1}{2}}$,   we deduce the estimate, 
\EQ{
\Big| \frac{ \iota_j}{2\log(\frac{\lam_{j+1}}{\lam_j}) }\La \chi_{L\sqrt{\lam_j \lam_{j+1}}} \Lam Q_{\U{\lam_j}} \mid \dot g \Ra\Big| \lesssim \frac{ \max_{ i  \in \calA}  (\lambda_i(t) / \lambda_{i+1}(t))^\frac 12}{(\log(\frac{\lam_{j+1}}{\lam_j}))^{\frac{1}{2}}} 
}
which completes the proof of~\eqref{eq:xi_j'-k1}. 

Next we compare $\be_j$  and $2\xi_j' \log( \lam_{j+1}/ \lam_j)$. Using~\eqref{eq:beta-def-k1} we have, 
\EQ{
\Big|  \ang{ \uln A( \lam_j(t)) g(t) \mid \dot g(t)} \Big| \lesssim \| \bs g \|_{\E}^2 \lesssim \max_{ i  \in \calA}  (\lambda_i / \lambda_{i+1}), 
}
We also note the estimate
\EQ{
\Big| \La ( \chi_{L\sqrt{\lam_j \lam_{j+1}}} - \chi_{L\sqrt{\xi_j \lam_{j+1}}}) \Lam Q_{\U{\lam_j}} \mid \dot g \Ra\Big|  \ll \max_{ i  \in \calA}  (\lambda_i / \lambda_{i+1})^\frac 12. 
}
which is a consequence of~\eqref{eq:xi_j-lambda_j-k1}. Using~\eqref{eq:xi'-lead} the estimate~\eqref{eq:xi_j'-beta_j-k1} follows. 

Finally, the proof of the estimate~\eqref{eq:beta_j'-k1} is nearly identical to the argument used to prove~\eqref{eq:beta_j'}, differing only in a  few places where the cut-off $\chi_{L \sqrt{\xi_j \lam_{j+1}}}$ is  involved. Arguing as in the proof of ~\eqref{eq:beta_j'} we arrive at the formula, 
\EQ{ \label{eq:beta'-exp-k1} 
\beta_j' &=  -  \frac{\iota_j }{\lam_j}  \ang{ \Lam Q_{\lam_j} \mid f_{\bfi}( m_n, \iota, \vec \lam) } +  \ang{ \uln A( \lam_j) g \mid \LL_0 g }   + \ang{ (A(\lam_j) - \uln A( \lam_j)) g \mid \ti  f_{\bfq}(m_n, \vec \iota,  \vec \lam, g)} \\
&\quad + \ang{ \chi_{L \sqrt{\xi_j \lam_{j+1}}}\Lam Q_{\U{\lam_j}} \mid ( \LL_{\calQ} - \LL_{\lam_j}) g} + \iota_j \frac{\lam_j' }{\lam_j}\ang{ \big( \frac{1}{\lam_j} \ULam - \U{A}( \lam_j) \big) \Lam Q_{\lam_j} \mid \dot g} \\
&\quad    - \ang{ {A}(\lam_j) \sum_{i =1}^K \iota_i Q_{\lam_i} \mid f_{\bfq}(m_n, \vec \iota,  \vec\lam, g)}   - \ang{ A( \lam_j) g \mid \ti  f_{\bfq}(m_n, \vec \iota,  \vec \lam, g)} \\
&\quad    + \iota_j \ang{ ({A}( \lam_j) -  \frac{1}{\lam_j}\chi_{L \sqrt{\xi_j \lam_{j+1}}}\Lam) Q_{\lam_j} \mid f_{\bfq}(m_n, \vec \iota,  \vec\lam, g)}  -    \frac{\lam_j'}{\lam_j} \ang{ \lam_j \p_{\lam_j} \uln A( \lam_j) g \mid \dot g} \\
&\quad + \sum_{ i \neq j} \iota_i \ang{ {A}(\lam_j) Q_{\lam_i} \mid  f_{\bfq}(m_n, \vec \iota,  \vec\lam, g)}  - \sum_{i \neq j} \iota_i \lam_{i}'  \ang{ \uln A( \lam_j) \Lam Q_{\U{\lam_i}} \mid \dot g}   - \ang{ \uln A( \lam_j) g \mid f_{\bfi}( m_n, \iota, \vec \lam) } \\
&\quad  - \iota_j \ang{ \chi_{L \sqrt{\xi_j \lam_{j+1}}}\Lam Q_{\U{\lam_j}} \mid  \dot \phi( u, \nu)} -  \ang{ \uln A( \lam_j) \phi( u, \nu) \mid \dot g} -  \ang{ \uln A( \lam_j) g \mid \dot  \phi( u, \nu)} \\
&\quad + \frac{\iota_j }{\lam_j}  \ang{ (1- \chi_{L \sqrt{\xi_j \lam_{j+1}}})\Lam Q_{\lam_j} \mid f_{\bfi}( m_n, \iota, \vec \lam) } -  \iota_j \frac{\lam_j' }{\lam_j}\ang{  (1- \chi_{L \sqrt{\xi_j \lam_{j+1}}})\ULam \Lam Q_{\U{\lam_j}} \mid \dot g}\\
&\quad + \ang{ \LL_{\lam_j} (\chi_{L \sqrt{\xi_j \lam_{j+1}}}\Lam Q_{\U{\lam_j}}) \mid  g}
 + \frac{ \iota_j}{2}\big( \frac{\xi_j'}{\xi_j} + \frac{\lam_{j+1}'}{\lam_{j+1}}\big)\La \Lam \chi_{L\sqrt{\xi_j \lam_{j+1}}} \Lam Q_{\U{\lam_j}} \mid \dot g \Ra 
}
All but the last four terms above are treated exactly as in the proof of~\eqref{eq:beta_j'}. For the fourth-to-last term a direct computation using the estimate~\eqref{eq:fi-point} gives, 
\EQ{
\Big|\frac{\iota_j }{\lam_j}  \ang{ (1- \chi_{L \sqrt{\xi_j \lam_{j+1}}})\Lam Q_{\lam_j} \mid f_{\bfi}( m_n, \iota, \vec \lam) }  \Big| \ll \frac{1}{\lam_j} \Big( \frac{\lam_j}{\lam_{j+1}} + \frac{\lam_{j-1}}{\lam_j} \Big). 
}
For the third-to-last term, we use that $\ULam \Lam Q \in L^2$, ~\eqref{eq:lam'-k1},  and~\eqref{eq:g-bound-k1} to deduce that, 
\EQ{
\Big| \iota_j \frac{\lam_j' }{\lam_j}\ang{  (1- \chi_{L \sqrt{\xi_j \lam_{j+1}}})\ULam \Lam Q_{\U{\lam_j}} \mid \dot g} \Big| \ll \frac{1}{\lam_j} \max_{ i  \in \calA}  (\lambda_i / \lambda_{i+1}),
}
The size of the constant $L>0$ becomes relevant only in the second-to-last term. Indeed, since $\LL \Lam Q = 0$, we have, 
\EQ{
\LL_{\lam_j} (\chi_{L \sqrt{\xi_j \lam_{j+1}}}\Lam Q_{\U{\lam_j}}) = \frac{1}{L^2 \xi_j \lam_{j+1}} (\De \chi)_{L \sqrt{\xi_j \lam_{j+1}}}\Lam Q_{\U{\lam_j}} + 2 \frac{1}{L \sqrt{ \xi_j \lam_{j+1}}} \chi'_{L \sqrt{\xi_j \lam_{j+1}}} \frac{1}{r} \Lam^2 Q_{ \U{\lam_j}}. 
}
And therefore, using~\eqref{eq:g-bound-k1} and~\eqref{eq:xi_j-lambda_j-k1} we obtain the estimate, 
\EQ{
\Big| \ang{ \LL_{\lam_j} (\chi_{L \sqrt{\xi_j \lam_{j+1}}}\Lam Q_{\U{\lam_j}}) \mid  g} \Big| \lesssim \frac{1}{L} \frac{1}{\lam_{j}} \max_{ i  \in \calA}  (\lambda_i / \lambda_{i+1})
}
for a uniform constant, independent of $L$. Taking $L>1$ large enough relative to $c_0$ makes this an acceptable error. Finally, for the last term we use the improved estimate~\eqref{eq:xi_j'-k1} for $\xi_j'$ and~\eqref{eq:xi_j-lambda_j-k1} to obtain, 
\EQ{
\abs{ \frac{\xi_j'}{\xi_j} + \frac{\lam_{j+1}'}{\lam_{j+1}}}  \lesssim \frac{1}{\lam_j} \Big( \abs{ \xi_j'} + \frac{\lam_j}{\lam_{j+1}} \abs{ \lam_{j+1}'} \Big) \ll \frac{1}{\lam_j} \max_{ i  \in \calA}  (\lambda_i / \lambda_{i+1})^{\frac{1}{2}}, 
}
and hence, 
\EQ{
\Big|  \frac{ \iota_j}{2}\big( \frac{\xi_j'}{\xi_j} + \frac{\lam_{j+1}'}{\lam_{j+1}}\big)\La \Lam \chi_{L\sqrt{\xi_j \lam_{j+1}}} \Lam Q_{\U{\lam_j}} \mid \dot g \Ra  \Big| \ll \frac{1}{\lam_j} \max_{ i  \in \calA}  (\lambda_i / \lambda_{i+1}). 
}
This completes the proof. 
\end{proof}
We note that Lemma~\ref{lem:virial-error} and its proof remain valid for $k = 1$.

\subsection{Demolition of the multi-bubble}
We define the weighted interaction energy in the same way as for $k \geq 2$,
see Definition~\ref{def:w-int}.
\begin{lem}
If $c_0$ in Proposition~\ref{prop:modul-k1} is taken sufficiently small,
then there exists a constant $C_1 > 0$ such that
\begin{equation}
\label{eq:U-bound-by-d-k1}
\bfd(t) < \eta_0\quad\text{implies}\quad C_1^{-1} \bfd(t)^2 \leq U(t) \leq C_1\bfd(t)^2.
\end{equation}
\end{lem}
\begin{rem}
In the case $k \geq 2$ the corresponding estimate \eqref{eq:U-bound-by-d}
follows immediately from \eqref{eq:xi_j-lambda_j}.
For $k = 1$, since the bound \eqref{eq:xi_j-lambda_j-k1}
does not imply that $\xi_j(t) \simeq \lambda_j(t)$ for all $j \in \cA$.
\end{rem}
\begin{proof}
Let $t$ be such that $\bfd(t) < \eta_0$, and let $j_0 \in \cA$ be such that
\begin{equation}
\lambda_{j_0}(t) / \lambda_{j_0+1}(t) = \max_{i \in \cA}\lambda_i(t) / \lambda_{i+1}(t).
\end{equation}
Then we deduce from \eqref{eq:xi_j-lambda_j-k1} that
\begin{equation}
\xi_{j_0}(t) / \lambda_{j_0+1}(t) \gtrsim \lambda_{j_0}(t) / \lambda_{j_0+1}(t) \gtrsim \bfd(t)^2,
\end{equation}
which yields the required lower bound on $U(t)$.

The upper bound follows directly from \eqref{eq:xi_j-lambda_j-k1}
and the fact that $\lambda_j(t)/\lambda_{j+1}(t) \lesssim \bfd(t)^2$ for all $j \in \cA$.
\end{proof}

The analog of Lemma~\ref{lem:ejection} for $k = 1$ is formulated as follows.
\begin{lem}
\label{lem:ejection-k1}
If $\eta_0$ is small enough,
then there exists $C_0 \geq 0$ depending only on $k$ and $N$ such that the following is true.
If $t_0$ is a local minimum from the right of $U$ such that $U(t_0) < +\infty$
and $t_* \geq t_0$ is such that $U(t) < \infty$ for all $t \in [t_0, t_*]$, then
\begin{gather}
\label{eq:lambdaK-bd-k1}
\frac 34 \lambda_K(t_0) \leq \lambda_K(t_*) \leq\frac 43\lambda_K(t_0), \\
\label{eq:int-d-bd-k1}
\int_{t_0}^{t_*} \bfd(t)\ud t\leq C_0\bfd(t_*)^2\sqrt{{-}\log\bfd(t_*)}\lambda_K(t_0).
\end{gather}
An analogous statement is true if $t_*$ is a local minimum from the left.
\end{lem}
\begin{proof}
Steps 1 and 2 are similar as for $k \geq 2$.
Step 3 differs significantly, so let us indicate the necessary changes.
First, the modulation estimates \eqref{eq:xi_j'-k1}, \eqref{eq:xi_j'-beta_j-k1} and \eqref{eq:beta_j'-k1}
only hold under the assumption \eqref{eq:max-ratio-ass}.
However, note that in Step 3 this last assumption is satisfied on the time interval $(t_{l_0}, T_0)$, on which the modulation estimates are used, see \eqref{eq:xi-really-bigger}.

Instead of \eqref{eq:xijl-init}, we claim that
\begin{equation}
\label{eq:xijl-init-k1}
\xi_{j_l}'(t_l)\sqrt{{-}\log\wt\xi_{j_l}(t_l)} \simeq \xi_{j_l}'(t_l)\sqrt{{-}\log(\xi_{j_l}(t_l)/\lambda_{j_l+1}(t_l))} \geq -c_0 \max_{i \in \cA}\sqrt{\xi_i(t_l) / \lambda_{i+1}(t_l)},
\end{equation}
where $c_0$ can be made arbitrarily small upon taking $\eta_0$ small.
Indeed, recalling that $\wt\xi_j(t) = 2^{-j}\xi_j(t)/\lambda_{j+1}(t)$,
\eqref{eq:xijl-init-0} yields
\begin{equation}
\xi_{j_l}'(t_l) \gtrsim -\wt\xi_{j_l}(t_l)|\lambda_{j_l+1}'(t_l)|- \frac{\lambda_{j_l+1}(t_l)}{\lambda_{j_{l-1}+1}(t_l)}\big(|\xi_{j_l-1}'(t_l)| + \wt\xi_{j_{l-1}}(t_l)|\lambda_{j_{l-1}+1}'(t_l)|\big).
\end{equation}
Since $\wt \xi_{j_l}(t_l)$ is small when $\eta_0$ is small, \eqref{eq:lam'-k1} yields
\begin{equation}
\wt\xi_{j_l}(t_l)\sqrt{{-}\log\wt\xi_{j_l}(t_l)}|\lambda_{j_l+1}'(t_l)| \leq c_0\max_{i \in \cA}\sqrt{\xi_i(t_l) / \lambda_{i+1}(t_l)}.
\end{equation}
Since $\wt \xi_{j_l}(t_l) = \wt\xi_{j_{l-1}}(t_l)$, \eqref{eq:xi_j'-k1} yields
\begin{equation}
\sqrt{{-}\log\wt\xi_{j_l}(t_l)}|\xi_{j_{l-1}}'(t_l)| \lesssim \max_{i \in \cA}\sqrt{\xi_i(t_l) / \lambda_{i+1}(t_l)}.
\end{equation}
For the same reason, and using again \eqref{eq:lam'-k1},
\begin{equation}
\sqrt{{-}\log\wt\xi_{j_l}(t_l)}\wt\xi_{j_{l-1}}(t_l)|\lambda_{j_{l-1}+1}'(t_l)|
\lesssim \max_{i \in \cA}\sqrt{\xi_i(t_l) / \lambda_{i+1}(t_l)}.
\end{equation}
Since $j_l < j_{l-1}$, $\lambda_{j_l+1}(t_l)/\lambda_{j_{l-1}+1}(t_l)$ is small when $\eta_0$
is small, so we get \eqref{eq:xijl-init-k1}.

In \eqref{eq:collaps-induction}, \eqref{eq:collaps-boot-ass} and
\eqref{eq:collaps-boot}, we replace $\bfd(t)^{k/2}$ by $\bfd(t)^2\sqrt{{-}\bfd(t)}$.
Next, we introduce the auxiliary function $\Phi(x) := \sqrt{-x\log x}$ for $0 < x < 1$.
Note that
\begin{equation}
\begin{aligned}
\sqrt{x} &\sim \Phi(x) / \sqrt{{-}\log \Phi(x)}, \\
\Phi'(x) &= \frac{\sqrt{-\log x}}{2\sqrt x} + O(({-}x\log x)^{-1/2}) > 0.
\end{aligned}
\end{equation}
With $c_2 > 0$ to be determined, consider the auxiliary function
\begin{equation}
\phi(t) := \beta_{j_l}(t) + c_2\Phi\big(\xi_{j_l}(t)/\lambda_{j_l+1}(t_l)\big).
\end{equation}
The Chain Rule gives
\begin{equation}
\phi'(t) = \beta_{j_l}'(t) + c_2 \frac{\xi_{j_l}'(t)}{\lambda_{j_l+1}(t_l)}\Phi'\Big(\frac{\xi_{j_l}(t)}{\lambda_{j_l+1}(t_l)}\Big).
\end{equation}
By \eqref{eq:xi_j'-k1} and \eqref{eq:lambdaj-const}, we have $|\xi_{j_l}'(t)| \leq c_3 (\xi_{j_l}(t) / \lambda_{j_l +1}(t_l))^\frac 12\log({-}\xi_{j_l}(t) / \lambda_{j_l +1}(t_l))^{-1/2}$, with $c_3$ depending only on $k$ and $N$,
hence \eqref{eq:beta_j'-k1} implies
\begin{equation}
\label{eq:f'-lbound-k1}
\phi'(t) \geq \frac{c_3}{\lambda_{j_l+1}(t_l)},
\end{equation}
with $c_2, c_3$ depending only on $k$ and $N$.
If we consider $\wt \phi(t) := \beta_{j_l}(t) + \frac{c_2}{2}\Phi\big(\xi_{j_l}(t)/\lambda_{j_l+1}(t_l)\big)$
instead of $\phi$, then the computation above shows that $\wt \phi$ is increasing.
From \eqref{eq:xijl-init-k1}, we have $\wt \phi(t_l) \geq 0$, so $\wt \phi(t) \geq 0$ for all $t \in (t_l, t_{l+1})$,
implying
\begin{equation}
\label{eq:d-ubound-k1}
\bfd(t) \lesssim \sqrt{\xi_{j_l}(t)/\lambda_{j_l+1}(t_l)} \lesssim
\phi(t)/\sqrt{{-}\log \phi(t)}.
\end{equation}

The bound \eqref{eq:f'-lbound-k1} yields
\begin{equation}
\label{eq:fund-thm-for-f-k1}
\begin{aligned}
\big(\lambda_{j_l+1}(t_l)\phi(t)^{2}/\sqrt{{-}\log \phi(t)}\big)' \gtrsim \phi(t)/\sqrt{{-}\log \phi(t)}.
\end{aligned}
\end{equation}
We observe that $|\phi(t)| \lesssim \Phi(\bfd(t)^2)$, hence $\phi(t)^2/\sqrt{-\log \phi(t)} \lesssim \bfd(t)^2\sqrt{{-}\log\bfd(t)}$ and
\begin{equation}
\int_{t_l}^{t_{l+1}}\phi(t)/\sqrt{{-}\log \phi(t)}\ud t \lesssim
\lambda_{j_l+1}(t_l)\phi(t_{l+1})^{2}/\sqrt{-\log \phi(t_{l+1})} \lesssim \bfd(t_{l+1})^{2}\sqrt{{-}\log\bfd(t_{l+1})}\lambda_{j_l+1}(t_l).
\end{equation}
Thus, \eqref{eq:d-ubound-k1} yields \eqref{eq:collaps-boot}
(with $\bfd(t)^2\sqrt{{-}\log \bfd(t)}$ instead of
$\bfd(t)^{2/k}$) if $C_0$ is sufficiently large
(but depending on $k$ and $N$ only).

We now prove \eqref{eq:increase-boot}.
By \eqref{eq:xijl-init-k1} and \eqref{eq:xi_j'-beta_j-k1},
we have $\beta_{j_l}(t_l) \geq -c_0 \Phi(\xi_{j_l}(t_l)/\lambda_{j_l+1}(t_l))$,
where $c_0$ can be made as small as needed, and
\begin{equation}
\beta_{j_l}'(t) \geq c_1/\lambda_{j_l+1}(t_l).
\end{equation}
We deduce that $\xi_{j_l}'(t) \geq 0$ provided
\begin{equation}
t - t_l \geq \frac{2c_0}{c_1}\lambda_{j_l+1}(t_l)\Phi(\xi_{j_l}(t_l)/\lambda_{j_l+1}(t_l)) = \frac{2c_0}{c_1}\sqrt{\xi_{j_l}(t_l)\lambda_{j_l+1}(t_l)}\sqrt{{-}\log(\xi_{j_l}(t_l)/\lambda_{j_l+1}(t_l))}.
\end{equation}
But, if the opposite inequality is satisfied, the bound
$$|\xi_{j_l}'(t)| \lesssim \sqrt{\xi_{j_l}(t_l)/\lambda_{j_l+1}(t_l)}/\sqrt{{-}\log(\xi_{j_l}(t_l)/\lambda_{j_l+1}(t_l))}$$
yields \eqref{eq:increase-boot}, if $c_0$ is small enough.

The proof is then finished as for $k \geq 2$.
\end{proof} 
The remaining arguments of Section 5 apply without major changes.
In \eqref{eq:error-on-eLeR}, one should replace $\theta_n^{2/k}$
by $\theta_n^2\sqrt{{-}\log\theta_n}$.

\bibliographystyle{plain}
\bibliography{researchbib}

\bigskip
\centerline{\scshape Jacek Jendrej}
\smallskip
{\footnotesize
 \centerline{CNRS and LAGA, Universit\'e  Sorbonne Paris Nord}
\centerline{99 av Jean-Baptiste Cl\'ement, 93430 Villetaneuse, France}
\centerline{\email{jendrej@math.univ-paris13.fr}}
} 
\medskip 
\centerline{\scshape Andrew Lawrie}
\smallskip
{\footnotesize
 \centerline{Department of Mathematics, Massachusetts Institute of Technology}
\centerline{77 Massachusetts Ave, 2-267, Cambridge, MA 02139, U.S.A.}
\centerline{\email{alawrie@mit.edu}}
}

\end{document}